\documentclass[11pt, letterpaper,english]{amsart}
\usepackage[left=1in,right=1in,bottom=1in,top=1in]{geometry}
\usepackage{amsmath,amsthm,amssymb}
\usepackage{mathrsfs}
\usepackage[all,cmtip]{xy}
\usepackage{cite}
\usepackage{hyperref}
\usepackage{enumerate}
\usepackage{graphicx}
\usepackage[usenames,dvipsnames]{color}
\usepackage{tikz-cd}

\theoremstyle{plain}
\newtheorem{lem}{Lemma}
\newtheorem{thm}[lem]{Theorem}
\newtheorem{prop}[lem]{Proposition}
\newtheorem{cor}[lem]{Corollary}

\newtheorem{fact}[lem]{Fact}

\theoremstyle{definition}
\newtheorem{defn}[lem]{Definition}
\newtheorem{example}[lem]{Example}
\newtheorem{remark}[lem]{Remark}

\numberwithin{equation}{subsection}
\numberwithin{lem}{subsection}

\newtheorem{thmA}{Theorem}

\newcommand{\mathfont}{\mathbf}

\newcommand{\Z}{\mathfont Z}
\newcommand{\ZZ}{\mathfont Z}

\newcommand{\Q}{\mathfont Q}

\newcommand{\F}{\mathfont F}
\newcommand{\FF}{\mathfont F}

\newcommand{\fM}{\mathfrak{M}}

\newcommand{\fP}{\mathfrak{P}}

\newcommand{\fS}{\mathfrak{S}}

\newcommand{\fg}{\mathfrak{g}}
\newcommand{\fm}{\mathfrak{m}}

\newcommand{\ft}{\mathfrak{t}}

\newcommand{\bA}{\mathfont{A}}

\newcommand{\bF}{\mathfont{F}}

\newcommand{\bK}{\mathfont{K}}

\newcommand{\cA}{\mathcal{A}}

\newcommand{\cC}{\mathcal{C}}

\newcommand{\cE}{\mathcal{E}}
\newcommand{\cF}{\mathcal{F}}

\newcommand{\cJ}{\mathcal{J}}

\newcommand{\cO}{\mathcal{O}}
\newcommand{\cP}{\mathcal{P}}

\DeclareFontFamily{OT1}{rsfs}{}
\DeclareFontShape{OT1}{rsfs}{n}{it}{<-> rsfs10}{}
\DeclareMathAlphabet{\mathscr}{OT1}{rsfs}{n}{it}

\newcommand{\into}{\hookrightarrow}

\newcommand{\ol}[1]{\overline{#1}}

\DeclareMathOperator{\Hom}{Hom}

\newcommand \tensor[1] {\otimes_{#1}}
\DeclareMathOperator{\Aut}{Aut}
\DeclareMathOperator{\Gal}{Gal}

\DeclareMathOperator{\End}{End}

\DeclareMathOperator{\GRep}{GRep}

\DeclareMathOperator{\Res}{Res}

\DeclareMathOperator{\Sym}{Sym}
\DeclareMathOperator{\der}{der}

\DeclareMathOperator{\Fil}{Fil}

\DeclareMathOperator{\Gr}{Gr}

\DeclareMathOperator{\Mat}{Mat}

\newcommand{\Zp}{\mathfont{Z}_p}

\newcommand{\Qp}{\mathfont{Q}_{p}}

\newcommand{\rhobar}{{\overline{\rho}}}

\DeclareMathOperator{\Spec}{Spec}
\DeclareMathOperator{\Proj}{Proj}
\DeclareMathOperator{\Spf}{Spf}

\DeclareMathOperator{\Lie}{Lie}

\DeclareMathOperator{\GL}{GL}

\DeclareMathOperator{\PGL}{PGL}

\DeclareMathOperator{\SO}{SO}

\DeclareMathOperator{\Sp}{Sp}
 
\DeclareMathOperator{\GSp}{GSp}

\newcommand{\Ga}{\mathfont{G}_a}
\newcommand{\Gm}{\mathfont{G}_m}
\newcommand{\ad}{\mathrm{ad}}
\DeclareMathOperator{\Ad}{Ad}

\DeclareMathOperator{\LG}{LG}
\newcommand{\LpG}{\textrm{L}^+G}

\renewcommand{\inf}{{\operatorname{inf}}}
\newcommand{\Ainfbasic}{\bA_{\inf}}
\newcommand{\Ainf}[1]{\bA_{\inf,#1}}
\newcommand{\Kbar}{\overline{K}}
\newcommand{\Gammahat}{\widehat{\Gamma}}
\newcommand{\fPbar}{\overline{\fP}}

\DeclareMathOperator{\GMod}{GMod}
\DeclareMathOperator{\dom}{dom}
\newcommand{\AdG}{\Ad_G}
\newcommand{\g}{\fg}

\newcommand{\pseries}[1]{[\![#1 ] \!]}
\newcommand{\lseries}[1]{(\!(#1 ) \!)}
\newcommand{\fRep}{{^f} \hspace{-2pt} \operatorname{Rep}} 

\newcommand{\betabar}{\overline{\beta}}
\newcommand{\tT}{\widetilde{T}} 
\newcommand{\rig}{{\operatorname{rig}}}
\newcommand{\triv}{{\operatorname{triv}}}
\newcommand{\id}{\operatorname{id}}
\newcommand{\dR}{\operatorname{dR}}
\newcommand{\pflat}{\operatorname{flat}}
\newcommand{\high}{\operatorname{high}}

\title{G-Valued Crystalline Deformation Rings in the Fontaine-Laffaille Range}

\author{Jeremy Booher}
\address{School of Mathematics and Statistics, University of Canterbury, Private Bag 4800, Christchurch 8140, New Zealand}
\email{jeremy.booher@canterbury.ac.nz}

\author{Brandon Levin}
\address{Department of Mathematics, The University of Arizona, 617 N. Santa Rita Ave., Tucson, AZ 85721 USA}
\email{bwlevin@math.arizona.edu}

\begin{document} 

\begin{abstract}
Let $G$ be a split reductive group over the ring of integers in a $p$-adic field with residue field $\mathbf{F}$.   
Fix a representation $\overline{\rho}$ of the absolute Galois group of an unramified extension of $\mathbf{Q}_p$, valued in $G(\mathbf{F})$.  We study the crystalline deformation ring for $\overline{\rho}$ with a fixed $p$-adic Hodge type that satisfies an analog of the Fontaine-Laffaille condition for $G$-valued representations.  In particular, we give a root theoretic condition on the $p$-adic Hodge type which ensures that the crystalline deformation ring is formally smooth.  Our result improves on all known results for classical groups not of type A and provides the first such results for exceptional groups.   
\end{abstract}

\maketitle

\section{Introduction}

\subsection{Crystalline Deformation Rings}

Let $p$ be a prime, and $\Lambda$ be the ring of integers in a $p$-adic field $L$ with residue field $\bF$.  We fix a split reductive group $G$ over $\Lambda$.  Then let $K$ be a $p$-adic field unramified over $\Qp$ with residue field $k$ and ring of integers $W(k)$, and denote the absolute Galois group of $K$ by $\Gamma_K$.  

For a fixed continuous homomorphism $\rhobar: \Gamma_K \to G(\bF)$, there has been considerable interest in studying lifts of $\rhobar$ with ``nice'' properties, in particular lifts closely connected to $p$-adic Hodge theory.  This began with Ramakrishna's results on flat deformations \cite{ramakrishna93}, which played an important role in the Taylor-Wiles proof of modularity of semistable elliptic curves over $\Q$.   Most automorphy lifting theorems for $\GL_n$ use either an ordinary or Fontaine--Laffaille condition at $p$ to ensure that the local deformation ring is nice. 
 Fontaine--Laffaille theory \cite{fl83} lets one study special cases of the crystalline deformation rings for $G = \GL_n$ constructed by Kisin \cite{kisin08} when $p$ is unramified in $K$ and the Hodge--Tate weights lie in a small interval relative to $p$.   In this paper, we address the natural question of finding an analogue of the Fontaine--Laffaille condition for $G$-valued crystalline deformations.   More precisely, we give a root theoretic condition on the Hodge--Tate cocharacter which ensures that the crystalline deformation ring is formally smooth.   Up to technical conditions on the isogeny class of the group $G$, we believe this result is close to optimal when the cocharacter is regular.  As discussed below, our result improves on all known results for classical groups not of type A and provides the first such results for exceptional groups.   

We begin by stating our result more precisely.
If $B$ is an $L$-algebra, a continuous representation of $\Gamma_K$ valued in $G(B)$ is crystalline if the composition with any representation of $G$ is crystalline. 
Assuming that $\Lambda$ contains a copy of $W(k)$, to such a representation we may associate a $p$-adic Hodge type, which is a collection of geometric conjugacy classes of cocharacters of $G$ indexed by the set $\cJ$ of embeddings of $W(k)$ into $\Lambda$.
This generalizes the notation of (labeled) Hodge-Tate weights of a representation of $\Gamma_K$ valued in $\GL_n(B)$.  

Let $\mu= (\mu_{\sigma}) _{\sigma \in \cJ}$ be a collection of dominant cocharacters of $G$.  Our primary goal is to study the framed crystalline deformation ring with $p$-adic Hodge type $\mu$, whose $L$-points are crystalline representations with $p$-adic Hodge type given by $\mu$.  We denote this $\Lambda$-algebra by $R_{\rhobar}^{\mu,\square}$. 

\begin{defn} \label{defn:FL}
We say that $\mu$ is Fontaine-Laffaille, or lies in the Fontaine-Laffaille range, provided that $\langle \mu_\sigma,\alpha \rangle < p-1$ for every root $\alpha \in \Phi_{G}$ and every embedding $\sigma$ of $W(k)$ into $L$.  We say $\mu$ is strongly Fontaine-Laffaille provided that $\langle \mu_\sigma,\alpha \rangle < \frac{p-1}{2}$ for every root $\alpha \in \Phi_{G}$ and every embedding.
\end{defn}

There is a natural way to associate potential $p$-adic Hodge types $\mu$ to $\rhobar$.  

\begin{thmA} \label{thm:mainintro}
Suppose $p$ is unramified in $K$ and that $p \nmid \# \pi_1(G^{\ad})$, where $G^{\ad}$ is the adjoint group of $G$.
Fix a Galois representation $\rhobar : \Gamma_K \to G(\bF)$ and a potential $p$-adic Hodge type $\mu$ for the representation.\footnote{More precisely, $\rhobar|_{\Gamma_\infty} \simeq \widetilde{T}_{G,\F}(\fPbar)$ for some $G$-Kisin module $\fPbar$ of shape $\mu$; see Definition~\ref{defn:shapemu} and \S\ref{ss:deformationproblems}.}  
 If $\Spf R_{\rhobar}^{\mu,\square}$ is non-empty and either
\begin{enumerate}[(i)]
\item $\mu$ is Fontaine-Laffaille and the derived group $G^{\der}$ is simply connected, or
\item  $\mu$ is strongly Fontaine-Laffaille,
\end{enumerate}
then $\Spf R_{\rhobar}^{\mu,\square}$ is formally smooth over $\Spf(\Lambda)$.
\end{thmA}

\begin{remark}
If $\Spf R_{\rhobar}^{\mu,\square}$ is non-empty, its dimension is $\dim G_\F +  \sum_{\sigma \in \cJ} \dim P_{\mu_{\sigma},\F} \backslash G_\F$ where $P_{\mu_{\sigma},\F}$ is the parabolic associated to $\mu_{\sigma}$; this follows from studying the generic fiber \cite[Theorem 5.1.5]{balaji}.  
The assumption that $\Spf R_{\rhobar}^{\mu,\square}$ is non-empty is harmless in many applications, for example to modularity lifting theorems.  Furthermore, when $\mu$ is regular, one does not need the assumption that $\mu$ is a potential $p$-adic Hodge type; the assumption that $\Spf R_{\rhobar}^{\mu,\square}$ is non-empty suffices (see Remark~\ref{rmk:cocharacterhyp}).
\end{remark}

Theorem~\ref{thm:mainintro} generalizes previous results about crystalline deformation rings obtained using Fontaine-Laffaille theory for the general linear, symplectic, and orthogonal groups.  Beyond these cases, for example for spin and exceptional groups, it provides completely new information.

\begin{example}
When $G = \GL_n$, the condition that $\mu$ is Fontaine-Laffaille is equivalent to the Hodge-Tate weights of the $p$-adic Hodge type (using the standard representation) lying in an interval of length less than $p-1$ for each embedding of $L$ into $K$.  In this situation, Clozel, Harris, and Taylor applied Fontaine-Laffaille theory to show formal smoothness of the Fontaine-Laffaille deformation ring \cite[\S2.4]{cht08}.  
This deformation ring is none other than the crystalline deformation ring we consider.  Our proof uses the theory of Kisin modules and so is independent of Fontaine-Laffaille theory.
Theorem~\ref{thm:mainintro} has an extra hypothesis, that $p \nmid n = \# \pi_1(\PGL_n)$.  This hypothesis is needed to apply Proposition~\ref{prop:ts}, where it is used to reduce to the adjoint case.  For the special case of $G=\GL_n$, it is not hard to do a direct analysis and remove this hypothesis (see Remark~\ref{remark:specialGLn}).
\end{example}

\begin{example}
When $p \neq 2$, we can apply Theorem~\ref{thm:mainintro} to symplectic and orthogonal groups.  Now $\Sp_{2n}$ is simply connected while $\SO_n$ is not, and the highest roots (using the standard descriptions of the root system) are $2e_1$ and $e_1+e_2$ respectively.  Using the standard representation, we can relate $p$-adic Hodge types to Hodge-Tate weights.  We see that our result applies to $\Sp_{2n}$ for Fontaine-Laffaille $\mu$, when the  Hodge-Tate weights of the $p$-adic Hodge type lie in the open interval $(-\frac{p-1}{2}, \frac{p-1}{2})$ for each embedding of $L$ into $K$.  
There is an obvious extension to $\GSp_{2n}$ with Hodge-Tate weights lying in an interval of length less than $p-1$. 

Similarly, our results apply to $\SO_n$ for strongly Fontaine-Laffaille $\mu$, when the sum of the two largest Hodge-Tate weights is less than $\frac{p-1}{4}$ for each embedding of $L$ into $K$.  This is the case, for example, if the Hodge-Tate weights lie in  $(-\frac{p-1}{4}, \frac{p-1}{4})$.
  
This generalizes results of Patrikis and the first author (see \cite{patrikis06} and \cite[Theorem 5.2]{booher19}) which use Fontaine--Laffaille modules with pairings to show formal smoothness under the additional assumption that the Hodge-Tate weights with respect to the standard embedding lie in an interval of length less than $(p-1)/2$ (for each embedding of $L$ into $K$) and that the $p$-adic Hodge type is regular (i.e. for each embedding the Hodge-Tate weights are distinct).  In the symplectic case, our method gives roughly double the range of Hodge-Tate weights, while we obtain a similar range for special orthogonal groups but with some added flexibility involving the second-largest Hodge-Tate weight.   
\end{example}

\begin{example}
As spin groups are simply connected, Theorem~\ref{thm:mainintro} applies to Fontaine-Laffaille $p$-adic Hodge types.  Under the quotient map $\operatorname{Spin}_n \to \SO_n$ and embedding $\SO_n \into \GL_n$, this implies the sum of the two largest Hodge-Tate weights would be less than $p-1$.

If we attempted to deduce things about the crystalline deformation ring for $\operatorname{Spin}_n$-valued representations via the quotient map $\operatorname{Spin}_n \to \SO_n$ and applying 
 Fontaine-Laffaille theory with pairings to study $\SO_n$-valued representations, we would need a much stronger condition, that the largest Hodge-Tate weight is less than $\frac{p-1}{4}$.
 \end{example}

\begin{remark}
When $G^{\der}$ is not simply connected, the hypothesis that $\mu$ is strongly Fontaine-Laffaille and not just Fontaine-Laffaille is necessary, even though it might not be initially expected.  Example~\ref{ex:pgl2} gives
an example of a representations $\rho$ valued in $\GL_2$ with quotient $\rho'$ valued in $\PGL_2$ with associated Hodge-Tate weights $0$ and $\frac{p-1}{2}$.  This is Fontaine-Laffaille but not strongly Fontaine-Laffaille. Our result shows the crystalline deformation ring for $\rho$ is formally smooth, but does not apply for $\rho'$.  
We do not expect the deformation ring for $\rho'$ to be smooth.

The group $G^{\der}$ being simply connected or $\mu$ being strongly Fontaine-Laffaille is actually a proxy for a more technical condition (the Kisin variety being trivial, see Theorem~\ref{thm:maintechnical}) which we expect to hold for most $\rhobar$ when $\mu$ is Fontaine-Laffaille even if $G^{\der}$ is not simply connected.  
\end{remark}

\begin{remark} The condition that $p \nmid \# \pi_1(G^{\ad})$ in Theorem~\ref{thm:mainintro} is actually two conditions: the center of $G^{\der}$ is prime to $p$ order and $p \nmid \# \pi_1(G^{\der})$ where $G^{\der}$ is the derived group of $G$.  The former only appears in a tangent space estimate (Proposition~\ref{prop:ts}) and can likely be removed.   The latter condition arises out of serious technical issues with the theory of mixed characteristic affine Schubert varieties for these groups.   We do not know if this condition can be removed.  
\end{remark} 

Finally, we discuss the condition that $R_\rhobar^{\mu,\square}$ is non-empty, or equivalently that $\rhobar$ admits a crystalline lift with $p$-adic Hodge type $\mu$.   As discussed in \S\ref{ss:overview}, our primary analysis is in characteristic $p$.  We expect proving the existence of crystalline lifts with particular $p$-adic Hodge type to require unrelated techniques, and so are content to assume the existence of a lift in  Theorem~\ref{thm:mainintro}.   This is not an issue for applications to automorphy lifting where a lift is often given.  However, it is a manner in which our methods are weaker than usual Fontaine--Laffaille theory.   Given that $\rhobar$ admits a lift of type $\mu$ (i.e., $\Spf R^{\mu, \square}_{\rhobar}$ is nonempty), \cite{levin15} shows that there is a $G$-Kisin module $(\fP, \phi_{\fP})$ with coefficients in $\F$ giving rise to $\rhobar|_{\Gamma_\infty}$, and the conditions in the theorem ensure that $\fP$ is unique.   Furthermore, the elementary divisors of $\phi_{\fP}$ (with respect to variable $u$ and decomposed over embeddings of $W(k)$ into $\Lambda$)  are given by dominant cocharacters $\mu' = (\mu'_\sigma)_{\sigma \in \cJ} $ such that $\mu'_\sigma \leq \mu_\sigma$ in the Bruhat order.  We call $\mu'$ the \emph{shape} of $\fP$.   The condition that $\mu$ is a potential $p$-adic Hodge type in Theorem~\ref{thm:mainintro} is that $\mu' = \mu$.   The theorem we prove in \S \ref{sec:proof} is more general than this, for example, if we assume that $\mu$ is regular then the assumption that $\mu' = \mu$ is not necessary as long as we assume existence of a lift (see Remark~\ref{rmk:cocharacterhyp}). 

If $\rhobar$ is tamely ramified such that the image of $\rhobar$ lies in the normalizer of a maximal torus $T \subset G$,  it is often possible in a combinatorial way to construct crystalline lifts of specified $\mu$ which are also similarly valued in the normalizer of $T$.  These are often referred to as \emph{obvious} or \emph{explicit} crystalline lifts and give a class of $\rhobar$ to which our theorem applies.   The combinatorics of these lifts for a general group $G$ is explored to some extent in Gee--Herzig--Savitt \cite{ghs18} in the context of the weight part of Serre's conjecture.

Lin has also recently investigated the existence of crystalline lifts  \cite{lin2021crystalline,lin2021lyndondemushkin}.  

\begin{remark} Although we have not attempted to prove it here, our expectation is that if $\rhobar$ admits a lift of type $\mu$ then the shape of $\fP$ must be exactly $\mu$ rather than strictly smaller.   This is consistent with the fact that Fontaine--Laffaille modules only deform in fixed weight for $\GL_n$.
\end{remark}      

\begin{remark}
It is not clear whether an alternate approach using a theory of Fontaine--Laffaille modules with $G$-structure could be used to prove Theorem~\ref{thm:mainintro}.  The advantage of this approach would be that it would also produce a crystalline lift, showing that the crystalline deformation ring is non-empty.
While the category of Fontaine--Laffaille modules is a tensor category, the functor relating Galois representations and Fontaine--Laffaille modules has limited compatibility with tensor products.  In particular, for Fontaine--Laffaille modules $M_1$ and $M_2$ the Galois representation associated to $M_1 \otimes M_2$ can be shown to be the tensor product of the Galois representations associated to $M_1$ and $M_2$ only if the ``weights'' of all three Fontaine-Laffaille modules are in an interval of length $p-2$.
See \cite[Fact 4.12]{booher19}, and for more details see the appendix of the arXiv version of \emph{loc. cit}.  This causes substantial technical problems relating $G$-valued Galois representations to Fontaine--Laffaille modules with $G$-structure which we do not know how to resolve.  This limitation explains why in \cite{patrikis06,booher19} the Hodge-Tate weights must lie in an interval of length less than $(p-1)/2$ instead of $p-1$: one tensor product is needed to study the duality pairing.
\end{remark}

\subsection{Overview of the Proof} \label{ss:overview}

We prove Theorem~\ref{thm:mainintro} by relating $G$-valued Galois representations to $G$-Kisin modules.
 Many of the techniques are inspired by \cite{levin15}, which dealt with the very special case when the $p$-adic Hodge type is minuscule, and by \cite{lllm18} and \cite{localmodels} which introduce a $p$-adic approximation to the monodromy condition in $p$-adic Hodge theory and connect this approximation to the geometry of affine Schubert varieties in the case of tamely potentially crystalline representation of small weight.

In \S\ref{sec:prelim}, we review some basics about $G$-bundles, affine Grassmanians, the Tannakian formalism for dealing with $G$-valued representations, and $p$-adic Hodge theory for $G$-valued representations, in preparation for introducing $G$-Kisin modules in \S\ref{sec:gkisin}.  
We define a variety of deformation problems for $G$-Kisin modules and $G$-valued Galois representations in \S\ref{ss:deformationproblems}, and make precise the notion of ``potential $p$-adic Hodge type for $\rhobar$'' that appears in Theorem~\ref{thm:mainintro}.  The Kisin resolution $X_{\rhobar}^{\mu} \rightarrow \Spec R^{\mu, \square}_{\rhobar}$ is our main tool for relating Galois representations and Kisin modules.  We will ultimately show the resolution is an isomorphism and that $X^{\mu}_{\rhobar}$ is formally smooth.

The analysis of the Kisin resolution occurs in \S\ref{ss:kisinvariety} and \S\ref{ss:forgetkisin}, and uses that $p$ is unramified in $K$ and that $\mu$ is strongly Fontaine-Laffaille or that $G^{\der}$ is simply connected and $\mu$ is Fontaine-Laffaille.  (This is the only step of the argument where $\mu$ being Fontaine-Laffaille does not suffice.)  An argument using the Bruhat order on affine Schubert varieties shows the Kisin variety is trivial, and hence $X_{\rhobar}^{\mu}$ is local.    We then make a delicate tangent space argument to show the Kisin resolution is an isomorphism. 

Next, we embed $X_{\rhobar}^{\mu}$ (up to formal variables)  into rigidified deformations of $G$--Kisin modules of type $\leq \mu$, which in this introduction we will denote $\widetilde{D}_{\overline{\fP}}^{\leq \mu}$.   We show it is an embedding in \S\ref{ss:forgetgalois} using the theory of $(\varphi,\widehat{\Gamma})$-modules with $G$-structure developed in \cite{levin15}.  We have to show that $G$--Kisin modules of type $\leq \mu$ admit at most one crystalline $\widehat{\Gamma}$-structure; this requires that $\mu$ is Fontaine-Laffaille and $p$ is unramified in $K$.  

Given that the generic fiber of  $X_{\rhobar}^{\mu}$ consists of crystalline representation, it factors through the flat locus $\widetilde{D}_{\overline{\fP}}^{\leq \mu, \nabla_{\infty}} \subset \widetilde{D}_{\overline{\fP}}^{\leq \mu}$ which satisfies the monodromy condition characterizing $G$-Kisin modules corresponding to crystalline representations.
Let $R^{\nabla_{\infty}}$ (resp. $R$) represent $\widetilde{D}_{\overline{\fP}}^{\leq \mu, \nabla_{\infty}}$ (resp. $\widetilde{D}_{\overline{\fP}}^{\leq \mu}$). 
We let $C \in G((W(k) \otimes R/pR)[\![u]\!] [1/u])$ represent the Frobenius on the universal $G$--Kisin module over $R/p R$.   Via Tannakian formalism, we define an element $u \frac{dC}{du} C^{-1}$ in $\Lie G \otimes (k \otimes R/pR)(\!(u)\!)$.  The condition that 
\begin{equation} \label{intro:modpmono}
u \frac{dC}{du} C^{-1} \in  \Lie G \otimes (k \otimes R/pR)[\![u]\!],
\end{equation}
i.e., that $u \frac{dC}{du} C^{-1}$ has no poles, is a closed condition on $R/pR$ which we refer to as the \emph{mod $p$ monodromy condition} and denote by $(R/pR)^{\nabla_1}$.  By $p$-adically approximating the true monodromy condition, we prove the following theorem:

\begin{thm}[Theorem~\ref{thm:monoapprox}] The special fiber of $\widetilde{D}_{\overline{\fP}}^{\leq \mu, \nabla_\infty}$ is contained in the locus of the special fiber of 
$\widetilde{D}_{\overline{\fP}}^{\leq \mu}$ where the mod-$p$ monodromy condition holds; equivalently, 
\[
\Spec R^{\nabla_{\infty}}/p R^{\nabla_{\infty}} \subset \Spec (R/pR)^{\nabla_1}.
\]
  \end{thm} 
 
The singularities of $\widetilde{D}_{\overline{\fP}}^{\leq \mu}$ are related to the singularities of a (mixed characteristic) affine Schubert variety $\prod_{\sigma \in \cJ} \Gr^{\leq \mu_{\sigma}}_G$, and the mod-$p$ monodromy condition can be descended to a condition on $ \prod_{\sigma \in \cJ} \Gr^{\leq \mu_{\sigma}}_{G_{\F}}$, an affine Schubert variety.   The ultimate source of smoothness then is that the differential equation \eqref{intro:modpmono} cuts out a smooth subvariety of the affine Schubert variety (Theorem \ref{thm:monodromyschubert}).

To complete the proof of Theorem~\ref{thm:mainintro}, we collect the relationships between all the deformation problems in Theorem~\ref{thm:diagram} and compare dimensions.   Theorem~\ref{thm:maintechnical} provides a more technical version of Theorem~\ref{thm:mainintro}.  It also offers some information for free about the non-existence of crystalline lifts with particular $p$-adic Hodge types (Corollary~\ref{cor:nolift}).

\subsection{Acknowledgments}
We thank  Matthew Emerton, Toby Gee,  
 Florian Herzig, 
  Timo Richarz, 
and  Niccolo' Ronchetti 
for helpful conversations. We thank the referees for a very careful reading.
The first author was partially supported by the Marsden Fund Council administered by the Royal Society
of New Zealand. 
The second author was supported by a grant from the Simons Foundation/SFARI (\#585753) and supported in part by NSF Grant DMS-1952556.
 
 \subsection{Notation} \label{ss:notation}
 
 We collect some standard notation that will repeatedly arise for easy reference.
 
 Fix a prime $p$.
Our standard convention is for $\Lambda$ to be the ring of integers in a $p$-adic field $L$ with residue field $\F$, and for $G$ to be a split reductive group over $\Lambda$.  We denote the derived group (resp. adjoint group) of $G$ by $G^{\der}$ (resp. $G^{\ad}$) and often use $Z$ and $Z^{\der}$ to denote the centers of $G$ and $G^{\der}$.   In many places, we will assume that $p \nmid \pi_1(G^{\der})$. 

We often fix a finite field $k$ and a $p$-adic field $K$ with ring of integers $W(k)$.  We let $G' := \Res_{(W(k) \otimes \Lambda)/ \Lambda} G$.  When $\Lambda$ contains a copy of $W(k)$,  we have that $G' = \prod_{\sigma \in \Hom(K, L)} G$.
 
\subsubsection{Root Systems} \label{notation:rootdata}
 We fix a split maximal torus $T \subset G$ and a Borel subgroup $B$ of $G$ containing $T$.  Let $X^*(T)$ and $X_*(T)$ denote the character and cocharacter lattices of $T$.  We let $\Phi_G \subset  X^*(T)$ and $\Phi_G^\vee \subset X_*(T)$ denote the roots and coroots for $(G,T)$.
   We let $\langle \mu, \chi \rangle$ denote the standard pairing between a cocharacter $\mu$ and a character $\lambda$.
 
Let $\Phi_G^+$ denote the set of positive roots with respect to $B$, and $X_*(T)_+$ the set of dominant cocharacters.  For a cocharacter $\mu$, let $\mu^{\dom}$ be the unique dominant cocharacter in the same Weil-orbit as $\mu$.  There is a partial ordering on $X_*(T)_+$ where $\lambda \leq \mu$ if $\mu - \lambda$ is a non-negative combination of simple coroots.

For a cocharacter $\mu$ of $G$, we define
\begin{equation}
h_\mu := \max_{\alpha \in \Phi_G} \langle \mu,\alpha \rangle.
\end{equation}

\subsubsection{Lie Algebras} \label{notation:liealg}
We let $\g$ denote the Lie algebra of $G$, $\ft$ the Lie algebra of a split maximal torus $T$, and denote the root space for $\alpha \in \Phi_G$  by $\g_\alpha$.  For a $\Lambda$-algebra $A$, we let $\g_A$ denote $\g \tensor{\Lambda} A$, or equivalently the Lie algebra of $G_A$.  We use $\g'$ (resp. $\ft'$, ...) for the Lie algebra of the Weil restriction $G'$ (resp. the Lie algebra of a split maximal torus $T'$, ...).

\subsubsection{Affine Grassmanians} \label{notation:affinegrassmanian}
Almost all of our work with affine Grassmanians will involve the group $G'$.    We use $\LG'$, $\LpG'$, and $\Gr_{G'}$ to denote the loop group, positive loop groups, and affine Grassmanian for $G'$ over $\Lambda$.  When $p \nmid \pi_1(G^{\der})$ (or equivalently $p \nmid \pi_1((G')^{\der})$), we also use $\Gr_{G'}^{\leq \mu}$ over $\Lambda$; see \S\ref{ss:affinegrassmanian}. 

Over the residue field $\F$, we let $\Gr^{\leq \mu}_{G'_{\F}}$ and  $\Gr^{\circ, \mu}_{G'_{\F}}$ denote the Schubert variety and open Schubert cell for a cocharacter $\mu$ of $G'$.

\subsubsection{Galois Representations} \label{notation:galoisgps}
Let $K$ be a $p$-adic field unramified over $\Q_p$ with fixed algebraic closure $\Kbar$.  Let $k$ be its residue field and $W(k)$ the ring of integers in $K$.  Let $\Gamma_K = \Gal(\Kbar/ K)$ be its absolute Galois group.  We usually fix a continuous homomorphism $\rhobar : \Gamma_K \to G(\F)$.

Fix a compatible system $\{p^{1/p}, p^{1/p^2}, \ldots \}$ of $p$-power roots of $p$, and let $K_\infty = K(p^{1/p},p^{1/p^2},\ldots)$.  We let $\Gamma_\infty$ be the absolute Galois group of $K_\infty$ (with respect to $\overline{K}$).

\subsubsection{Big Rings} \label{notation:bigrings}
Let $\fS := W(k)\pseries{u}$ and $E(u) = u -p$.  We set $\cO_{\cE}$ to be the $p$-adic completion of $\fS[\frac{1}{u}]$.  The rings $\fS$ and $\cO_{\cE}$ are equipped with a Frobenius $\varphi$ by extending the standard Frobenius on $W(k)$ by sending $u$ to $u^p$.

For any $p$-adically complete $\Z_p$-algebra $A$, we set $\fS_A := (W(k) \otimes_{\Zp} A)[\![u]\!]$ and $\cO_{\cE,A}$ is the $p$-adic completion of $\fS_A[1/u]$.  We extend Frobenius to both by having it act trivially on $A$.  Note that if $A$ is finite over $\Zp$ (for example, Artinian) then $\fS_A = \fS \otimes_{\Zp} A$.   

\subsubsection{Kisin Modules} \label{notation:kisinmodules}

A type for $G$ is a cocharacter $\mu$ for the group $G'$.  
We usually use $\fP$ to denote a $G$-Kisin module (Definition~\ref{defn:gkisin}), and denote the category of $G$-Kisin modules of type $\leq \mu$ with coefficients in a $\Lambda$-algebra $A$ by $Y^{\leq \mu}(A)$ (Definition~\ref{defn:typemu}).

\subsubsection{Deformation Functors} \label{notation:deformations}
Let $\cC_{\Lambda}$ (respectively $\widehat{\cC}_{\Lambda}$) denote the categories of coefficient $\Lambda$-algebras: local Artinian $\Lambda$-algebras with residue field $\F$ (respectively complete local Noetherian $\Lambda$-algebras with residue field $\F$).   Morphisms are local $\Lambda$-algebra maps. 

Our deformation problems are functors from $\cC_{\Lambda}$ (or $\widehat{\cC}_{\Lambda}$) to sets.  A variety of deformation problems ($D_{\rhobar}^{\mu,\square}$, $D_{\rhobar, \fPbar}^{\mu,\square}$, etc.) and their associated deformation rings ($R_{\rhobar}^{\mu,\square}$, $R_{\rhobar, \fPbar}^{\mu,\square}$, etc.) are found in Fact~\ref{fact:crystallinedeformation}, Definition~\ref{defn:deformationproblems1}, Definition~\ref{defn:deformationproblems2}, Definition~\ref{defn:flatclosures},
and Definition~\ref{defn:defringmonodromy}.  Theorem~\ref{thm:diagram} shows the relationships between many of them.

\section{Preliminaries} \label{sec:prelim}

Let $G$ be a split reductive group defined the ring of integers $\Lambda$ in a $p$-adic field $L$ with residue field $\bF$.
We begin by reviewing some background about $G$-bundles, the Tannakian formalism, affine Grassmanians, and $p$-adic Hodge theory for $G$-valued representations.  

\subsection{\texorpdfstring{$G$}{G}-Bundles}
 In what follows, all $G$-bundles are with respect to the fppf topology. We begin by recalling a few things about $G$-valued representations and trivializations of $G$-bundles. 
 
 \begin{defn}
For any profinite group $\Gamma$ and a finite $\Lambda$-algebra $A$, let $\GRep_A(\Gamma)$ be the category of pairs $(P,\rho)$ where $P$ is a $G$-bundle on $\Spec A$ and $\rho : \Gamma \to \Aut_G(P)$ is a continuous homomorphism (giving $A$ the $p$-adic topology).
\end{defn}

\begin{defn} \label{defn:trivialization}
A trivialization of a $G$-bundle $\fP$ on $X$ is an isomorphism with the trivial $G$-bundle $\cE^0_X$ on $X$.  
\end{defn}

We know that if $A$ is a complete local $\Lambda$-algebra with finite residue field, then any $G$-bundle on $A$ is trivializable \cite[Proposition 2.1.4]{levin15}.  Thus in all cases we will consider, $\Aut_G(P)$ is (non-canonically) isomorphic to $G(A)$.
Identifying $\Aut_G(P)$ with $G(A)$ is equivalent to choosing a trivialization. 
A continuous homomorphism $\rho : \Gamma \to G(A)$ is equivalent to $(P,\rho') \in \GRep_A(\Gamma)$ together with a trivialization of $P$.

\begin{remark}
For $G = \GL_n$, $G$-bundles are vector bundles in the Zariski topology.  If $A$ is a complete local $\Lambda$-algebra with finite residue field then the vector bundle is given by a finitely generated projective (hence free) module over $A$ and a trivialization is a choice of basis.
\end{remark}

\subsection{Tannakian Formalism}
 
Now let $\fRep_\Lambda(G)$ denote the category of representations of the algebraic group $G$ on finite free $\Lambda$-modules, 
and $\Proj_A$ denote the category of projective $A$-modules.
For a $\Lambda$-algebra $A$, a $G$-bundle $\fP$ on $\Spec (A)$ is equivalent to a fiber functor $\eta : \fRep_\Lambda(G) \to \Proj_A$ (a faithful exact tensor functor sending the trivial representation to the trivial $A$-module of rank $1$); the fiber functor associated to $\fP$ sends a representation $G \to GL(V)$ to the pushout $\fP_V$ (see \cite[Theorem 2.1.1]{levin15} for this level of generality).  This allows us to give ``additional structure'' on the $G$-bundle $\fP$ by factoring $\eta$ through a category of projective $A$-modules with ``additional structure''; equivalently, by specifying ``additional structure'' on the $\Lambda$-modules $\eta(V)$ for each $G \to \GL(V)$ that are compatible in an appropriate sense. 

\begin{example} \label{ex:tannakian} We make this idea precise in a few important examples.
\begin{enumerate}[(a)]
\item \label{firsttannakian} An automorphism of the fiber functor $\eta$ corresponds to an automorphism of the $G$-bundle (see \cite[Theorem 2.5.2]{levinthesis} for this level of generality).  In our situations, the $G$-bundle can be trivialized, so $\eta(V) \simeq V_A$ 
and a choice of trivialization identifies $\Aut(\eta) \simeq \Aut_G(\fP) \simeq G$.  In particular, we see that $G(A)$ is in bijection with collections of elements $g_V \in \GL(V_A)$ for $V \in \fRep_\Lambda(G)$ that are compatible with tensor products and exact sequences  in the sense that $g_{V_1 \otimes V_2 } = g_{V_1} \otimes g_{V_2}$ and such that for short exact sequences $0 \to V_1 \to V \to V_2 \to 0$ in $\fRep_\Lambda(G)$ the following diagram commutes:
\[
\xymatrix{
0 \ar[r] & V_1 \ar[d]^{g_{V_1}} \ar[r]& V\ar[r] \ar[d]^{g_{V}} & V_2\ar[r]\ar[d]^{g_{V_2}} & 0 \\
0 \ar[r] & V_1 \ar[r]& V\ar[r] & V_2\ar[r] & 0. \\
}
\]
Furthermore, $g_{\mathbf{1}} = \id$.

\item  We can extend the previous example to treat a representation of a group $\Gamma$ valued in $G(A)$ as a compatible set of representations of $\Gamma$ on $V_A$ for $V \in \fRep_\Lambda(G)$.

\item  \label{endomorphismtannakian} Similarly, an element $X \in \g_A$ is equivalent to endomorphisms $X_V$ in $\End(V_A)$ for $V \in \fRep_\Lambda(G)$ (with $X_{\mathbf{1}} =0$) that are compatible with exact sequences and with tensor products in the sense that $X_{V \otimes V'} = X_{V} \otimes 1 + 1 \otimes X_{V'}$.  This follows from the dual-number interpretation of the Lie algebra and \eqref{firsttannakian}.

\item \label{gradingstannakian} A $G$-grading is a fiber functor which factors through the category of graded vector bundles on $\Spec(A)$.  This corresponds to giving a grading on each $\eta(V) \simeq V_A$ for each $V \in \fRep_\Lambda(G)$.  The gradings are compatible with exact sequences, and with tensor products in the sense that (using subscripts to denote graded pieces)
\[
(V \otimes V')_n = \bigoplus_{ i+j=n} V_i \otimes V'_j.
\]
Furthermore, $(\mathbf{1})_0 = \mathbf{1}$.  This can equivalently be described via a cocharacter $\mu : \Gm \to \Aut_G(\fP)$, arising by letting $t \in \Gm$ act on $V_n$ via multiplication by $t^n$. 

\item  \label{filtrationstannakian}  There is an analogous description of $G$-filtrations as fiber functors $\eta$ factoring through the category $\Fil_{A}$ of vector bundles on $\Spec(A)$ with decreasing filtration. 
 A splitting is an isomorphism of a $G$-grading with
the composition of $\eta$ with the associated graded functor.  In the setting we are working, all $G$-gradings are split \cite[Proposition IV.2.2.5]{saavedra72}.
The \emph{type} of the filtration is the geometric conjugacy class of cocharacters of $G$ giving the splitting.

\end{enumerate}
\end{example}

\subsection{Affine Grassmannians} \label{ss:affinegrassmanian}
Let $G$ be a split reductive group over $\Lambda$, and fix a split maximal torus $T$ and set of positive roots.    We will make use of a mixed characteristic affine Grassmannian.   In the context where they arise, the affine Grassmannians are centered at $u = p$ and involve a Weil--restriction. We will focus on that setup.   

We let $\LG$ and $\LpG$ denote the loop group and the positive loop group for $G$ over $\Lambda$.  
For a $\Lambda$-algebra $A$, we have that
\[
\LG(A) = G(A\lseries{u - p}) \quad \text{and} \quad \LpG(A) = G(A\pseries{u-p}).
\]
The affine Grassmanian $\Gr_{G}$ is the fpqc quotient $\LpG \backslash \LG$.   It is an ind--projective scheme.  If $\rho:G \rightarrow H$ is any homomorphism of algebraic groups, then there is natural induced map $\rho_*:\Gr_G \rightarrow \Gr_H$.  Given $C \in \LG(A)$, we let $[C]$ denote the equivalence class in the quotient $\Gr_G(A)$.

The special fiber $\Gr_{G_{\F}}$ is just the usual affine Grassmannian for $G$ centered at $u = 0$.   For any cocharacter $\mu$ of $T$, we use $(u-p)^{\mu}$ to denote the element $\mu(u-p) \in L T(\Lambda) \subset LG(\Lambda)$ and $[(u-p)^{\mu}]$ denote the corresponding $\Lambda$-point of $\Gr_{G}$. When working in special fiber, we use $u^{\mu}$ and $[u^{\mu}]$ with the obvious meaning.

Given dominant cocharacters $\mu'$ and $\mu$ of $G$, recall that $\mu' \leq \mu$ if $\mu-\mu'$ is a non-negative linear combination of simple coroots. 
 Let $\Gr^{\leq \mu}_{G_{\F}}$ and  $\Gr^{\circ, \mu}_{G_{\F}}$ denote the affine Schubert variety and open affine Schubert cell for the cocharacter $\mu$,  which are the orbit closure and orbit  respectively of $\LpG$ on $u^{\mu}$.  These only depend on the conjugacy class of $\mu$, so it is convenient to work with dominant representative.  
 
 \begin{thm} \label{thm:affineschubert}
 Assume that $p \nmid \# \pi_1(G^{\mathrm{der}})$ 
 then there is a projective flat $\Lambda$--scheme $\Gr_G^{\leq \mu}$ such that 
\[
(\Gr_G^{\leq \mu})_{\F} =   \Gr^{\leq \mu}_{G_{\F}}
\]
and the generic fiber $(\Gr_G^{\leq \mu})_{L}$ is the reduced closure of the $\LpG_{L}$--orbit of $(u-p)^{\mu}$. 
\end{thm} 

\begin{proof}  When $G$ is simply-connected, 
there is a construction of affine Schubert varieties over any base on page 52 of \cite{FaltingsAlgebraicLoopGroups}.  Under the assumption on the fundamental group, one can reduce to this case.  For lack of a better reference, the theorem is also a consequence of \cite[Theorem 9.3]{pz13} since $\Gr^{\leq \mu}_{G}$ is a very special case of the Pappas--Zhu local model associated to $G$, the conjugacy class of $\mu$, and the maximal compact parahoric $G(\Lambda)$. 
\end{proof}

\begin{remark} \label{remark:closure}
Let $\mu$ be dominant.  
 The affine Schubert variety $\Gr^{\leq \mu}_{G_\F}$ is the union of $\Gr^{\circ, \mu'}_{G_\F}$ for dominant $\mu' \leq \mu$. 
This can also be phrased in terms of the Cartan decomposition.  For $C \in \LG(\F)$,  we have that $[C] \in \Gr^{\leq \mu}_{G_\F}$ if and only if there exists dominant $\mu' \leq \mu$ such that
$$ C \in \LpG(\bF) u ^{\mu'} \LpG(\bF).$$
\end{remark}

For later use, we now record a standard fact about tangent spaces.  We use our standard notation about root systems and Lie algebras from \S\ref{notation:rootdata}-\ref{notation:liealg}.

\begin{lem} \label{lem:schuberttangent}
The map on tangent spaces induced by right multiplication by $u^\mu$ on $\Gr_{G_\F}$ identifies 
 \begin{equation} \label{eq:tspace}
 V_\mu := \bigoplus _{\alpha \in \Phi_{G}, \langle \mu ,\alpha \rangle < 0 } \left( u^{\langle \mu,\alpha \rangle} \g_{\alpha, \F[\![u]\!]} \right)/ \g_{\alpha, \F[\![u]\!]} \subset \Lie \Gr_{G_\F} := \g_{\F(\!(u)\!)} / \g_ {\F[\![ u]\!]}
 \end{equation}
with the tangent space of $\Gr^{\circ, \mu}_{G_{\F}}$ at $[u^\mu]$ as a subspace of the tangent space of $\Gr_{G_\F}$ at $[u^\mu]$.
\end{lem}

\begin{proof}
First, notice that the stabilizer of $[u^{\mu}] \in \Gr_{G_\F}$ under right multiplication by $\LpG_\F$ is  $(\LpG_\F \cap u^{-\mu} \LpG_\F u^\mu )$, and so we get a locally closed immersion 
\[
 (\LpG_\F \cap  u^{-\mu} \LpG_\F u^\mu )  \backslash (\LpG_\F )  \into \LpG_\F \backslash \LG_\F
\]
with image $\Gr_{\F}^{\circ, \mu}$ (cf. the proof of \cite[Proposition 2.1.5]{zhu17}).  Right multiplication by $u^{-\mu}$ gives an isomorphism from the tangent space of $\Gr_{\F}^{\circ, \mu}$ at $u^\mu$ to
\[
\left( \Ad_G(u^\mu) \g_{\F\pseries{u}} \right) / \left( \g_{\F\pseries{u}} \cap \Ad_G(u^\mu) \g_{\F\pseries{u}} \right) = V_\mu \subset \g_{\F\lseries{u}} / \g_{\F\pseries{u}} = \Lie \Gr_{G_\F}. \qedhere
\]
\end{proof}

In our analysis, the affine Grassmanians that appear will actually be for the group 
$$G' := \Res_{(W(k) \tensor{\Z_p} \Lambda)/\Lambda} G.$$  
  Assuming $\Lambda$ is sufficiently large, $G'$ is just a product of $\Hom_{\Qp}(K, L)$--copies of $G$ and so the above discussion applies with $G$ replaced by $G'$.  
Any $\mu \in X_*(T') = X_*(T)^{\Hom(K, L)}$ defines a cocharacter of the $G'$ and an affine Schubert variety $\Gr_{G'}^{\leq \mu}$.

\subsection{\texorpdfstring{$p$}{p}-adic Hodge Theory}

We briefly review the translation of $p$-adic Hodge theory to apply to $G$-valued representations; see for  example \cite[\S2.4]{levin15} for details.
Let $K$ be a finite extension of $\Q_p$ and $\Gamma_K$ be the absolute Galois group of $K$.  For any $L$-algebra $B$ and continuous representation $\rho : \Gamma_K \to  G(B)$, we say that $\rho$ is crystalline (resp. semi-stable, de Rham) if $\rho_V$ is crystalline (resp. semi-stable, de Rham) for all $V \in \fRep_\Lambda(G)$.  As $G$ is reductive, this can be checked on a single faithful representation.

\begin{defn}
A \emph{$p$-adic Hodge type} for $G$ is a geometric conjugacy class of cocharacters of $(\Res_{K \tensor{\Qp} L /L  } G) _{\ol{L}}$.  
\end{defn}

A $p$-adic Hodge type is equivalent to a collection of geometric cocharacters of $G_{\ol{L}}$ indexed by $\Qp$-embeddings of $K$ into $\ol{L}$.  For a fixed split maximal torus $T$ in $G$, each type can be represented by an element $\mu  \in X_*(T)^{\Hom_{\Qp}(K,\ol{L})}$.  We denote the conjugacy class by $[\mu]$.  When there is no danger of confusion, we often speak of a cocharacter $\mu$  of $G$ being a $p$-adic Hodge type (or just a type) instead of the geometric conjugacy class.

The functor $D_{\dR}$ from the category of de Rham representations on projective $B$-modules to the category of filtered $K \tensor{\Qp} B$-modules is a tensor exact functor.  For a de Rham $\rho : \Gamma_K \to  G(B)$, composing with $D_{\dR}$ defines a tensor-exact functor from $\fRep_L(G_L)$ to $\Fil_{K \tensor{\Qp} B}$ which we denote by $\cF_{\rho}^{\dR}$.   Let $[\mu]$ be a $p$-adic Hodge type for $G$.  We say that $\rho$ has type $[\mu]$ provided that $\cF_{\rho}^{\dR}$ has type $[\mu]$.

Now fix a continuous $\rhobar : \Gamma_K \to G(\F)$.  
There are framed deformation rings whose characteristic zero points incorporate conditions from $p$-adic Hodge theory.

\begin{fact} \label{fact:crystallinedeformation}
Let $\mu$ be a $p$-adic Hodge type of $G$ and $R_{\rhobar}^{\mu,\square}$ be the framed crystalline deformation ring 
for $\rhobar$ with $p$-adic Hodge type $\mu$.  It is a complete local $\Lambda$-algebra which is $\Lambda$-flat and reduced.
Its relative dimension over $\Lambda$ is $\dim G_{\bF} + \dim P_{\mu , \bF} \backslash G'_{\bF}$.
\end{fact}

The ring $R_{\rhobar}^{\mu, \square}$ is the flat closure of the locus in $(R^{\square}_{\rhobar}[1/p])$ of crystalline representations with $p$-adic Hodge type $\mu$ constructed in \cite[Theorem 4.0.12]{balaji}.  More precisely, for a finite local $L$-algebra $A$, an $A$-valued point of $R_{\rhobar}^{\square}$ factors through $R_{\rhobar}^{\mu,\square}$ if and only if the associated Galois representation is crystalline with $p$-adic Hodge type $\mu$. 
Since the local rings of the generic fiber $(R_{\rhobar}^{\mu,\square}[1/p])$ are formally smooth, the generic fiber is reduced.  Since $R_{\rhobar}^{\mu,\square}$ is $\Lambda$-flat by construction, it is also reduced.  The relative dimension comes from a calculation of the dimension of the generic fiber \cite[Theorem 5.1.5]{balaji}.

\section{Deformations of \texorpdfstring{$G$}{G}-Kisin modules and Galois Representations} \label{sec:gkisin}

 Fix a $p$-adic field $K$ unramified over $\Q_p$ with residue field $k$ and ring of integers $W(k)$. 
 Let $\Gamma_K$ denote the absolute Galois group of $K$ with respect to a fixed algebraic closure $\overline{K}$.  
As before, let $\Lambda$ be the ring of integers in a $p$-adic field $L$ with residue field $\F$, and $G$ be a split reductive group defined over $\Lambda$.   Assume $L$ is sufficiently large such that it contains a copy of $K$.  Throughout this section, we also assume that $p \nmid \# \pi_1(G^{\der})$.

\subsection{\texorpdfstring{$G$}{G}-Kisin modules and Types}

 We begin by recalling the definition of a Kisin module with $G$-structure and some related notions from \cite[\S2.2]{levin15}.  These are generalizations of the objects introduced in \cite{kisin09}, which addressed the case that $G = \GL_n$.
We work with $p$-adically complete $\Lambda$-algebras $A$, and use the 
 ring $\fS_A$ equipped with Frobenius $\varphi$ introduced in \S\ref{notation:bigrings}.
   We continue to work with $G$-bundles in the fppf topology.   

\begin{defn}\label{defn:gkisin}
For a $p$-adically complete  $\Lambda$-algebra $A$, a $G$-Kisin module with coefficients in $A$ is a pair $(\fP, \phi_{\fP})$ where $\fP$ is a $G$-bundle on $\Spec \fS_A$ and $\phi_{\fP}: \varphi^*(\fP)[1/ E(u)] \simeq \fP[1/E(u)] $ is an isomorphism of $G$-bundles.  
\end{defn}

We will often suppress $\phi_\fP$ and speak of $\fP$ as being a $G$-Kisin module (denoting $\phi_\fP$ by $\phi$ if the subscript is not necessary to disambiguate).

\begin{defn}   When $G = \GL(V)$ where $V$ is finite free over $\Lambda$,  a $G$--Kisin module with coefficients in $A$ is equivalent to a finitely generated projective module $\fM$ over $\fS_A$ of rank equal to the rank of $V$, together with an isomorphism $\phi_{\fM}:\varphi^*(\fM)[1/ E(u)] \simeq \fM[1/E(u)]$.  Let $a \leq b$ be integers.   We say that $\fM$ has height in $[a, b]$ if 
\[
E(u)^a \fM \supset \phi_{\fM} (\varphi^*(\fM)) \supset E(u)^b \fM.
\]
We say $\fM$ has \emph{bounded height} if it has height in $[a, b]$ for some integers $a$ and $b$. 
\end{defn} 

\begin{remark}
In the case of $\GL(V)$, there is often an effectivity condition which corresponds to $a \geq 0$, but there is no such notion for $G$-bundles since all $G$-bundle maps are isomorphisms.
\end{remark}

\begin{example} \label{gkisintannnakian}  For any representation $V \in \fRep_\Lambda(G)$, one can pushout a  $G$--Kisin module $\fP$ with coefficients in $A$ to a $\GL(V)$-Kisin module, which we  will denote by $\fP(V)$.   This construction gives an equivalence between $G$-Kisin modules and faithful, exact tensor functor from $\fRep_\Lambda(G)$ to the category of Kisin modules with bounded height.  More precisely, a $G$-Kisin module is equivalent to the data of a fiber functor $\eta$ corresponding to the underlying $G$-bundle plus the data of isomorphisms $$\Phi_V : \varphi^*(\eta(V)[1/E(u)]) \to \eta(V)[1/E(u)]$$ for each $V \in \fRep_\Lambda(G)$ that are compatible with exact sequences and tensor products.
\end{example}

\begin{example} \label{ex:matrixofphi}  Let $\fP$ be a $G$--Kisin module with coefficients in $A$.   If $\fP$ is a trivial $G$--bundle (this is equivalent to $\fP \mod u$ being trivial) and if one chooses a trivialization $\beta$, then we can associate to $\phi_{\fP}$ an element $C_{\fP, \beta} \in G(\fS_A[1/E(u)])$.  Changing trivialization by an element $D \in G(\fS_A)$ replaces $C_{\fP, \beta}$ by its $\varphi$-conjugate $D C_{\fP, \beta} \varphi(D)^{-1}$.  
\end{example} 

Recall that $G' = \Res_{(W(k) \tensor{\Z_p} \Lambda)/ \Lambda} G$. By assumption $L$ contains a copy of $K$, and so $G' = \prod_{\sigma \in \Hom(K, L)} G$.   A cocharacter $\mu$ of $G'$ is equivalent to a collection $(\mu_{\sigma})_{\sigma \in \Hom(K, L)}$ where each $\mu_{\sigma}$ is a cocharacter of $G$.

As we have assumed that $p \nmid \# \pi_1(G^{\der})$, we have an affine Schubert variety $\Gr^{\leq \mu}_{G'}$ associated to $\mu$ over $\Lambda$ as in \S\ref{ss:affinegrassmanian}.   It is projective, $\Lambda$-flat, and stable under right multiplication by $\LpG'$.   As the notation suggests, the special fiber of $\Gr_{G'}^{\leq \mu}$ is the affine Schubert variety $\Gr_{G'_{\F}}^{\leq \mu}$.

Let $\fP$ be a $G$--Kisin module with coefficients in a $p$-adically complete $\Lambda$-algebra $A$ together with a trivialization $\beta$.   The Frobenius then defines an element $C_{\fP, \beta} \in G(\fS_A[1/E(u)])$.   Since $A$ is $p$-adically complete, $\fS_A = (W(k) \otimes A)[\![u]\!] = (W(k) \otimes A)[\![u - p]\!]$ and so $ G(\fS_A[1/E(u)]) = \LG' (A)$.    Letting $[C_{\fP, \beta}]$ denote the class of $C_{\fP, \beta}$ in the quotient $\LpG'(A) \backslash  \LG' (A)$, the map $(\fP, \beta) \mapsto [C_{\fP, \beta}]$ defines an element $\Psi(\fP, \beta)$ of $\Gr_{G'}(A)$.   

 \begin{example} \label{ex:embeddingcomponents}
 As we assume that $L$ is sufficiently large to contain a copy of $K$, we can be more explicit.  Letting $\cJ$ be the set of embeddings of $W(k)$ into $\Lambda$, 
we see $$\LG'(A) = G(\fS_A[1/E(u)]) = \prod_{\sigma \in \cJ} G(A\pseries{u} [1/E(u)]).$$
 Thus we may represent $C_{\fP,\beta}$ as a tuple $(C^{\sigma}_{\fP,\beta})_{\sigma \in \cJ}$.  Notice that $\varphi(C_{\fP,\beta}) = ( \varphi(C^{\sigma \varphi^{-1}}_{\fP,\beta})) _{\sigma \in \cJ}$.
 \end{example} 

\begin{defn} \label{defn:typemu} Let $\fP$ be a $G$--Kisin module with coefficients in a $p$-adically complete $\Lambda$-algebra $A$.   Assume $\fP$ admits a trivialization $\beta$.  We say $\fP$ has \emph{type $\leq \mu$} if $\Psi(\fP, \beta) \in \Gr^{\leq \mu}_{G'}(A)$.    Note that this condition does not depend on the choice of trivialization because $ \Gr^{\leq \mu}_{G'}$ is stable under right multiplication by $\LpG'$.    Let $Y^{\leq \mu}(A)$ denote the category of $G$--Kisin modules with coefficients in $A$ and type $\leq \mu$.  
\end{defn}

The following elementary observation is key to many computations:

\begin{lem} \label{lem:adjheight}  Let $\fP$ be a $G$--Kisin module of type $\leq \mu$ with coefficients in a $p$-adically complete $\Lambda$-algebra $A$.    Then $\fP(\Lie G)$ has height in $[-h_{\mu}, h_{\mu}]$, where $h_{\mu} = \max_{\alpha \in \Phi_{G'}} \langle \mu,\alpha \rangle$.    
\end{lem}

\begin{proof}
This follows from considering the adjoint representation.
\end{proof}

If $\fP$ has coefficients in $\F$,  then $\fP$ has type $\leq \mu$ if and only if for some $\mu' \leq \mu$ and choice of trivialization $\beta$, $C_{\fP, \beta} \in L^+ G'(\F) u^{\mu'} L^+ G'(\F)$ (see Remark~\ref{remark:closure}).   We give this element a name following \cite{lllm18}: 

\begin{defn} \label{defn:shapemu} Let $\fP$ be a $G$-Kisin module with coefficient in $\F$ and $\mu'$ a dominant cocharacter of $G'$.   
We say that $\fP$ has \emph{shape} $\mu'$ provided that for any choice of trivialization $\beta$ we have  $C_{\fP, \beta} \in L^+ G'(\F) u^{\mu'} L^+ G'(\F)$.
\end{defn}

\subsection{\texorpdfstring{$G$}{G}-Kisin Modules and Galois representations}

The connection between $G$-Kisin modules and Galois representation is via the theory of \'{e}tale $\varphi$-modules, which uses the rings $\cO_{\cE,A}$ with Frobenius $\varphi$ recalled in \S\ref{notation:bigrings}.

\begin{defn}
For a $p$-adically complete $\Lambda$-algebra $A$, a $(\cO_{\cE,A},\varphi)$-module with $G$-structure with coefficients in $A$ is a pair $(P,\phi_P)$ where $P$ is a $G$-bundle on $\Spec \cO_{\cE,A}$ and $\phi_P : \varphi^*(P) \to P$ is an isomorphism.  We denote the category of such pairs by $\GMod_{\cO_{\cE,A}}^{\varphi}$.
\end{defn}

\begin{example} \label{ex:tannakian3}
An element $(P,\Phi_P) \in \GMod^\varphi_{\cO_{\cE,A}}$ is equivalent to a faithful, exact tensor functor from $\fRep_\Lambda(G)$ to the category of $(\cO_{\cE,A},\varphi)$-modules.
\end{example}

Since $E(u)$ is invertible in $\cO_{\cE}$, there is a natural map $\fS_A[1/E(u)] \to \cO_{\cE,A}$ for any $p$-adically complete $\Lambda$-algebra $A$.  We let $\epsilon_G$ denote the induced functor from the category of $G$-Kisin modules with coefficients in $A$ to the category of $(\cO_{\cE,A},\varphi)$-modules with $G$-structure.

\begin{defn}
Let $A$ be a $p$-adically complete $\Lambda$-algebra and $P \in \GMod_{\cO_{\cE,A}}^{\varphi}$.  A $G$-Kisin lattice of $P$ is a $G$-Kisin lattice inside $P$, i.e. a $G$-Kisin module $\fP$ with coefficients in $A$ together with an isomorphism $\alpha: \epsilon_G(\fP) \simeq P$ (compatible with $\phi_{\fP}$ and $\phi_{P}$).
\end{defn}

The category $\GMod_{\cO_{\cE,A}}^{\varphi}$ is connected with $G$-valued Galois representations.  
As in \S\ref{notation:galoisgps}, fix a compatible system $\{p^{1/p}, p^{1/p^2}, \ldots \}$ of $p$-power roots of $p$, and let $K_\infty  = K(p^{1/p}, p^{1/p^2}, \ldots )$ and $\Gamma_\infty = \Gal(\Kbar/ K_\infty)$. 
The following result is \cite[Proposition 2.2.4]{levin15}:

\begin{fact} \label{fact:etaleequiv}
For a $\Lambda$-algebra $A$ which is finite over $\Z_p$ and either $\Z_p$-flat or Artinian, there is a functor $T_{G,A} : \GMod_{\cO_{\cE,A}}^{\varphi} \to \GRep_A(\Gamma_\infty)$ giving an equivalence of categories with quasi-inverse $M_{G,A}$.
\end{fact}

\begin{defn} \label{defn:tildeT}
Let $\tT_{G,A}$ be the composition of $T_{G,A}$ with $\epsilon_G$. 
\end{defn}

\subsection{Deformation Problems} \label{ss:deformationproblems}
  
We will define a variety of deformation problems on the categories $\cC_{\Lambda}$ and $\widehat{\cC}_{\Lambda}$ of coefficient $\Lambda$-algebras from \S\ref{notation:deformations}.  Subscripts (like $\fPbar$ or $\rhobar$) denote the basic objects being deformed, while superscripts (like $\leq \mu$ or $\square$) impose conditions or specify auxiliary information to be included.

We fix a faithful representation $V_0 \in \fRep_\Lambda(G)$.  Let $\fPbar$ be a $G$-Kisin module with coefficients in $\F$, and fix a dominant cocharacter $\mu$ for $G'$.
For any integers $a \leq b$, we can define a deformation groupoid $ D^{[a,b]}_{\fPbar}$ on $\cC_{\Lambda}$ consisting of deformations $\fP$ of $\fPbar$ such that $\fP(V_0)$ has height in $[a, b]$.  
Choosing $a$ sufficiently small and $b$ sufficiently large depending on $\mu$, \cite[Proposition 3.3.9]{levin15} constructs a closed subgroupoid $D^{\leq \mu}_{\fPbar} \subset D^{[a,b]}_{\fPbar}$ consisting of deformations with type $\leq \mu$ (in \emph{loc. cit.} it is denoted $D^{\mu}_{\fPbar}$ but we find $\leq \mu$ to be more descriptive here).  It is non-trivial if and only if $\fPbar$ has type $\leq \mu$. 

We begin with recalling the construction of representable formally smooth covers of both $D^{\leq \mu}_{\fPbar}$ and $D^{[a,b]}_{\fPbar}$ from \cite[Proposition 3.1.1]{levin15}.
 Fix $N > \frac{b-a}{p-1}$ and a trivialization $\betabar$ of $\fPbar$ mod $E(u)^N$.  
\begin{defn} \label{defn:deformationproblems1}
 Define  
  $D^{\leq \mu , \beta}_{\fPbar}(A)$ to be the category of  pairs $(\fP, \beta)$ where
  $D^{\leq \mu}_{\fPbar}(A)$  and $\beta$ is a trivialization of $\fP \mod E(u)^N$ lifting $\betabar$.  (It is denoted 
   $\widetilde{D}_{\fPbar}^{(N), \mu}$ in \emph{loc. cit..})
 \end{defn}
 
$D^{\leq \mu, \beta}_{\fPbar}$ is representable on $\widehat{\cC}_{\Lambda}$ by a complete local Noetherian ring which we denote $R_{\fPbar}^{\leq  \mu,\beta}$.  
Now fix a continuous Galois representation $\rhobar: \Gamma_K \to G(\F)$.
We will define some additional deformation functors assuming there exists a $G$-Kisin module $\fPbar$ with coefficients in $\F$ with 
an isomorphism $\overline{\gamma} : \tT_{G,\F}(\fPbar) \simeq \rhobar|_{\Gamma_\infty}$. 
(This assumption is natural for our purposes as this is a necessary condition for $\rhobar$ to admit a crystalline lift.) 
Let $D_{\rhobar}^{\mu,\square}$ be the deformation functor associated to the deformation ring $R_\rhobar^{\mu,\square}$ for crystalline representations of $p$-adic Hodge type $\mu$ from Fact~\ref{fact:crystallinedeformation}.

\begin{defn} \label{defn:deformationproblems2}  For fixed $\rhobar$, $\fPbar$, and $\overline{\gamma}$, we define the following deformation problems on $\cC_{\Lambda}$:
\begin{enumerate}[(i)]
\item  define  $D_{\rhobar, \fPbar}^{\mu,\square}$  by
\[
D_{\rhobar, \fPbar}^{\mu,\square} (A) = \left \{ (\fP, \rho, \delta)\, | \, \fP \in D_{\fPbar}^{\leq \mu}(A), \, \rho \in D^{\mu,\square}_{\rhobar}(A), \, \rho|_{\Gamma_{\infty}} \overset{\delta} \simeq \tT_{G,A}(\fP) \right\}
\]
 where $\delta$ lifts $\overline{\gamma}$;

\item define $D_{\rhobar, \fPbar}^{\mu,\beta, \square}$ by
\[
D_{\rhobar, \fPbar}^{\mu,\beta,\square} (A) = \left \{ (\fP, \rho, \delta,\beta)\, | \, (\fP,\rho,\delta) \in D_{\rhobar, \fPbar}^{\mu,\square}(A), \, (\fP,\beta) \in D_{\fPbar}^{\leq \mu, \beta} (A) \right \}
\]
where $\beta$ lifts $\overline{\beta}$;

\item  define $D_{\fPbar}^{\leq \mu,\beta,\square}$ by
\[
D_{\fPbar}^{\leq \mu,\beta,\square}(A) = \left \{ (\fP, \beta , \alpha ) : (\fP, \beta) \in D_{\fPbar}^{\leq \mu,\beta}(A),\, \alpha \text{ trivializes } \tT_{G,A}(\fP) \right \}.
\]
\end{enumerate}
\end{defn}

There are natural forgetful maps between these deformation problems: 
\begin{equation} \label{eq:deformationdiagram}
\begin{tikzcd}
 D_{\rhobar,\fPbar}^{\mu,\beta,\square} \ar[r]\ar[d] & D_{\fPbar}^{\leq \mu, \beta,\square} \ar[r] & D_{\fPbar}^{\leq \mu, \beta} \\
 D_{\rhobar}^{\mu,\square}
 \end{tikzcd}
\end{equation}

\begin{remark}
If $A$ is the ring of integers in a finite extension of $\Lambda$, or if $A$ is a $L$-algebra, then the Galois representation $\rho$ in the above definitions has $p$-adic Hodge type $\mu$.  However, in general (for example, if $A$ is torsion) we cannot speak of the Galois representation or the $G$-Kisin module having type $\mu$.  We can only require that the $G$-Kisin module have type $\leq \mu$ or that the Galois representation be an $A$-point of $R^{\mu,\square}_{\rhobar}$.  We reiterate that we require $ p \nmid \# \pi_1(G^{\der})$ in order to use Theorem~\ref{thm:affineschubert} to define $Y^{\leq \mu}$ and study $G$-Kisin modules of type $\leq \mu$.
\end{remark}

Because of the presence of framings or trivializations, these deformation problems are easily seen to be pro-representable by complete local Noetherian rings which we denote by
$R_{\rhobar, \fPbar}^{\mu,\square}$, by $R_{\rhobar, \fPbar}^{\mu,\beta,\square}$, and by
$R_{\fPbar}^{\leq  \mu,\beta,\square}$
respectively.

\begin{lem} \label{lem:flatmoduli}
The rings $R^{\leq \mu, \beta}_{\fPbar}$  and  $R^{\leq \mu, \beta, \square}_{\fPbar}$ are $\Lambda$-flat with reduced generic fiber.
\end{lem}

\begin{proof}
This follows from \cite[Proposition 3.3.6]{levin15}, the local model diagram \cite[eq. (3-3-9-2)]{levin15}, and Theorem~\ref{thm:affineschubert}.
\end{proof}

In contrast, it is not clear that $R_{\rhobar, \fPbar}^{\mu,\square}$ and $R_{\rhobar, \fPbar}^{\mu,\beta,\square}$ are $\Lambda$-flat.

\begin{defn} \label{defn:flatclosures}
Define $R_{\rhobar, \fPbar}^{\mu,\square,\pflat}$ (resp. $R_{\rhobar, \fPbar}^{\mu,\beta,\square,\pflat}$) to be the flat closure of $R_{\rhobar, \fPbar}^{\mu,\square}$ (resp. $R_{\rhobar, \fPbar}^{\mu,\beta,\square}$).
\end{defn}

We also analyze $D^{\leq \mu, \beta}_{\fPbar}$ a bit more in characteristic $p$.  
We work over $\F$-algebras, where $E(u)^n = u^n$. Define  $G'_{(n)} := \Res_{(\fS_{\F}/u^n)/ \F} G_{\F}$, which is a smooth affine group scheme as it is the Weil--restriction of $G_{\F}$ along the finite flat map $(k[u]/u^n \tensor{\F_p} \F)/\F$.  There is a natural map $\LpG'_{\F} \rightarrow G'_{(n)}$ given by reduction mod $u^n$.   Define a group scheme $L^{+, (n)} G'_{\F}$ over $\F$ to sit in the exact sequence
\[
1 \rightarrow L^{+, (n)} G'_{\F} \rightarrow \LpG'_{\F} \rightarrow G'_{(n)} \rightarrow 1.  
\]

The following lemma says that the $\varphi$-conjugation action of $L^{+, (N)} G'$ can be ``straightened,'' where $N$ is a previously-fixed integer satisfying $N > \frac{b-a}{p-1}$.

\begin{lem} \label{lem:straighten} Let $A$ be an $\F$-algebra, and consider $C_1, C_2 \in G(\fS_A)$ such that $\rho_0(C_2)$ has height in $[a,b]$.   If $C_1 = g C_2$ for $ g\in L^{+, (N)} G'_{\F}(A)$ then there exists $g' \in   L^{+, (N)} G'_{\F}(A)$ such that $C_1 = g' C_2 \varphi(g')^{-1}$. 
\end{lem}
\begin{proof}
We are trying to solve the equation 
\[
gC_2 \varphi(g') C_2^{-1} = g'.   
\]
This can be solved by the usual successive approximation argument taking $g_0' = g$ and defining $g'_i = g C_2 \varphi(g'_{i-1}) C_2^{-1}$. Using that if $g'_{i-1} \in L^{+, (n)} G'_{\F}(A)$ for $n \geq N$ then $C_2 \varphi(g'_{i-1}) C_2^{-1} \in  L^{+, (n+1)} G'_{\F}(A)$, the $g'_i$ converge $u$-adically to the desired element $g'$.  
\end{proof}

\begin{prop}  \label{prop:straightening}  Assume $\fPbar$ has type $\leq \mu$ and so $x =\Psi(\fPbar, \betabar) \in \Gr^{\leq \mu}_{G'_{\F}}(\F)$.    Then $\Psi$ induces a formally smooth morphism 
\[
\Psi^{\mu}_{\F}:(D^{\leq \mu, \beta}_{\fPbar})_{\F} \rightarrow \Spf \cO^{\wedge}_{\Gr^{\leq \mu}_{G'_{\F}}, x}
\]
 of relative dimension $\dim G'_{(N)}$. 
\end{prop}

\begin{proof}  Let $A \in \cC_{\Lambda}$ be killed by $p$ and let $I$ be a nilpotent ideal.  Given $(\fP_{A/I}, \beta_{A/I}) \in D^{\leq \mu, \beta}_{\fPbar}(A/I)$, choose a trivialization $\widetilde{\beta}_{A/I}$ of $\fP_{A/I}$ over $\fS_{A/I}$ lifting $\beta_{A/I}$.  The Frobenius trivializes to an element $C_{A/I} \in G(\fS_{A/I}[1/u])$.   Notice that $\Psi(\fP_{A/I}, \beta_{A/I}) = x = [C_{A/I}]$ in $\Gr_{G'_{\F}}(A/I)$.    

Assume we have a point $\widetilde{x} \in \Gr^{\leq \mu}_{G'_{\F}}(A)$ lifting $x=[C_{A/I}]$.  Since $\LpG'$ is formally smooth as functor on $\Lambda$-algebras, we can choose a representative $C_A$ for the class $\widetilde{x}$ such that $C_A \mod I = C_{A/I}$. 
We can construct a $G$--Kisin module $\fP_A$ with coefficients in $A$ lifting $\fP_{A/I}$ equipped with a trivialization $\widetilde{\beta}_A$ lifting $\widetilde{\beta}_{A/I}$ and with Frobenius given by $C_A$.
If we take $\beta_A = \widetilde{\beta}_A \mod E(u)^N$, it is clear that $(\fP_A, \beta_A)$ deforms $(\fP_{A/I}, \beta_{A/I})$ and it has type $\leq \mu$ since $\widetilde{x} \in \Gr^{\leq \mu}_{G'_{\F}}(A)$.  

The fiber of $\Psi^{\mu}_{\F}$ over $[C_A]$ can be identified with
\[
  \{  (\LpG'(A))_e C_A \}/ (L^{+, (N)} G'_{\F}(A)_e, \varphi) 
\]
where the action is by $\varphi$-conjugation and the subscript $e$ indicates that we require elements to be the identity modulo the maximal ideal of $A$. By Lemma \ref{lem:straighten}, we have
\[
  \{  (\LpG'(A))_e C_A \}/ (L^{+, (N)} G'_{\F}(A)_e, \varphi)  =  (L^{+, (N)} G'_{\F}(A))_e \backslash \{  (\LpG'(A))_e C_A \} \}
\]
and hence the fiber is a torsor for $(G'_{(N)}(A))_e$.  This proves the dimension formula.    
\end{proof}

Finally, we recall a resolution of $\Spec R^{\mu, \square}_\rhobar$ introduced by Kisin for $\GL_n$ and constructed in \cite{levin15} for $G$--valued representations.  

\begin{prop} \label{prop:resolution}  The Kisin resolution $X^{\mu}_{\rhobar}$ is a projective $R^{\mu, \square}_{\rhobar}$-scheme $($flat over $\Lambda)$ such that:
\begin{enumerate}[(i)]
\item  \label{resolutionI}
For any $x \in X^{\mu}_{\rhobar}(\F')$, let $\widehat{\cO}_x^{\mu}$ denote the complete local ring of  $X^{\mu}_{\rhobar}$ at $x$.  There is a corresponding $G$-Kisin lattice $\fP$ and a closed immersion 
\[
\Spf \widehat{\cO}_x^{\mu} \rightarrow D_{\rhobar_{\F'}, \fP}^{\mu,\square}
\]   
which is an isomorphism modulo $p$-power torsion.  

\item \label{resolutionII}
  Let $\fm$ denote maximal ideal of $R^{\mu, \square}_{\rhobar}$,
$\rho$ denote the universal deformation,  
   and $X^{\mu}_{\rhobar, \fm}$ denote the fiber over the closed point of $\Spec R^{\mu,\square}_\rhobar$.  For any Artinian local $\F$-algebra $A$, 
\[
X^{\mu}_{\rhobar, \fm}(A) \subset \{ \fP  \mid \fP  \text{ is a }G\text{-Kisin lattice in } M_{G,\F}(\rho|_{\Gamma_\infty}) \otimes_{\cO_{\cE, \F}} \cO_{\cE, A} \text{ with type } \leq \mu \}. 
\]

\item  \label{resolutionIII}
If $\Theta:X^{\mu}_{\rhobar} \rightarrow \Spec R^{\mu, \square}_{\rhobar}$ is the structure map then $\Theta [\frac{1}{p}]$ is an isomorphism.

\end{enumerate}
\end{prop}

\begin{proof}
Parts (\ref{resolutionI}) and (\ref{resolutionIII}) follow from Corollary 3.3.15 of \cite{levin15}, and (\ref{resolutionII}) is an immediate consequence of (\ref{resolutionI}).  The only difference with \emph{loc. cit.} is rather than all weights $\leq \mu$, we restrict to those with exactly weight $\mu$.   In particular, $X^{\mu}_{\rhobar}$ is the closure of $\Spec  R^{\mu, \square}_{\rhobar}[1/p]$ in  $X^{\mathrm{cris}, \leq \mu}_{\rhobar}$ (see Definition 3.3.14 in \emph{loc. cit.}).  
\end{proof}
 
\section{The Monodromy Condition} \label{sec:monodromy}

As before, we let $\Lambda$
be the ring of integers in a $p$-adic field $L$ with residue field $\F$, and $G$ be a split reductive group defined over $\Lambda$.   We continue to fix a $p$-adic field $K$ unramified over $\Q_p$ with residue field $k$ and ring of integers $W(k)$. 
Throughout this section, we assume that $p \nmid \# \pi_1(G^{\der})$.

In this section, we study the difference between $\Spf R_{\fPbar,\rhobar}^{\leq \mu,\beta,\square}$ and $\Spf R_{\fPbar}^{\leq \mu,\beta,\square}$.  The difficulty is identifying which Kisin modules give rise to crystalline representations as opposed to just $\Gamma_{\infty}$-representations.  To prove our main theorem, it suffices to obtain a bound on the special fiber.   This is accomplished by a monodromy condition introduced in \cite{kisin09}, which we recall and adapt to $G$-Kisin modules in \S\ref{ss:monodromy}.  In \S\ref{ss:approxmonodromy}, we follow the strategy introduced in \cite{lllm18} to find a $p$-adic approximation for the monodromy condition and understand its reduction modulo $p$.

\subsection{The Monodromy Condition} \label{ss:monodromy}

Let $\cO^\rig$ denote the ring of rigid analytic functions on the open unit disc over $K$, and fix an embedding $\cO^\rig \into K\pseries{u}$.  Note that $\fS[1/p]$ is identified with the subring of bounded functions on that disc, that $\cO^\rig$ consists of power series $\sum a_n u^n$ such that $\lim_{n \to \infty} |a_n |_p r^n = 0$ for any $r<1$, and that the Frobenius of $\fS$ extends to $\cO^\rig$. 
For a Kisin module $\fM$ with coefficients in a finite flat $\Lambda$-algebra $A$, we define $\cO^\rig_A := \cO^\rig \tensor{\Z_p} A$ and $\fM^\rig := \fM \tensor{\fS} \cO^\rig$.  For a $G$-Kisin module $\fP$ with coefficients in $A$, we define $\fP^\rig := \fP \times_{\Spec \fS} \Spec(\cO^{\rig})$.

Define $\lambda \in \cO^\rig$ by 
\begin{equation} \label{defn:derivation}
\lambda := \prod_{n=0}^\infty \varphi^n \left( \frac{E(u)}{-p} \right) = \prod_{n=0}^\infty \varphi^n \left( 1- \frac{u}{p} \right).
\end{equation}
We define a derivation $N_\nabla$ on $\cO^\rig$ by $N_{\nabla} = - u \lambda \frac{d}{du}$.

For a $\Lambda$-algebra $A$ equipped with a $\Lambda$-linear derivation $N_A$, remember that a derivation over $N_A$ on an $A$-module $M$ is a function $N_M : M \to M$ such that 
$N_M(v+w) = N_M(v) + N_M(w)$ and $N_M(cv) = N_A(c) v + c N_M(v)$ for $v,w \in M$ and $c \in A$.

\begin{example}
Let $M$ be a $\Lambda$-module.
There is a trivial derivation $N^\triv_M$ on $M \otimes_{\Zp} \cO^\rig $ over $N_{\nabla}$ given by $1 \otimes N_{\nabla}$. 
\end{example}

\begin{remark} \label{rmk:matrixforderivation}
Given a basis $\beta = \{v_1,\ldots , v_n\}$ for a finitely generated free $A$-module $M$, a derivation $N_M$ on $M$ can be represented by a matrix $[N_M]_\beta$ whose $i$th column is the coefficients of $N_M(v_i)$ written in the basis $\{v_1,\ldots, v_n\}$.  For example, the matrix for $N^\triv_M$ is the zero matrix.
If $[T]_\beta$ is the matrix for a homomorphism $T : M \to M$, then
\[
[T \circ N_M]_\beta = [T]_\beta \cdot [N_M]_\beta \quad \text{ and } \quad [N_M \circ T]_\beta = [N_M]_\beta \cdot [T]_\beta + N_\nabla([T]_\beta).
\]
\end{remark}

\begin{fact} \label{fact:monodromy}
Let $A$ be a finite flat $\Lambda$-algebra $A$ and let $\fM$ be a Kisin module of bounded height.  There is a unique derivation $N_{\fM}$ on $\fM^\rig [1/\lambda]$ over $N_\nabla$ such that $N _{\fM} \equiv 0 \mod{u}$
 and such that as endomorphisms of $\fM^\rig[1/\lambda]$
 we have
\begin{equation} \label{eq:derivationcompatibility}
N_{\fM} \phi_{\fM} = E(u) \phi_{\fM} \varphi^*(N_{\fM}).
\end{equation}
The module $\fM^\rig$ is preserved by $N_{\fM}$ if and only if 
$(\widetilde{T}_{\GL_n,A} (\fM))[1/p]$ is the restriction to $\Gamma_\infty$ of a crystalline $\Gamma_K$-representation.
\end{fact} 

Except for uniqueness, this is essentially \cite[Lemma 1.3.10, Corollary 1.3.15]{kisin06}.  The argument for uniqueness is routine, and is spelled out in \cite[Lemma 2.2.1]{bl}.  The condition that $N_\fM$ preserves $\fM^\rig$ is often referred to as the \emph{monodromy condition}.

We now generalize this to apply to $G$-Kisin modules. 
 We begin by defining a notion of derivations on $G$-bundles.
Let $B$ be a $\Lambda$-algebra equipped with a $\Lambda$-linear derivation $N_B$.  

\begin{defn} \label{defn:Gdiff}
Let $\fP$ be a $G$-bundle over $B$ corresponding to a fiber functor $\eta: \fRep_\Lambda(G) \to \Proj_{B}$.
  A derivation on $\fP$ over $N_B$ is
 the data, for every $G \to \GL(V)$, of a derivation $N_V$ on $\eta(V)$ over $N_B$ such that
\begin{enumerate}[(i)]  
\item for a short exact sequence $0 \to V_1 \to V \to V_2$, the operators $N_V$, $N_{V_1}$, and $N_{V_2}$ are compatible;
\item  $N_{V_1 \otimes V_2} = 1 \otimes N_{V_2} + N_{V_1} \otimes 1$.
\end{enumerate}
\end{defn}

\begin{lem} \label{lem:Gdiff}
Let $\fP$ be a $G$-bundle over $B$ corresponding to a fiber functor $\eta: \fRep_\Lambda(G) \to \Proj_{B}$.  The set of derivations on $\fP$ over $N_B$ are a $\Lie G$-torsor.  Fixing a trivialization of $\fP$, the torsor is trivialized by the derivation $N^{\triv}$ given by $N^\triv_V = 1 \otimes N_{\nabla}$ on $V \tensor{\Lambda} B$ for each $V \in \fRep_\Lambda(G)$.
\end{lem}

\begin{proof}
Let $N$ and $N'$ be derivations on $\fP$, and $X \in \g_{B}$.  For representations $V \in \fRep_\Lambda(G)$, there are associated $N_V, X_V : \eta(V) \to \eta(V)$ as in Definition~\ref{defn:Gdiff} and Example~\ref{ex:tannakian}\eqref{endomorphismtannakian}.
 A straightforward computation shows that $N + X$ (defined for $V \in \fRep_\Lambda(G)$ by $N_V + X_V$) is a derivation, and that $N - N'$ (defined for $V \in \fRep_\Lambda(G)$ by $N_{V} - N'_{V}$) is an element of $\g_{B}$ again using Example~\ref{ex:tannakian}\eqref{endomorphismtannakian}.  This establishes the first statement.
  Yet another straightforward check shows that $N^\triv$ is a derivation, establishing the second.
\end{proof}

\begin{remark} \label{rmk:glntrivialization}
When $G = \GL(V)$ for a free $B$-module $V$ of rank $n$, after picking a basis for $V$ the Lie algebra $(\Lie G)_B$ is identified with $\Mat_n(B)$.  Unwinding definitions, the matrix associated to a derivation in Remark~\ref{rmk:matrixforderivation} agrees with the element of $(\Lie G)_B$ obtained by trivializing the torsor of derivations by $N^\triv$.
\end{remark}

By working with each $V \in \fRep_\Lambda(G)$ individually, we 
may compose derivations with (semi-linear) automorphisms.  Of course, the result is only a function $V_B \to V_B$, and will not necessarily be a derivation.  We may likewise define multiplication by scalars and base change.

\begin{prop} \label{prop:gderivation}
Let $\fP$ be a $G$-Kisin module with coefficients in a finite flat $\Lambda$-algebra $A$.  
There is a unique derivation $N_{\fP}$ on $\fP^\rig[1/\lambda]$  over $N_\nabla$ such that $N_{\fP}  \equiv 0 \pmod{u}$ and 
\begin{equation} \label{eq:gderivationcompatibility}
N_\fP \phi_{\fP} = E(u) \phi_{\fP} \varphi^*(N_\fP).
\end{equation}
\end{prop}

\begin{proof}
By Example~\ref{gkisintannnakian}, each $V \in \fRep_\Lambda(G)$ 
induces a Kisin module $\fP(V)$.
Using Fact~\ref{fact:monodromy}, we obtain a unique derivation $N_{\fP,V}$ compatible with $\phi_V$ in the sense it satisfies \eqref{eq:derivationcompatibility}.  Compatibility of the $N_{\fP,V}$ with exact sequences and tensor products follows from uniqueness, so the $N_{\fP,V}$ collectively define a necessarily unique derivation $N_{\fP}$ on $\fP^\rig[1/\lambda]$ satisfying \eqref{eq:gderivationcompatibility}.
\end{proof}

\begin{cor} \label{cor:gderivationtriv}
Let $\fP$ be a $G$-Kisin module with coefficients in $A$ as in Proposition~\ref{prop:gderivation}.
Fix a trivialization of $\fP$ and trivialize the torsor of derivations on $\fP^\rig[1/\lambda]$ by $N^\triv_{\fP}$.  
Letting $\phi_{\fP}$ trivialize to $C \in G(\cO^\rig_A [1/\lambda])$ and $E(u)\phi_{\fP} \varphi^*( N^\triv_{\fP}) \phi_{\fP}^{-1}$ trivialize to $N_1 \in \g_{\cO^\rig_A [1/\lambda]}$, we have that $N_\fP$ trivializes to an $N_\infty \in \g_{\cO^\rig_A [1/\lambda]}$
such that
\begin{equation} \label{eq:trivializedderivation}
N_\infty = E(u) \Ad_G(C) (\varphi(N_\infty)) + N_1.
\end{equation}
\end{cor}

\begin{proof}
It is straightforward to verify that $\phi_{\fP} \varphi^*(N^\triv_{\fP}) \phi^{-1}_{\fP}$ is a derivation on $\varphi^*(\fP) ^{\rig}[1/\lambda]$.
We may rewrite \eqref{eq:gderivationcompatibility} as
\[
N_\fP = E(u) \phi_\fP \varphi^*(N_\fP) \phi_{\fP}^{-1}
\]
again with the equality interpreted as equality of functions on $V$ for each $V \in \fRep_\lambda(G)$.   The left side is a derivation which trivializes to $N_\infty$.  On the right, $E(u) \phi_\fP \varphi^*(N_\fP) \phi_{\fP}^{-1}$ trivializes to 
\begin{align*}
 E(u) \phi_\fP \varphi^*(N_\fP) \phi_{\fP}^{-1}  - N_\fP^\triv &= E(u) \phi_\fP \varphi^*(N_\fP) \phi_{\fP}^{-1} - E(u) \phi_\fP \varphi^*(N_\fP^\triv) \phi_\fP^{-1} 
\\
& \phantom{=} + E(u) \phi_\fP \varphi^*(N_\fP^\triv) \phi_\fP^{-1} - N_\fP^\triv.
\end{align*}
Using Example~\ref{ex:tannakian}\eqref{endomorphismtannakian}, the endomorphism $E(u) \phi_\fP \varphi^*(N_\fP) \phi_{\fP}^{-1} - E(u) \phi_\fP \varphi^*(N_\fP^\triv) \phi_\fP^{-1}$ corresponds to $E(u) \Ad_G(C) (\varphi(N_\infty)) \in \g_{\cO^\rig_A [1/\lambda]}$.  By definition, $E(u) \phi_\fP N_\fP^\triv \phi_\fP^{-1} - N_\fP^\triv$ corresponds to $N_1$.  
\end{proof}

\begin{prop} \label{prop:Gmonodromy}
Let $\fP$ be a $G$-Kisin module 
 with coefficients in a finite flat $\Lambda$-algebra $A$, and $N_\fP$ be the derivation on $\fP^\rig[1/\lambda]$.  Then $\tT_{G,A}(\fP)[1/p]$ is the restriction to $\Gamma_\infty$ of a crystalline $G$-valued representation of $\Gamma_K$ if and only if $N_{\fP,V}( V \otimes \cO^{\rig}_A) \subset V \otimes \cO^{\rig}_A$ for every $V \in \fRep_\Lambda(G)$. 
\end{prop}

\begin{proof}
This follows from the second part of Fact~\ref{fact:monodromy} and the fact that a $G$-valued representation is crystalline if and only if the representation on $V$ is crystalline for every $V \in \fRep_\Lambda(G)$. 
\end{proof}

As usual, it suffices to check that $N_{\fP,V} ( \fP_V) \subset \fP_V$ for a single faithful representation $V \in \fRep_\Lambda(G)$.  It also suffices to check, using the equivalence of Example~\ref{ex:tannakian}\eqref{endomorphismtannakian}, that the $N_\infty \in \g_{\cO_A^\rig[1/\lambda]}$ corresponding to $N_\fP$ actually lies in $\g_{\cO^\rig_A}$.

\begin{defn} \label{defn:mono}
Let $\fP$ be a $G$-Kisin module 
with coefficients in a finite flat $\Lambda$-algebra $A$.  Let $N_\fP$ be the natural derivation on $\fP_{\cO^\rig[1/\lambda]}$ which trivializes to $N_\infty \in \g_{\cO^\rig_A[1/\lambda]}$.  We say $\fP$ (or $N_\infty$) satisfies the \emph{monodromy condition} if $N_\infty \in \g_{\cO^\rig_A}$.
\end{defn}

We immediately obtain:

\begin{cor} \label{cor:Gmonodromy}
With the notation of Proposition~\ref{prop:Gmonodromy}, $\tT_{G,A}(\fP)[1/p]$ is the restriction to $\Gamma_\infty$ of a crystalline $G$-valued representation of $\Gamma_K$ if and only if $\fP$ satisfies the monodromy condition.
\end{cor}

\begin{remark} \label{rmk:explicitn1}
When $G = \GL(V)$, after fixing a basis $\beta$ for $V$ the trivial derivation corresponds to the zero matrix.  
Furthermore, we claim that 
the derivation $E(u) \phi_{\fP} \varphi^*(N^\triv_{\fP}) \phi^{-1}_{\fP}$ trivializes to $N_{\nabla}(C) C^{-1} = - u \lambda \frac{dC}{du} C^{-1}$ where $C = [\phi_\fP]_\beta$ is the matrix of $\phi_\fP$ with respect to $\beta$.

Since $\phi_\fP$ is semi-linear, notice that $C \varphi([\phi_\fP^{-1}]_\beta) = 1$, and since $N_\nabla$ is a derivation notice that 
\[
 0 = N_\nabla(\id_V) = N_\nabla(C C^{-1}) = C  N_\nabla( C^{-1}) + N_\nabla(C) C^{-1}.
\]
Using that $-E(u) \varphi \circ N_\nabla =  N_\nabla \circ \varphi $, which can be checked on power series, we  compute that 
\begin{align*}
 E(u) [\phi_{\fP} \varphi^*(N^\triv_{\fP}) \phi^{-1}_{\fP}]_\beta & = E(u) C \cdot \varphi \left( [N^\triv_{\fP}]_\beta [\phi^{-1}_{\fP}]_\beta  + N_\nabla( [\phi^{-1}_{\fP}]_\beta ) \right)\\
 & = E(u) C  \varphi( N_\nabla( [\phi^{-1}_{\fP} ]_\beta ))  = -C N_\nabla(\varphi([\phi_\fP^{-1}]_\beta)) \\
 &=  -C N_\nabla(C^{-1}) = N_\nabla(C) C^{-1}.
\end{align*}
This gives the claim, and also an explicit description of $N_1$.
Using this, \eqref{eq:trivializedderivation} is equivalent to
\[
[N_\infty]_\beta = E(u)  C \varphi([N_\infty]_\beta) C^{-1} + N_\nabla(C) C^{-1}.
\]

For general $G$, if we trivialize  $\fP$, and hence for every $V \in \fRep_\Lambda(G)$ obtain a basis for $V$, we obtain a derivation and semi-linear automorphism of $V$.  These are represented by matrices $N_{\infty,V}$ and $C_V$  such that 
\[
N_{\infty,V} = E(u) C_V \varphi(N_{\infty,V}) C_V^{-1} + N_\nabla(C_V) C_V^{-1}.
\]
We see that $N_{1,V} = - u \lambda \frac{d C_V}{du} C_V^{-1}$.
\end{remark}

In light of Remark~\ref{rmk:explicitn1}, we introduce some convenient notation for later.

\begin{defn} \label{def:dc}
For any $\Lambda$-algebra $A$ and $C \in \LG'(A)$,
  let
$\frac{dC}{du} C^{-1} \in \Lie \LG'(A)$ correspond to the endomorphism given by $\frac{ d C_V}{du} C_V^{-1}$ for $V \in \fRep_\Lambda(G')$ under 
the equivalence of Example~\ref{ex:tannakian}(\ref{endomorphismtannakian}). 
\end{defn}

It is straightforward to verify that the $\frac{ d C_V}{du} C_V^{-1}$ for $V \in \fRep_\Lambda(G')$ are compatible with exact sequences and tensor products, so the definition is valid.
By Remark~\ref{rmk:explicitn1}, for $C = [\phi_\fP]_\beta$ we have $$- u \lambda \frac{d C}{du } C^{-1} =  N_1.$$

\begin{remark} \label{remark:leibniz}
The individual terms of $\frac{d C}{du } C^{-1}$ have no intrinsic meaning, but overall expression behaves as the Leibniz rule would formally predict:
\[
\frac{d (C C')}{du } (C C')^{-1} = \frac{dC}{du} C^{-1} + \Ad_G(C) \left( \frac{d C'}{du} (C')^{-1} \right).
\]
This can be checked on $V \in \fRep_\Lambda(G)$.
\end{remark}

\subsection{Approximating the Monodromy Condition} \label{ss:approxmonodromy}

We now want to study the monodromy condition in families and $p$-adically approximate it.   We will apply this to families of deformations like the one over $R^{\leq \mu, \beta, \square}_{\fPbar}$.   We follow the setup in \cite[\S7.1]{localmodels}.   

Let $R$ be a flat local Noetherian $\Lambda$-algebra with finite residue field that is $p$-adically complete.   Define
\[
\cO^{\rig}_R := \varprojlim_n (W(k) \otimes_{\Zp} R)[\![u, \frac{u^n}{p}]\!] [1/p]
\]
which is naturally a subring of $(K \otimes_{\Zp} R)[\![u]\!]$.    Clearly, we have an inclusion $\cO^{\rig} \subset \cO^{\rig}_R$ and $\fS_R \rightarrow \cO^{\rig}_R$.    The operators $\varphi$ and $N_{\nabla}$ extend to $\cO^{\rig}_R$ with trivial action on the coefficients.   Note that for any $f \in \cO^{\rig}_R$, we can evaluate $f$ at $u = p$ to get an element of $K \otimes_{\Zp} R$ which we denote by $f|_{u = p}$.  

Fix a $G$-Kisin module $\fP$ with coefficients in $R$, a trivialization of $\fP$, and a type $\mu$ for $G$.   Assume that $\fP$ has type $\leq \mu$, and as usual set $h_{\mu} = \max_{\alpha \in \Phi_{ G'}}  \langle \mu, \alpha \rangle$.    We will frequently use Lemma~\ref{lem:adjheight}, that since $\fP$ has type $\leq \mu$ then $\fP(\Lie G)$ has height in $[-h_{\mu}, h_{\mu}]$.
 
The same argument as in Proposition~\ref{prop:gderivation} using  \cite[Proposition 7.1.3(1)]{localmodels} shows that $\fP^{\rig}[1/\lambda]$ (defined at the beginning of \S\ref{ss:monodromy}) admits a unique derivation $N_{\fP}$ such that $N_{\fP} \equiv 0 \mod u$.  Note that $N_{\fP}$ interpolates the connections at all the finite flat $\Lambda$-algebra points of $R$.  

 Let $C \in G(\fS_R[1/E(u)])$ correspond to $\phi_\fP$.  Trivialize the torsor of derivations by $N_\fP^\triv$, and let $N_\infty$ and $N_1$ be as in Corollary~\ref{cor:gderivationtriv}; in $\g_{\cO^\rig_R[1/\lambda]}$, we have
\[
N_\infty = E(u) \Ad_G(C) (\varphi(N_\infty)) + N_1.
\]

\begin{defn} \label{defn:approxderivations}
Let $N_0 := 0$ and for $i > 0$, inductively define
\begin{equation} \label{eq:recurrence}
 N_{i+1} -N_i:=  E(u) \Ad_G(C) \left( \varphi(N_{i} - N_{i-1}) \right).
\end{equation}
Furthermore, set $L_1(C) := \frac{E(u)^{h_{\mu}}}{u \lambda} N_1 = - E(u)^{h_\mu} \frac{dC}{du} C^{-1}$.    
\end{defn}

\begin{lem} \label{lem:convergence}
The sequence $\{N_i\}$ converges to $N_\infty$ in $\g_{(K \otimes_{\Zp} R)\pseries{u}}$ in the $u$-adic topology.
\end{lem}

\begin{proof}
When $G = \GL_n$, this follows from \cite[Proposition 7.1.3 (2)]{localmodels}.  In general, one may check the result on each $V \in \fRep(G)$.
\end{proof}

\begin{prop} 
The $G$-Kisin module $\fP$ satisfies the monodromy condition if and only if $N_\infty$ has 
no poles at $u=p$.  
\end{prop}

\begin{proof}
When $G = \GL_n$, this is \cite[Proposition 2.2.4]{bl}.  The general case follows as both conditions may be checked on each $V \in \fRep(G)$.
\end{proof}

Thus we are reduced to studying the condition that $N_{\infty}$ has no poles at $u = p$.   We aim to show that the condition that $N_1$ has no poles at $u = p$ is a good approximation to this condition when $h_{\mu} < p-1$.  To make this precise, we have to look more carefully at the sequence $N_i$.

\begin{lem} \label{L1pole}   Assume $\fP$ has type $\leq \mu$ and that $R[1/p]$ is reduced.  Then 
\begin{enumerate}[(i)]
\item \label{L1polei} $L_1(C) \in \g_{\fS_R}$;
\item  \label{lieintegralityi}
 $E(u)^{h_\mu} Ad_G(C) (\g_{\fS_R} )\subset \g_{\fS_R} $ and   $E(u)^{h_\mu} Ad_G(C) (\g_{\cO^\rig_R} )\subset \g_{\cO^\rig_R} $; 
\item \label{L1poleiii}$\lambda^{h_{\mu}-1} N_{\infty} \in  \g_{\cO^\rig_R}$.
\end{enumerate}
\end{lem} 
\begin{proof}
As can be checked on any faithful representation, $L_1(C) \in \g_{\fS_{R}}[1/E(u)]$.   The condition that $L_1(C) \in \g_{\fS_R}$ is then a Zariski closed condition on $R$.   Since $R$ is $\Lambda$-flat and reduced, it suffices to check that condition on the $\overline{L}$-points of $R$.   If $x$ is any such point, let $C_x$ denote the base change of $x$ to $\overline{L}$.   

Since $\fP$ has type $\leq \mu$, there exists $\mu' \leq \mu$ such that $C_x = K_1 E(u)^{\mu'} K_2$ where $K_1, K_2 \in \LpG'(\overline{L})$.    Applying Remark~\ref{remark:leibniz},  we have
\[
L_1(C_x) = L_1(K_1) +   \Ad_G(K_1)  L_1(E(u)^{\mu'})  +  E(u)^{h_{\mu}} \Ad_G(K_1 E(u)^{\mu'}) (\frac{dK_2}{du} K_2^{-1}).
\]
We directly see that $L_1(K_1)$ and $E(u) \frac{ d E(u)^{\mu'}}{du} E(u)^{-\mu'}$ are in $ \g_{\fS_{\overline{L}}}$.  For the third term, notice that $\frac{dK_2}{du} K_2^{-1} \in   \g_{\fS_{\overline{L}}}$ and since 
$\mu' \leq \mu$, we know $E(u)^{h_{\mu}} \Ad_G(K_1 E(u)^{\mu'})$ preserves $\g_{\fS_{\overline{L}}}$.  Thus we conclude $L_1(C_x) \in \g_{\fS_{\overline{L}}}$.

Since $\fP$ has type $\leq \mu$, we see that $\fP(\Lie G)$ has height in $[-h_\mu,h_\mu]$ which gives part \eqref{lieintegralityi}.

For the last part, it suffices to show $\lambda^{h_{\mu}-1} N_{i} \in  \g_{\cO^\rig_R}$ which we prove by induction on $i$.   The base case follows from part (\ref{L1polei}).  Since 
$
N_{i+1} = N_i +  E(u) \Ad_G(C) \left( \varphi(N_{i} - N_{i-1}) \right),
$
it suffices to show that if $\lambda^{h_{\mu}-1} N_{i}  \in  \g_{\cO^\rig_R}$ then  $E(u) \lambda^{h_{\mu}-1}  \Ad_G(C) \left( \varphi(N_{i}) \right) \in  \g_{\cO^\rig_R}$.  This follows from part (\ref{lieintegralityi}) using  that
\[
E(u) \lambda^{h_{\mu}-1}  \Ad_G(C) \left( \varphi(N_{i}) \right) = E(u)^{h_{\mu}}/(-p)^{h_{\mu} -1} \Ad_G(C) \left( \varphi( \lambda^{h_{\mu}-1} N_{i}) \right).  \qedhere
\]  
\end{proof}

Since $\lambda$ has simple pole at $u = p$, the condition that $N_1$ has no poles at $u = p$ is equivalent to $L_1(C)$ having a zero of order at least $h_{\mu}-1$.  Similarly, by Lemma \ref{L1pole}(\ref{L1poleiii}), the monodromy condition is equivalent to $E(u)^{h_\mu-1}N_{\infty}$ having a zero of order at least $h_\mu-1$ at $u = p$.  

Given an element of element $X = (X_\sigma)_{\sigma \in \cJ} \in \g_{K \tensor{\Lambda} R} = \prod_{\sigma \in \cJ} \g_R$, we say the entries of $X$ are the coefficients of each $X_\sigma$ with respect to a fixed basis for $\g$ over $\Lambda$.

\begin{defn}
Let $\widetilde{I}_{N_{\infty}}$ denote the ideal in $R[1/p]$ generated by the entries of  $\frac{ d^i \left( E(u)^{h_\mu-1} N_{\infty} \right)}{du^i} |_{u=p}$ for $0 \leq i \leq h_\mu-2$.  Furthermore, define
\[
I_{N_{\infty}} := \widetilde{I}_{N_{\infty}} \cap R.
\]
so $\Spec R/I_{N_{\infty}}$ is closure of the locus on the generic fiber satisfying the monodromy condition.

Define $I_{N_1} \subset R$ to be the ideal generated by the entries of $ \frac{ d^i L_1(C)}{du^i} |_{u=p}$ for $0 \leq i \leq h_\mu-2$. 
\end{defn}

 Our goal is to $p$-adically approximate $I_{N_{\infty}}$ and relate it to $I_{N_1}$ by studying the condition on the pole of $N_1$.      
This will yield our main theorem: 

\begin{thm} \label{thm:monoapprox}  Let $\fP$ be a $G$-Kisin module with coefficients in a complete flat local Noetherian $\Lambda$-algebra $R$ with finite residue field.  Assume $\fP$ has type $\leq \mu$, that $h_{\mu} < p-1$, and that $R[1/p]$ is reduced.   Then 
\[
I_{N_1}  \subset (I_{N_{\infty}}, p).
\]
\end{thm}

We will prove this in \S\ref{ss:proofmonodromy}.  The key idea is to control the $p$-divisibility of terms in a series relating $N_\infty$ to $N_1$.

\begin{defn} \label{defn:modpmonodromy}
Let $\fP$ be a $G$-Kisin module with coefficients in a $\Lambda$-algebra $A$ of characteristic $p$.  Let $C$ correspond to $\varphi_\fP$ upon choice of trivialization.  We say that $\fP$ satisfies the mod-$p$ monodromy condition provided that $L_1(C)$ has a zero of order at least $h_\mu-1$ at $u=p$.
\end{defn}

It is straightforward to verify that this notion is independent of the choice of trivialization, and equivalent to the condition that $\frac{dC}{du} C^{-1} \in \frac{1}{u} \g_{\fS_A}$.

\begin{cor} \label{cor:modpmonodromy}
Assume $h_\mu < p-1$ and that $R[1/p]$ is reduced.  The base change $\fP \otimes_{\fS_R} \fS_{R/ (p,I_{N_\infty})}$ satisfies the mod-$p$ monodromy condition.
\end{cor}

\begin{proof}
Theorem~\ref{thm:monoapprox} shows $\Spf R/(I_{N_\infty},p) \into \Spf R/(I_{N_1},p)$.  
\end{proof}

\subsection{Proof of Theorem~\ref{thm:monoapprox}} \label{ss:proofmonodromy}
We continue the notation of the previous subsection and set $h := h_\mu$.
 Throughout we assume that $R[1/p]$ is reduced.  We begin with some preliminaries. 

\begin{lem} \label{lem:lieintegrality}  Assume that $\fP(\Lie G)$ has height in $[-h, h]$.  Then,  
\begin{enumerate}[(i)]
 \item  \label{lieintegralityii} $E(u)^h N_1 \in u \lambda \g_{\cO_R^\rig}$;

\item  \label{lieintegralityiii} Fix $j<p$ and $X \in \g_{\cO^\rig_R}$.  
If $\frac{d^i X}{du^i}  |_{u=p} \in p^r \g_{W(k) \tensor{\Z_p} R}$ for all $0 \leq i \leq j$ then for any $i \leq j$ $$\frac{ d^i \left( E(u)^h Ad_G(C) (X) \right)}{du^i} |_{u=p} \in p^r \g_{W(k) \tensor{\Z_p} R}.$$
\end{enumerate}
\end{lem}

\begin{proof}
The first statement follows from Lemma~\ref{L1pole}(\ref{L1polei}).

For the second statement, we can expand $X = \sum_{i \geq 0} X_i \frac{E(u)^i} {i!}$ with $X_i \in p^r \g_{W(k) \tensor{\Z_p} R}$ for $i \leq j$ and otherwise $X_i \in \g_{K \tensor{\Z_p} R}$.    By Lemma~\ref{L1pole}(\ref{lieintegralityi}),  we see that $E(u)^h Ad_G(C)(X_i) \in p^r \g_{\fS_R}$ for $i \leq j$.  For any $i \leq j$ we conclude that 
\[
\frac{ d^i \left( E(u)^h Ad_G(C) (X) \right)}{du^i} |_{u=p} \in p^r \g_{W(k) \tensor{\Z_p} R}. \qedhere
\]  
\end{proof}

\begin{defn} \label{defn:Cs}
We define $\cA_C : \g_{\cO^\rig_R} \to \g_{\cO^\rig_R}$ by $\cA_C(X) = E(u)^h \Ad_G(C) (\varphi(X))$.  
Furthermore, define $L_2(C) := \cA_C(L_1(C)) = E(u)^h \Ad_G(C) ( \varphi(L_1(C)))$.  
\end{defn}

Note that $\cA_C$ is well defined and that $L_2(C)$ lies in $\g_{\cO^\rig_R}$ by Lemma~\ref{L1pole}\eqref{lieintegralityi}.  

\begin{remark}
When $G=\GL_n$, using Remark~\ref{rmk:explicitn1} we see $L_1(C) = -E(u)^h \frac{dC}{du} C^{-1}$ and $L_2(C) = E(u)^h C \varphi( L_1(C)) C^{-1}$.
\end{remark}

Now we let
\begin{equation}  \label{defn:zs}
 z_0 := -\frac{ u \varphi(\lambda)}{p} \quad \text{and} \quad z_i := -\frac{u^{p^i} \varphi^{i+1}(\lambda)}{p \prod_{j=1}^i \varphi^j(E(u))^{h-1} }. 
\end{equation} 

\begin{lem} \label{lem:telescoping}
Letting $\cA_C^i$ denote the $i$-fold composition of $\cA_C$, we have that
 \begin{equation} \label{eq:differenceformula}
  E(u)^{h-1}(N_{i+1} - N_i) = z_i \cA_C^i(L_1(C)) .
 \end{equation} 
and that
\begin{equation} \label{eq:keyequation}
  E(u)^{h-1} N_\infty =  E(u)^{h-1} N_1 + \sum_{i=1}^\infty  E(u)^{h-1}(N_{i+1} - N_i) = z_0 L_1(C) + \sum_{i=1}^\infty z_i  \cA_C^i(L_1(C)).
\end{equation}
Furthermore we have that $\cA_C^i(L_1(C))  \in  \g_{\cO^\rig_R}$
\end{lem}

\begin{proof}
We directly see that $E(u)^{h-1} N_1 = z_0 L_1(C)$, and then the first equation follows from induction using \eqref{eq:recurrence}.
The second equation telescopes. The last claim follows from Lemma~\ref{L1pole}(\ref{L1polei}).
\end{proof}

We now collect together information about $z_0$ and the $z_i$.  Let $v_p$ denote the $p$-adic valuation on $\ZZ_p$ normalized so $v_p(p) = 1$. 

\begin{lem} \label{lem:coefficients}  Let $i \geq 1$.
\begin{enumerate}[(i)]
\item  \label{part:varphi} We have that $\varphi^i(\lambda)|_{u=p}$ is a unit in $\ZZ_p$, and that for $0 < n < p$
 \[
   \frac{d ^n \varphi^i(\lambda)}{du^n} |_{u=p} = -\frac{(p^i-1)!}{(p^i-n)!} p^{p^i+(i-1)-n} + O(p^{p^i+i-n}).
 \]
 
 \item \label{part:z0val} For $1 < n < p$, we have that
\[
v_p\left( z_0|_{u=p} \right) =0  \quad \text{and} \quad v_p \left( \frac{dz_0}{du}|_{u=p} \right) = -1 \quad \text{and} \quad v_p \left( \frac{d^{n}z_0}{du^n} |_{u=p} \right) = p-n.
\]

\item  \label{part:zi} For $0 < n < p-1$, we have that
$$v_p \left( z_i \right) \geq p^i - i(h-1)-1 \quad \text{and} \quad v_p \left( \frac{d^n z_i}{du^n} \right) \geq p^i + \min(i,p) - n - i (h-1) -1 .$$
\end{enumerate}
\end{lem}

\begin{proof}
\begin{enumerate}[(i)]
\item  This an elementary calculation using the product rule on
 \[
  \varphi^i(\lambda) = \prod_{j=i}^\infty \frac{u^{p^j}-p}{-p}. 
 \]
 
 \item  This will follow using part (\ref{part:varphi}).  The first claim is immediate.  The second follows from 
\[
 \frac{dz_0}{du}|_{u=p} = -\frac{d \varphi(\lambda)}{du} |_{u=p} - \frac{1}{p} \varphi(\lambda)|_{u=p}.
\]
For the third, note that
\[
 \frac{d^n z_0}{du^n}|_{u=p} = -\frac{1}{p} \frac{d^{n-1} \varphi(\lambda)}{du^{n-1}}|_{u=p} - \frac{ d^n \varphi(\lambda)}{du^n}|_{u=p} = \frac{(p-1)!}{(n-2)!} p^{p-n}  + \frac{(p-1)!}{(n-1)!} p^{p-n} + O(p^{p-n+1})
\]
and that $1 + (n-1)^{-1} \not \equiv 0 \mod{p}$ when $1<n < p$.

\item  An elementary analysis using the product rule shows that for $0 < n < p$, we have
\[
 v_p \left( \frac{d^n}{du^n} \left(  \prod_{j=1}^i  \varphi^j(E(u)/(-p))^{1-h} \right)|_{u=p} \right) \geq p-n .
\]
Furthermore, $ \prod_{j=1}^i  \varphi^j(E(u)/(-p))^{1-h}$ evaluated at $u=p$ is a $p$-adic unit.
Combining this with part~(\ref{part:varphi}), we see using the product rule that for $0< n < p$ the $n$th derivative of
\[
 z_i = (-p)^{-i (h-1)-1}  u^{p^i} \cdot  \varphi^{i+1}(\lambda) \cdot \prod_{j=1}^i  \varphi^j(E(u)/(-p))^{1-h} 
\]
when evaluated at $u=p$ will have valuation at least $p^i + \min(i,p) - n - i (h-1) -1$.   The statement about $z_i$ itself is elementary given the previous calculations. \qedhere
 \end{enumerate}
\end{proof}

We now begin to analyze the monodromy condition to obtain information about its reduction modulo $p$.  At first glance, it is not clear that the evaluation of derivatives of terms in \eqref{eq:keyequation} at $u=p$ even lie in $\g_{W(k) \tensor{\Lambda} R} $, let alone are multiples of $p$.  For conciseness, we write
$\g_R'$ instead of $\g_{W(k) \tensor{\Lambda} R}$ in the following arguments. 

\begin{lem} \label{lem:partialmonodromy}
Suppose that $h < p-1$. Then there exists $\alpha_0 \in \ZZ_p^\times$ such that
 \begin{equation} \label{eq:pm0}
  \alpha_0 L_1(C) |_{u=p} + \left( z_1 L_2(C) \right) |_{u=p} \in  (p^p,I_{N_\infty}) \g_R'
 \end{equation}
and for  $0 < m < h$ we have that
\begin{equation} \label{eq:pmn}
 \sum_{j=0}^m \binom{m}{j} \frac{ d^j z_0 }{du^j}|_{u=p} \frac{d^{m-j} L_1(C)}{du^{m-j}} |_{u=p} + \sum_{j=0}^m \binom{m}{j} \frac{d^j z_1}{du^j} |_{u=p} \frac{d^{m-j} L_2(C)}{du^{m-j}} |_{u=p} \in (p^p,I_{N_\infty}) \g_R'.
\end{equation}
\end{lem}

\begin{proof}
 We begin by evaluating \eqref{eq:keyequation} at $u=p$.  We see that
\begin{itemize}
\item   on the left side, the entries of  $E(u)^{h-1} N_\infty$ are in $I_{N_\infty}$ by definition;

\item  $v_p(z_i |_{u=p}) \geq p^2 - 2h +1 \geq p^2 - 2p +3 \geq p$ for $i \geq 2$ using Lemma~\ref{lem:coefficients}(\ref{part:zi});

\item  $\cA_C^i(L_1(C)) \in  \g_{\cO^\rig_R} $.
\end{itemize} 
Thus all of the terms with $i\geq 2$ on the right side are multiples of $p^p$.  
Since $\alpha_0 := z_0 |_{u=p}$ is in $\ZZ_p^\times$ by Lemma~\ref{lem:coefficients}(\ref{part:z0val}), equation \eqref{eq:pm0} follows.
 
The second statement follows a similar argument, beginning by evaluating the $m$th derivative of by \eqref{eq:keyequation} at $u=p$.   Again, the entries in the left side are in $I_{N_\infty}$.  We apply the product rule to the terms with $i \geq 2$: the $m$th deriviate of the $i$th term is
\[
\sum_{j=0}^m \binom{m}{j} \frac{d^j z_i}{du^j} |_{u=p} \frac{d^{m-j} \cA_C^i(L_1(C))}{du^{m-j}} |_{u=p}.
\]
By Lemma~\ref{lem:telescoping}, $\cA^i_C(L_1(C)) |_{u=p}$ and its derivatives lie in $\g_R'$.  Lemma~\ref{lem:coefficients}(\ref{part:zi}) 
gives that for $0 < j < p-1$
\[
v_p\left(\frac{d^j z_i}{du^j} |_{u=p} \right) \geq p^i + \min(i,p) - i (h-1) -j -1
\]
It is straightforward to see this is greater than or equal to $p$ when $i >2$ as the $p^i$ term dominates.  When $i=2$, the valuation is at least
\[
p^2 + 2 -2(h-1) -j -1 \geq p^2 +1 - 2(h-1) - j \geq p^2 +1 - 3(p-2).
\]
The right side is always at least $p$.
(For $j=0$, we already say that $v_p(z_i |_{u=p}) \geq p$.)  Thus all the terms on the right side of the $m$th derivative of \eqref{eq:keyequation} at $u=p$ with $i\geq 2$ are multiples of $p^p$.  
\end{proof}

Now expand $L_1(C) = \sum_{i=0}^\infty A_i E(u)^i$ with $A_i \in \g'_R$; there exists $Q \in \g'_{R \pseries{u}}$ 
such that
\begin{equation} \label{4.3.6}
 \varphi(L_1(C)) = A_0 + A_1( u^p -p) + A_2 (u^p-p)^2 + \ldots =  (A_0 - pA_1 + p^2 A_2 - \ldots ) + u^p Q(u).
\end{equation}

\begin{lem} \label{lem:x1}
For a fixed positive integer $n < p$, if for all $0 \leq i \leq n$ we have
$$\displaystyle \frac{ d^i L_1(C)}{du^i} |_{u=p} \in (p^{n-i},I_{N_\infty}) \g_R',$$
then it follows that
$\varphi(L_1(C))|_{u=p} \in p^n \g_R'$ and that for $0 < i \leq n$ 
$$\frac{ d^i \varphi(L_1(C))}{du^i} |_{u=p} \in (p^{p+ n -i},I_{N_\infty}) \g_R'. $$
\end{lem} 

\begin{proof}
The hypothesis implies that $A_i \in (p^{n-i} ,I_{N_\infty})\g_R'$ for $ 0\leq i \leq n$.  The first claim about $\varphi(L_1(C))|_{u=p} = A_0$ is immediate from equation \eqref{4.3.6}.  
For the second statement, 
we look closer at $Q(u)$.  In particular, there exists $Q_2(u) \in \g'_{R\pseries{u}}$ such that
\[
Q(u) = (A_1 - 2 p A_2 + 3 p^2 A_3 - \ldots) + u^{p} Q_2(u). 
\]
Since $A_i \in (p^{n-i},I_{N_\infty}) \g_R'$ for $ 0\leq i \leq n$, it follows that $Q(u)|_{u = p} \in (p^{n-1},I_{N_\infty}) \g_R'$.   By inspecting derivatives we see that $\frac{d^i Q(u)}{du^i} |_{u=p} \in (p^{p+1-i},I_{N_\infty}) \g_R'$ for $i>0$.  Using the product rule to compute derivatives of $u^p Q(u)$ gives the desired result about derivatives of $\varphi(L_1(C))$.
\end{proof}

\begin{cor} \label{cor:L2C}
For fixed $n  < p$, suppose for all $0 \leq i < n $, we know that
$$\frac{ d^i L_1(C)}{du^i} |_{u=p} \in (p^{n-i},I_{N_\infty}) \g_R'.$$
Then it follows that for all $0\leq i \leq n$, we have 
\[
\frac{d^i L_2(C)}{du^i}|_{u=p} \in (p^{n},I_{N_\infty}) \g_R'.
\]
\end{cor}

\begin{proof}
This follows from Lemma~\ref{lem:x1} and Lemma~\ref{lem:lieintegrality}\eqref{lieintegralityiii} applied to $\varphi(L_1(C))$.
\end{proof}

\begin{lem} \label{lem:inductivestep}
Suppose that $h < p-1$ and fix $n < h$.  
If for $0 \leq i < n$, we know that 
 $$\frac{d^i L_1(C)}{du^i} |_{u=p} \in (p^{n-i}, I_{N_\infty} ) \g_R',$$
 then for $0 \leq j < n+1$ we have that
 \[
  \frac{d^j L_1(C)}{du^j}|_{u=p} \in (p^{n+1-j}, I_{N_\infty} ) \g_R'.
 \]
\end{lem}

\begin{proof}
 Using Corollary~\ref{cor:L2C}, the hypothesis implies that 
\begin{align} \label{eq:l2c}
 \frac{d^i L_2(C)}{du^i}|_{u=p} \in (p^{n}, I_{N_\infty}) \g_R' \,\text{ for  }\, 0 \leq i \leq n.
 \end{align}
    By Lemma~\ref{lem:coefficients}(\ref{part:zi}), since $h < p-1$ we know that for $0 < i <p$
 \begin{equation}\label{eq:z1facts}
  v_p \left( z_1 \right) \geq p - (h-1)-1 \geq 2 \quad \text{and} \quad v_p \left( \frac{d^i z_1}{du^i} \right) \geq p - h -i +1 \geq 3-i.
 \end{equation}
  We will now show that  $\frac{d^m L_1(C)}{du^m}|_{u=p} \in (p^{n+1-m}, I_{N_\infty}) \g_R'$ for $0 \leq m < n+1$ by induction on $m$.  
  
  For the base case $m=0$, 
 we know that $L_2(C)|_{u=p} \in (p^n, I_{N_\infty}) \g_R'$, and that $v_p(z_1) >0$.  Therefore using \eqref{eq:pm0}
we conclude that $L_1(C)|_{u=p} \in (p^{n+1}, I_{N_\infty}) \g_R'$ as desired.

For the inductive step, fix $0<m < n+1$ and suppose that for all $0 \leq j < m$,
\[
 \frac{d^j L_1(C)}{du^j}|_{u=p} \in (p^{n+1-j}, I_{N_\infty}) \g_R'.
\]
We will consider each of the terms in equation \eqref{eq:pmn}.  First, using \eqref{eq:l2c} and \eqref{eq:z1facts} we see that
\[
z_1 |_{u=p} \cdot \frac{d^{m} L_2(C)}{du^{m}} |_{u=p} \in (p^{2+n}, I_{N_\infty}) \g_R' \subset  (p^{n+1-m}, I_{N_\infty}) \g_R'.
\]
Second, for $ 0<  i \leq m$ we see that
\[
 \frac{d^i z_1}{du^i} |_{u=p} \cdot \frac{d^{m-i} L_2(C)}{du^{m-i}} |_{u=p} \in (p^{n+3-i}, I_{N_\infty}) \g_R' \subset (p^{n+1-m}, I_{N_\infty}) \g_R'.
\]
Furthermore, for $ 0 < i \leq m$ we see that
 \[
  \frac{ d^i z_0 }{du^i}|_{u=p}  \cdot \frac{d^{m-i} L_1(C)}{du^{m-i}} |_{u=p} \in ( p^{n+1 -m}, I_{N_\infty}) \g_R'
 \]
 using the inductive hypothesis and Lemma~\ref{lem:coefficients}(\ref{part:z0val}).  Then equation \eqref{eq:pmn} plus the fact that $z_0 |_{u=p}$ is a unit imply that 
\[
 \frac{d^m L_1(C)}{du^m} |_{u=p} \in (p^{n+1-m}, I_{N_\infty}) \g_R'
\]
which completes the induction.
\end{proof}

\begin{cor} \label{cor:finaldivisibility}
Suppose that $h < p-1$. Then for any 
$ 0\leq i \leq h-1$, we have that 
 \[
  \frac{d^i L_1(C)}{du^i} |_{u=p} \in  (p^{h-i}, I_{N_\infty} )\g_R' \quad \text{ and } \quad \frac{d^i L_2(C)}{du^i} \in (p^h, I_{N_\infty} )\g_R' . 
 \]
\end{cor}

\begin{proof}
The first part follows from induction using Lemma~\ref{lem:inductivestep}.  The base case is that $L_1(C) \in (p, I_{N_\infty}) \g_R'$, which follows from \eqref{eq:pm0} and the information about $v_p(z_1|_{u=p})$ in Lemma~\ref{lem:coefficients}.  The statement about $L_2(C)$ is then Corollary~\ref{cor:L2C}.
\end{proof}

\begin{proof}[Proof of Theorem~\ref{thm:monoapprox}]
It suffices to show that the entries of $\frac{d^i L_1(C)}{du^i} |_{u=p}$ lie in $(I_{N_\infty},p)$ for $0 \leq i \leq h-2$, which follows from Corollary~\ref{cor:finaldivisibility}.
\end{proof}

\section{Affine Schubert varieties and monodromy} \label{sec:affsch}

 For this section, we allow $k$ to be any field and $G$ to be any split connected reductive group over $k$.    Let $G^{\mathrm{der}}$ denote its derived group with center $Z^{\mathrm{der}}$.
 Let $T$ denote a split maximal torus and choose a set of positive roots $\Phi_G^+$. Let $\mu \in X_*(T)$ be a dominant cocharacter.     We let $\LG$ and $\LpG$ denote the loop group and the positive loop group for $G$  respectively over $k$.  The affine Grassmanian $\Gr_G$ is the quotient $\LpG \backslash \LG$.   As before, the affine Schubert cell $\Gr^{\circ, \mu}_{G}$ is the reduced $\LpG$-orbit of $u^{\mu}$, and $\Gr^{\leq \mu}_{G}$ its closure.   

\subsection{Tangent spaces} 
Let $\mu'$ be a dominant cocharacter such that $\mu' \leq \mu$.  
For the key calculations in the next section, we will need some basic control over the tangent space of $\Gr_{G}^{\leq \mu}$ at the $T$-fixed point $[u^{\mu'}]$.  We will use the subgroup ind-scheme $L^{--} G \subset LG$ given by 
\[
L^{--} G(A) =  \{  g \in G(A[u^{-1}]) \mid g \mod u^{-1} \equiv 1 \}   
\]
for any $k$-algebra $A$. 
Recall that the natural map $L^{--} G \rightarrow \Gr_{G}$ is representable by an open immersion (see \cite[Lemma 3.1]{HRtest1} for example).    Thus we can identify the tangent space at the base point $e$ of $\Gr_{G}$ with $\Lie L^{--} G = u^{-1} (\g_{k[u^{-1}]}) \subset \g_{ k(\!(u)\!)}$.  Recall that we defined
\[
h_{\mu} = \max_{\alpha \in \Phi_G} \{ \langle \mu, \alpha \rangle \}.
\]

\begin{prop} \label{prop:ts} Assume that $Z^{\mathrm{der}}$ is \'etale over $k$.  Let $\mu$ and $\mu'$ be dominant coweights such that $\mu' \leq \mu$.  The tangent space to $(\Gr_{G}^{\leq \mu}) u^{-\mu'}$ at $e$ is contained in the subspace  $V^{\mu}_{\mu'}$ generated as a $k$-vector space by the following subspaces of $\Lie L^{--} G$: 
\begin{itemize}
\item for $\alpha \in \Phi_G$ and $1 \leq j \leq \frac{1}{2}(h_{\mu} -   \langle \mu', \alpha \rangle)$,  the subspaces $u^{-j}  \g_{\alpha}$;
\item  and for $1 \leq j \leq h_{\mu}$, the subspaces $u^{-j} (\Lie T)$.
\end{itemize}
\end{prop}

\begin{proof}
For any $\alpha \in \Phi_G$, we have the corresponding root group $U_{\alpha} \subset G$, which is a copy of $\Ga$.  For any integer $n$, $U_{\alpha} (t u^n) \subset L G$ is an affine root group with coordinate $t$.  For $j \geq 1$,  the tangent space to $U_{\alpha}(t u^{-j})$ is $u^{-j} \g_{\alpha}$.        

To bound the tangent space, we will use the adjoint representation denoted $\mathrm{Ad}$. The map $\Ad:G \rightarrow \GL(\Lie G)$ induces a map of affine Grassmannians  $\Ad_* :\Gr_{G} \rightarrow \Gr_{\GL(\Lie G)}$.    There is a closed subscheme $\Gr^{[-h_{\mu}, h_{\mu}]}_{\GL(\Lie G)}$ which is the image of the subfunctor $\{ g \in \textrm{L} {\GL(\Lie G)}(A) \mid u^{h_{\mu}} g, u^{h_{\mu}} g^{-1}  \in \End(\Lie G)(A[\![u]\!]) \}$.  This subfunctor is clearly closed under right multiplication by $\textrm{L}^+ {\GL(\Lie G)}$.    Since $\Ad_*(u^{\mu}) \in \Gr^{[-h_{\mu}, h_{\mu}]}_{\GL(\Lie G)}$, it follows that  
\[
\Ad_*( \Gr^{\leq \mu}_G) \subset \Gr^{[-h_{\mu}, h_{\mu}]}_{\GL(\Lie G)} .
\] 

Assuming that $U_{\alpha}(tu^{-j})$ lies in $( \Gr^{\leq \mu}_G) u^{-\mu'}$,   then $\Ad_{*}(U_{\alpha}(tu^{-j})) \Ad_{*}(u^{\mu'})$  lies in $ \Gr^{[-h_{\mu}, h_{\mu}]}_{\GL(\Lie G)}$ and so satisfies the height condition.   Let $e_{-\alpha}$ and $e_{\alpha}$ generate $\g_{-\alpha}$ and $\g_{\alpha}$ respectively in $\Lie G$.   The height condition defining $ \Gr^{[-h_{\mu}, h_{\mu}]}_{\GL(\Lie G)}$ implies that 
\begin{equation} \label{f1}
\Ad(U_{\alpha}(tu^{-j})) \Ad(u^{\mu'}) (e_{-\alpha}) \in u^{-h_{\mu}} \Lie L^+ G.   
\end{equation}
A straightforward computation, for example using the map $\mathrm{SL}_2 \rightarrow G$ sending the upper triangular $\Ga$ to $U_{\alpha}$, shows that
the coefficient on $e_{\alpha}$ of $\Ad(U_{\alpha}(tu^{-j})) \Ad(u^{\mu'}) (e_{-\alpha})$ is a scalar multiple of $-t^2 u^{-2j - \langle  \mu',\alpha \rangle}$.   
Thus, equation \eqref{f1} implies that 
\[
2j +  \langle  \mu', \alpha \rangle \leq h_{\mu} 
\]     
which implies the first item. 

We now explain how to reduce checking the second item to the adjoint case.   Let $G^{\mathrm{der}}, G^{\mathrm{ad}}$ denote the derived and adjoint groups of $G$ respectively.   By \S 6.a.1 of \cite{PRtwisted}, $\Gr_{G^{\mathrm{der}}} \rightarrow \Gr_{G^{\mathrm{ad}}}$ is a closed immersion.  Furthermore, by Proposition 6.6 in \emph{loc. cit.}, the reduced connected component of $\Gr_G$ containing the base point is identified with the same for $\Gr_ {G^{\mathrm{der}}}$.   Thus, $\Gr_G^{\leq \mu} u^{-\mu'}  \subset \Gr_ {G^{\mathrm{der}}}$.     It is clear then that if we let $\overline{\mu}, \overline{\mu}'$ denote the image of $\mu$ and $\mu'$ in the $X_*(T^{\mathrm{ad}})$, then 
\[
\Gr_G^{\leq \mu} u^{-\mu'}   \cong \Gr_{G^{\mathrm{ad}}}^{\leq \overline{\mu}} u^{-\overline{\mu}'}. 
\]
Thereby, we may reduce to the case when $G$ is adjoint to analyze the torus contribution.  

Let $Y \in \Lie T$, and assume $u^{-j} Y$  lies in the tangent space to $(\Gr_G^{\leq \mu}) u^{-\mu'}$.   By the adjointness assumption,
 there exists some $\alpha \in \Phi_G^+$ such that $\alpha(Y) \neq 0$.   By the same argument as before, considering $\Ad(u^{-j}Y) \Ad(u^{\mu'}) e_{-\alpha}$ gives that
\[
j  \leq h_{\mu} - \langle \mu', \alpha \rangle \leq h_{\mu}. \qedhere
\]  
\end{proof}

\begin{remark} There is a formula for the tangent space to $\Gr_{G}^{\leq \mu}$ at $[ u^{-\mu'} ]$ in terms of affine Demazure modules in characteristic 0 due to Kumar \cite{kumar96} for Kac--Moody groups.  It is unclear to us whether this formula holds in characteristic $p$ and so we opted to prove the upper bound directly.  
\end{remark}

\begin{remark} \label{remark:specialGLn} 
If $k$ has characteristic $p$ and $G = \GL_n$ the hypothesis that $p \nmid n$ (so $Z^{\der} = \mu_n$ is \'{e}tale) is not needed.  We can argue directly, without reducing to checking the second item for adjoint case, by considering the action of $T$ on the standard representation.  The key is the elementary observation that if $\mu$ is the cocharacter given by the $n$-tuple of integers $(a_1,\ldots, a_n)$ then the standard representation has height in $[\min_i a_i, \max_i a_i]$, while $h_\mu = \max_i a_i - \min_i a_i$. 
\end{remark}
 
\subsection{Monodromy locus}
Recall that for any $k$-algebra $A$ and any $C \in \LG(A)$, we defined an element $\frac{dC}{du} C^{-1} \in \Lie \LG (A)$ characterized using the Tannakian formalism in Definition~\ref{def:dc}.   Define a closed subfunctor 
\[
\LG^{\nabla}(A) = \{ C \in \LG(A) \mid \frac{d C}{du} C^{-1} \in \frac{1}{u} \Lie \LpG (A) \} \subset \LG(A).  
\]
Remark~\ref{remark:leibniz} shows that $LG^{\nabla}$ is stable under left multiplication by $\LpG$, so we can define closed sub-(ind)-schemes
\begin{equation}
\Gr_{G}^{\nabla} := \LpG \backslash \LG^{\nabla} \subset \Gr_G, \quad  \Gr_{G}^{\leq \mu,\nabla} := \Gr_{G}^{\leq \mu} \cap \Gr^{\nabla}_G \subset \Gr_{G}^{\leq \mu}.  
\end{equation}

We can now state the main theorem of the section. 
 
 \begin{thm} \label{thm:monodromyschubert}  Let $\mu \in X_*(T)$ be a dominant coweight. If $k$ has finite characteristic $p$, then assume that $\langle \mu, \alpha \rangle < p$ for all positive roots $\alpha$ and that $Z^{\der}$ is \'{e}tale over $k$.   Then we have that
\[
\Gr_{G}^{\leq \mu, \nabla} \cong \coprod_{\mu' \leq \mu}  P_{\mu'} \backslash G,
\]
where $P_{\mu'}$ is the parabolic subgroup associated to $\mu'$.  In particular, $\Gr_{G}^{\leq \mu, \nabla}$ is smooth.   
  \end{thm} 
  
 \begin{remark}
  The strategy is similar to the approach taken in \cite[\S3-4]{localmodels} where an analogous fact is proven for $\GL_n$.
  \end{remark}
    
 The proof will reduce to the following key computation on tangent spaces.
 
 \begin{prop} \label{prop:monots}  If $k$ has finite characteristic $p$, assume that  $h_{\mu}  < p$ and that $Z^{\der}$ is \'{e}tale over $k$. For a dominant $\mu' \leq \mu$, the dimension of the tangent space of  $\Gr_{G}^{\leq \mu, \nabla}$ at $[u^{\mu'}]$ is equal to $\dim P_{\mu'} \backslash G$.  
\end{prop}
\begin{proof}
As in Proposition \ref{prop:ts}, we translate and consider the tangent space $W^{\mu}_{\mu'}$ of $(\Gr_{G}^{\leq \mu}) u^{-\mu'}$ at the base point $e$ as a subspace of $\Lie L^{--} G$.   By \emph{loc. cit.}, $W^{\mu}_{\mu'} \subset V^{\mu}_{\mu'}$.     Since $\Gr^{\circ,\mu'}_{G} \subset \Gr_{G}^{\leq \mu}$, Lemma \ref{lem:schuberttangent} gives that $V_{\mu'} \subset W^{\mu}_{\mu'}$.       

It is elementary to check that $u^{\mu'} \in \LG^{\nabla}$ and so $[u^{\mu'}] \in \Gr_{G}^{\leq \mu, \nabla}$.  Let  $V_{\mu'}^{\nabla}$ and $W^{\mu, \nabla}_{\mu'}$ respectively denote the tangent spaces to $(\Gr_{G}^{\circ,\mu'} \cap \Gr_G^{\nabla}) u^{-\mu'}$ and $(\Gr^{\leq \mu, \nabla}_{G}) u^{-\mu'}$ at the identity.   For $Y \in W^{\mu}_{\mu'}$ that translates to $C \in \LG(k[\epsilon]/(\epsilon)^2)$, the key observation is that $Y \in W^{\mu, \nabla}_{\mu'}$ if and only if 
\begin{equation} \label{eq:liemono}
u \frac{dY}{du} + [ Y, \mu'] \in \g_{k\pseries{u}}.
\end{equation}
This can be checked using any faithful representation $V$ of $G$, where we compute that $[C]_V= ([1]_V + \epsilon [Y]_V )u^{\mu'}$ and hence
\begin{align*}
 \frac{d[C]_V}{du} [C]_V^{-1} &= \frac{d\left( ([1]_V + \epsilon [Y]_V )u^{\mu'} \right)}{du} u^{-\mu'}([1]_V - \epsilon [Y]_V ) \\
 &= \frac{d u^{\mu'}}{du} u^{-\mu'} [1]_V + \epsilon \left( \frac{d[Y]_V}{du}  + [Y]_V \frac{du^{\mu'}}{du} u^{-\mu'} - \frac{d u^{\mu'}}{du} u^{-\mu'} [Y]_V \right).
\end{align*}
We see this is in $\frac{1}{u} \Lie \LpG(k[\epsilon]/(\epsilon^2))$ precisely when $u \frac{d[Y]_V}{du} + [ [Y]_V,\mu'] \in \g_{k\pseries{u}}$.

We first analyze $V_{\mu'}^{\nabla}$.  Using Lemma~\ref{lem:schuberttangent}, for any $\alpha$ such that $\langle \mu', \alpha \rangle < 0$ we can represent the $\alpha$-component of $Y$ as $\displaystyle Y_{\alpha} = \sum_{i=-1}^{\langle \mu', \alpha \rangle} Y_{\alpha, i} u^{i}$ where $Y_{\alpha, i} \in k$.   Then condition \eqref{eq:liemono} becomes 
\[
\sum_{i=-1}^{\langle \mu', \alpha \rangle} (i - \langle \mu', \alpha \rangle ) Y_{\alpha, i} u^{i}  \in k[\![u]\!].
\]
Since $-p < \langle \mu', \alpha \rangle < 0$, we see $(i - \langle \mu', \alpha \rangle) \neq 0$ in $k$ for $\langle \mu', \alpha \rangle < i < 0$ and hence $Y_{\alpha,i}=0$. There is no restriction on $Y_{\alpha, \langle \mu',\alpha \rangle}$.   Since $\dim P_{\mu'} \backslash G$ is equal to the number of positive roots such that $\langle \mu', \alpha \rangle > 0$, we have $\dim V_{\mu'}^{\nabla} = \dim P_{\mu'} \backslash G$ and hence
\begin{equation} \label{eq:lowerbound}
\dim P_{\mu'} \backslash G = \dim V_{\mu'}^{\nabla} \leq \dim W^{\mu,\nabla}_{\mu'} .
\end{equation}

Define $V_{\mu'}^{\mu, \nabla}$ to be the subspace of $V^{\mu}_{\mu'}$ satisfying \eqref{eq:liemono}, and let $Y \in V^{\mu}_{\mu'}$.   As before, we can write the $\alpha$-component of $Y$ as $\displaystyle Y_{\alpha} = \sum_{i=-1}^{-N} Y_{\alpha, i} u^{i}$ with $Y_{\alpha, i} \in k$ where $N := \frac{1}{2}(h_{\mu} -  \langle \mu', \alpha \rangle)$ is the bound given in Proposition \ref{prop:ts}.  Again, we must have
\[
\sum_{i=-1}^{-N} (i - \langle \mu', \alpha \rangle ) Y_{\alpha, i} u^{i}  \in k[\![u]\!].
\]
Notice that for $-N \leq i \leq -1$, we have $-1 - \langle \mu', \alpha \rangle \geq i - \langle \mu', \alpha \rangle \geq - \frac{1}{2} (h_{\mu}  + \langle \mu', \alpha \rangle)$.
For $\alpha \in \Phi_G^+$, since $\mu' \leq \mu$ we see that $ \langle \mu', \alpha \rangle \leq h_{\mu}$.  Hence $i - \langle \mu', \alpha \rangle > -p$, and thus
 $Y_{\alpha,i} =0$.   If $\alpha \in \Phi_G^-$, a similar check confirms that $ (i - \langle \mu', \alpha \rangle) \equiv 0 \mod p$ only if $i = \langle \mu', \alpha \rangle$.    Hence $Y_{\alpha,i}=0$ except if $i = \langle \mu', \alpha \rangle$ and $ \langle \mu', \alpha \rangle \neq 0$.

Finally, consider $Y \in V^{\mu}_{\mu'} \cap \ft_{k(\!(u)\!)}$.  By Proposition \ref{prop:ts}, $\displaystyle Y = \sum_{i=1}^{h_{\mu}}  Y_i u^{-i}$ where $Y_i \in \ft$.  Condition \eqref{eq:liemono} becomes 
\[
\sum_{i=1}^{h_{\mu}} ( -i Y_i) u^{-i} \in \ft_{ k[\![u]\!]}. 
\]
Since $h_{\mu} < p$, this implies that each $Y_i = 0$.  

Since $\dim P_{\mu'} \backslash G$ equals the number of positive roots such that $\langle \mu',\alpha \rangle >0$, we conclude that $\dim W^{\mu,\nabla}_{\mu'} \leq \dim V_{\mu'}^{\mu, \nabla} \leq \dim P_{\mu'} \backslash G$.  Together with \eqref{eq:lowerbound} this completes the proof.
\end{proof}

Letting $e_\alpha$ generate $\g_\alpha$, the previous proof also shows: 
   
\begin{cor} \label{cor:tspacenabla}
With hypotheses as in Proposition~\ref{prop:monots}, using the map on tangent spaces induced by right translation by $u^{\mu'}$, the 
tangent space of $\Gr_{G}^{\leq \mu, \nabla}$ at $u^{\mu'}$ is contained in
$$ \operatorname{span}_k \left( \left\{ u^{\langle \mu',\alpha \rangle} e_\alpha  \right \}_{\alpha \in \Phi_G, \langle \mu',\alpha \rangle <0} \right) + \g_ {k[\![ u]\!]}
\subset \Lie \Gr_{G_k} = \g_{k(\!(u)\!)} / \g_ {k[\![ u]\!]}.
$$
\end{cor}   

We now consider the open affine Schubert cell $\Gr_{G}^{\circ,\mu'}$ and its intersection with $\Gr_G^{\nabla}$ for any dominant cocharacter $\mu'$.  
We show that $\Gr_{G}^{\circ, \mu', \nabla} := \Gr_{G}^{\circ,\mu'} \cap \Gr_G^{\nabla} $ is equal to the orbit $u^{\mu'} G$; this is known to be a flag variety.  

\begin{prop} \label{prop:monodromyopenschubert}  Let $\mu' \in X_*(T)$ be a dominant coweight.
If $k$ has finite characteristic $p$, then assume that $\langle \mu', \alpha \rangle < p$ for all positive roots $\alpha \in \Phi_G^+$ and that $Z^{\der}$ is \'{e}tale over $k$.   Then we have
\[
\Gr_{G}^{\circ, \mu', \nabla} \cong P_{\mu'} \backslash G.
\]
\end{prop}
  
\begin{proof} 
We first show that the closed subscheme $\Gr_{G}^{\circ,\mu', \nabla} \subset \Gr_{G}^{\circ, \mu'}$ is stable under the right multiplication by $G$.  
It suffices to show that $\LG^{\nabla}$ is stable under right multiplication by $G$.   Let $A$ be a $k$-algebra.  For any $C \in LG^{\nabla}(A)$ and $g \in G(A)$, by the Remark~\ref{remark:leibniz}
\[
 \frac{d (Cg)}{du} (Cg)^{-1} = \frac{dC}{du} C^{-1} + \Ad_G(C) \left( \frac{dg}{du} g^{-1} \right). 
\]
As $g \in G(A)$, $ \frac{dg}{du} g^{-1} = 0$ and hence $Cg \in LG^{\nabla}(A)$.  

Next, let $X_{\mu'} \subset \Gr_{G}^{\circ,\mu'}$ denote the orbit $u^{\mu'} G$.  It is well-known that $X_{\mu'} \cong P_{\mu'} \backslash G$ and hence is projective (see top of pg. 100 of \cite{MV}).   Now note that $X_{\mu'} \subset  \Gr_{G}^{\circ,\mu', \nabla}$ since $u^{\mu'} \in LG^{\nabla}(k)$.  
Furthermore, we claim that the inclusion $X_{\mu'} \subset \Gr_{G}^{\circ,\mu', \nabla}$ is an open immersion and thus (the projective) $X_{\mu'}$ is a connected component of $ \Gr_{G}^{\circ, \mu',\nabla}$.  Since the inclusion is $G$-equivariant and $G$ acts transitively on $X_{\mu'}$, it suffices to show that map on tangent spaces at $u^{\mu'}$ is an isomorphism.  This follows from proof of Proposition \ref{prop:monots} which shows that the dimension of the tangent space of  $\Gr_{G}^{\circ,\mu', \nabla}$ at $u^{\mu'}$ is equal to $\dim P_{\mu'} \backslash G$.

Finally, recall the loop $\Gm$ action on $\Gr_G$ denoted $\check{\delta}$ from \cite[\S 2]{MV} which sends $u$ to $\alpha u$.   Equation (2.4) in \emph{loc. cit.}  says that $\Gm$ action of $\check{\delta}$ contracts $\Gr^{\circ,\mu'}_{G}$ onto $X_{\mu'}$.  It is easy to see that $\LG^{\nabla}$ is stable under $\check{\delta}$.    Thus, any component of $\Gr_{G}^{\circ,\mu', \nabla}$ must necessarily intersect $X_{\mu'}$.    
\end{proof}

\begin{proof}[Proof of Theorem \ref{thm:monodromyschubert}]
 Given that $\mu$ satisfies the hypotheses of Theorem \ref{thm:monodromyschubert}, so does any $\mu' \leq \mu$.  Since $\Gr_{G}^{\leq \mu}$ is topologically the union of $\Gr^{\circ,\mu'}_{G}$ for $\mu' \leq \mu$, by Proposition \ref{prop:monodromyopenschubert}
\begin{equation} \label{eq:reduced}
(\Gr^{\leq \mu, \nabla}_{G})_{\mathrm{red}} \cong \coprod_{\mu' \leq \mu}  P_{\mu'} \backslash G.
\end{equation}
To prove Theorem \ref{thm:monodromyschubert}, then, we just need to show that $\Gr^{\leq \mu, \nabla}_{G}$ is reduced.

By \eqref{eq:reduced}, $G$ acts transitively on each connected component of $\Gr^{\leq \mu, \nabla}_{G}$.   To show that $\Gr^{\leq \mu, \nabla}_{G}$ is reduced, it therefore suffices to show that the dimension of the tangent space at $[u^{\mu'}]$ is equal to $\dim P_{\mu'} \backslash G$ which follows from Proposition \ref{prop:monots}.  
\end{proof}

\section{Proof of the Main Theorem} \label{sec:proof}

Let $\Lambda$, $\F$, $L$, and $G$ be as before, and continue to assume that $p \nmid \# \pi_1(G^{\der})$.  Throughout this section, we fix a $G$-Kisin module $\fPbar$ over $\F$ and a continuous Galois representation $\rhobar: \Gamma_K \to G(\F)$ together with an isomorphism $\tT_{G,\F}(\fPbar) \simeq \rhobar|_{\Gamma_\infty}$ as we did in \S\ref{ss:deformationproblems} to define our deformation problems.

\subsection{Kisin varieties}  \label{ss:kisinvariety}
We begin by studying the Kisin variety and identifying some conditions that guarantee it is trivial.

\begin{defn} \label{defn:kisinvariety}
Fix $\cP \in \GMod^{\varphi}_{\cO_{\cE}, \F}$ and a dominant cocharacter $\mu$ for $G'$.  The Kisin variety $Y_{\cP}^{\leq \mu}$ is the projective scheme over $\F$ which represents the functor sending an $\F$-algebra $A$ to the set of $G$-Kisin lattices of type $\leq \mu$ in $\cP_A$.  
\end{defn}

Let $\cP = M_{G,\F}(\rhobar|_{\Gamma_\infty})$ and $\fm$ be the maximal ideal of $R^{\mu,\square}_\rhobar$.  
The Kisin variety is an ``upper bound'' on the fiber of $X^{\mu}_{\rhobar}$ over $\rhobar$.  More precisely, Proposition~\ref{prop:resolution}\eqref{resolutionII} gives an inclusion
\begin{equation} \label{eq:fiberresolution}
X^{\mu}_{\rhobar} \times_{\Spec R^{\mu,\square}_\rhobar} \Spec R^{\mu,\square}_\rhobar/ \fm \subset Y^{\leq \mu,}_{\cP}.
\end{equation}

\begin{lem} \label{lem:kvmonodromy}
Suppose $\mu$ is Fontaine-Laffaille and $A$ is a local Artinian $\F$-algebra with residue field $\F'$.  Fix $x \in X^{\mu}_{\rhobar}(\bF')$ and a lift $\widetilde{x} \in \Spf(\widehat{\cO}_x^{\mu})(A)$ where $\widehat{\cO}_x^{\mu}$ is the completion of $X^{\mu}_{\rhobar}$ at $x$.  The $G$-Kisin module $\fP$ corresponding to $\widetilde{x}$ using Proposition~\ref{prop:resolution}(\ref{resolutionI}) satisfies the mod-$p$ monodromy condition.  
\end{lem} 

\begin{proof}
Notice that $\widehat{\cO}_x^{\mu}$ is $\Lambda$-flat and reduced by 
Proposition~\ref{prop:resolution}(\ref{resolutionIII}) and Fact~\ref{fact:crystallinedeformation}.   Since $\Theta[1/p]$ is an isomorphism, the $\overline{L}$-points of $\widehat{\cO}_x^{\mu}$ correspond to crystalline representations with $p$-adic Hodge type $\mu$.  Using Corollary~\ref{cor:Gmonodromy} we conclude that $I_{N_{\infty}} = 0$.   The Lemma then follows from Corollary~\ref{cor:modpmonodromy}. 
\end{proof}

We set $\bK := G(\F\pseries{u})$, and begin by recalling 
several useful facts related to the Cartan decomposition.  We fix a split maximal torus $T$ contained in a Borel $B$ and adopt our standard notation for root systems (see \S\ref{notation:rootdata}).  Let $T^{\der}$ be the corresponding maximal torus of $G^{\der}$.

\begin{fact} \label{fact:cartan}
We have that:
\begin{enumerate}[(i)]
\item  (Cartan Decomposition) $\displaystyle G(\F\lseries{u}) = \coprod_{\mu \in X_*(T)_+} \bK u^\mu \bK $; 

\item  \label{fact:cartanii} for $\mu, \lambda, \omega \in X_*(T)_+$, if $\bK u^\lambda \bK u ^ {\omega} \bK \cap \bK u^\mu \bK \neq \emptyset$, then $\mu \leq \lambda + \omega$.

\end{enumerate}

\end{fact}

The first is the standard Cartan decomposition; see for example \cite[3.3.3]{tits79}.  The second is a special case of \cite[Proposition 4.4.4(iii)]{bt1}; this is somewhat complicated to apply but we do not know a more direct reference.  The footnote in \cite[\S6.9]{hv15} explains how to translate our set-up into the language of \cite[Proposition 4.4.4]{bt1}.  The $K$ in \emph{loc. cit.} is actually a larger group containing $\bK$, which is harmless.

\begin{lem}\label{lem:threecosets}
For dominant coweights $\mu, \lambda, \omega, \nu$, if $\bK u ^\lambda  \bK u^{\omega}  \bK u^{\nu} \bK \cap \bK u^\mu \bK \neq \emptyset$ then $\mu \leq \lambda + \omega + \nu$.
\end{lem}

\begin{proof}
Apply Fact~\ref{fact:cartan}\eqref{fact:cartanii} twice.
\end{proof}

Given a coweight $\lambda$, define $\lambda^{\dom}$ to be the unique dominant coweight in the same Weyl-orbit.

\begin{lem} \label{lem:cartaninverse}
Let $\mu = (\mu_\sigma)_{\sigma \in \cJ}$ be a type for $G$.  If $\displaystyle g \in \prod_{ \sigma \in \cJ} \bK u^{\mu_\sigma} \bK$ then 
\begin{enumerate}[(i)]
\item  $\displaystyle g^{-1}_\sigma \in \bK u^{(-\mu_\sigma)^{\dom}}  \bK$.
\item  $\displaystyle \varphi(g)_{\sigma \varphi} \in \bK u^{p \mu_\sigma}  \bK$. 
\end{enumerate}
\end{lem}

\begin{proof}
Clear.
\end{proof}

Let $\Phi_{G}^{\high}$ denote the set of highest roots of the  irreducible root systems which appear in the irreducible decomposition of the root system of $G$.
Recall that $\displaystyle h_\mu = \max_{\alpha \in \Phi_{G'}} \langle \mu,\alpha \rangle $.

\begin{prop} \label{prop:uniquelattice}
Let $\cP = (P,\phi) \in \GMod^{\varphi}_{\cO_{\cE},\FF}$ and $\mu$ be a dominant type of $G$ such that for every non-zero dominant coweight $\lambda$ of $G^{\der}$ there exists $\alpha_h \in \Phi_G^{\high}$ such that
\[
(p-1) \langle \lambda, \alpha_h \rangle > 2 h_\mu.
\]
Then for any finite extension $\F' / \F$ there is at most one $G$-Kisin lattice of type $\leq \mu$ in $\cP_{\F'}$ (up to isomorphism). 
\end{prop}

\begin{proof}
 Suppose that $(\fP',\alpha')$ and $(\fP'',\alpha'')$ are two $G$-Kisin lattices in $\cP_{\F'} = (P_{\F'}, \phi_{P_{\F'}})$.  After trivializing $P_{\F'}$ and identifying $\epsilon_G(\fP')$ and $\epsilon_G(\fP'')$ with $P_{\F'}$, we obtain $A, \Phi$, and $\Phi'$ in $G'(\F'\lseries{u})$ such that 
\begin{equation} \label{eq:changebasis1}
\Phi'' = A \Phi' \varphi(A)^{-1}.
\end{equation}
Write  $A = (A_\sigma)_{\sigma \in \cJ}, \Phi' = (\Phi'_\sigma)_{\sigma \in \cJ},$ and $\Phi'' = (\Phi''_\sigma)_{\sigma \in \cJ}$.  As the $G$-Kisin lattices are of type $\leq \mu$, we know that there are dominant cocharacters $\mu', \mu'' \leq \mu$ such that for each $\sigma \in \cJ$ we have $\Phi'_\sigma \in \bK u ^{\mu'_\sigma} \bK$ and $\Phi''_\sigma \in \bK u ^{\mu''_\sigma} \bK$.  There is a unique dominant coweight $\lambda$ of $G'$ such that $A_\sigma \in \bK u^{\lambda_\sigma} \bK$ for each $\sigma \in \cJ$.  It suffices to show that $\lambda=0$ so that $A \in G'(\F'\pseries{u})$.

Working with components as in Example~\ref{ex:embeddingcomponents}, we notice that \eqref{eq:changebasis1} is equivalent to 
\begin{equation} \label{eq:changebasis2}
\varphi(A _{\sigma \varphi^{-1}}) = \varphi(A)_{\sigma} = (\Phi''_{\sigma })^{-1} A_{\sigma } \Phi'_{\sigma } \quad \text{ for each } \quad \sigma \in \cJ.
\end{equation}
We know that $(\Phi''_{\sigma})^{-1} \in \bK u^{(-\mu_{\sigma} '')^{\dom}}  \bK$ and that $\varphi(A)_{\sigma} \in \bK u^{p \lambda_{\sigma \varphi^{-1}}} \bK$ by Lemma~\ref{lem:cartaninverse}.  
 Applying Lemma~\ref{lem:threecosets} to equation \eqref{eq:changebasis2} gives that
\begin{equation} \label{eq:coweights}
 p \lambda_{\sigma \varphi^{-1}} - \lambda_{\sigma } \leq \mu_\sigma' + (-\mu''_\sigma)^{\dom}.
\end{equation}
In other words, $\mu'_\sigma  + (-\mu''_\sigma)^{\dom} -  p \lambda_{\sigma \varphi^{-1}} + \lambda_{\sigma }$ is a non-negative linear combination of simple coroots.  Since any $\alpha_h \in \Phi_G^{\high}$ pairs non-negatively with positive coroots, 
this implies
\[
 \langle   \mu'_\sigma + (-\mu''_\sigma)^{\dom}  -  p \lambda_{\sigma \varphi^{-1}}+ \lambda_{\sigma}, \alpha_h \rangle \geq 0.
\]
 
Now as $\mu', \mu'' \leq \mu$, the difference $\mu'_\sigma - \mu''_\sigma$ is a combination of simple coroots for each $\sigma \in \cJ$.  Then $\mu'_\sigma - \mu''_\sigma$ and hence $\mu'_\sigma + (- \mu''_{\sigma})^{\dom}$ is in $X_*(T^{\der})$.  Likewise we see $p \lambda_{\sigma \varphi^{_1}} - \lambda_\sigma \in X_*(T^{\der})$ for each $\sigma \in \cJ$, and hence that $(p^{\# \cJ} -1) \lambda_{\sigma} \in X_*(T^{\der})$.  As the quotient of $X_*(T^{\der}) \to X_*(T^{\der})$ is torsion-free, we conclude that $\lambda_{\sigma}$ is a cocharacter valued in $G^{\der}$.

Fix a $\sigma \in \cJ$ and $\alpha_h \in \Phi_G^{\high}$ that maximize $\langle \lambda_{\sigma \varphi^{-1}},\alpha_h \rangle$: if $\lambda \neq 0$ then $(p-1)\langle \lambda_{\sigma \varphi^{-1}},\alpha_h \rangle >2 h_\mu$ by hypothesis.  There is an element $w$ of the Weyl group such that $w (-\mu''_\sigma)^{\dom} = -\mu''_\sigma$ and hence $$|\langle (-\mu''_\sigma)^{\dom} , \alpha_h \rangle| =  | \langle \mu''_\sigma, w \alpha_h \rangle | \leq  | \langle \mu_\sigma, w \alpha_h \rangle | \leq h_\mu.$$
Likewise as $\mu' \leq \mu$ we see that $\langle \mu'_\sigma, \alpha_h \rangle \leq h_\mu$. 
We know $\langle \lambda_{\sigma } ,\alpha_h \rangle \leq \langle \lambda_{\sigma \varphi^{-1}}, \alpha_h \rangle$ by choice of $\sigma$, so we see that 
\begin{equation} \label{eq:pairings}
  \langle   \mu'_\sigma + (-\mu''_\sigma)^{\dom}  -  p \lambda_{\sigma \varphi^{-1}} + \lambda_{\sigma }, \alpha_h \rangle \leq 2h_\mu - p \langle \lambda_{\sigma \varphi^{-1}},\alpha_h \rangle + \langle \lambda_{ \sigma \varphi^{-1}},\alpha_h \rangle < 0.
\end{equation}
This contradiction shows that $\lambda = 0$.
\end{proof}

We next show that the hypothesis on $\mu$ in Proposition~\ref{prop:uniquelattice} is satisfied when $\mu$ is sufficiently small.

\begin{lem} \label{lem:coweightpairing}
Let $\mu$ be a dominant type for $G$.  For every non-zero dominant cocharacter $\lambda$ of $G^{\der}$, there exists $\alpha_h \in \Phi_G^{\high}$ such that
\begin{align*}
(p-1) \langle \lambda, \alpha_h \rangle >  2 h_\mu = 2 \max_{\alpha \in \Phi_{G'}} \langle \mu, \alpha \rangle
\end{align*}
provided that either
\begin{enumerate}[(i)]
\item  $\mu$ is strongly Fontaine-Laffaille (i.e. $\langle \mu, \alpha \rangle < \frac{p-1}{2}$ for every $\alpha \in \Phi_{G'}$);
\item  $\mu$ is Fontaine-Laffaille (i.e. $\langle \mu, \alpha \rangle < p-1$ for every $\alpha \in \Phi_{G'}$) and there are no nonzero coweights $\lambda$ of $G$ such that $\langle \lambda, \alpha_h \rangle \leq 1$ for all $\alpha_h \in \Phi_G^{\high}$;

\item  $\mu$ is Fontaine-Laffaille and $G^{\der}$ is simply connected.
\end{enumerate}
\end{lem}

\begin{proof}
The first and second are immediate, as $\langle \lambda ,\alpha_h\rangle$ must be a positive integer for any non-zero dominant coweight $\lambda$.

When the root system of $G$ is irreducible and $G^{\der}$ is simply connected, it turns out there are no coweights $\lambda$ with $\langle \lambda , \alpha_h \rangle =1$ so the third follows from the second.  In particular, the condition $\langle \lambda , \alpha_h \rangle =1$ is one of several equivalent definitions for $\lambda$ being a 
minuscule weight of the dual root system; see \cite[Ch. VIII \S7.3]{bourbaki} for the properties of minuscule weights.  Let $\alpha_1, \ldots, \alpha_r$ be the positive simple roots, and write
\[
 \alpha_h = n_1 \alpha_1 + \ldots n_r \alpha_r
\]
with the $n_i$ positive integers.   Then $\lambda$ being minuscule is equivalent to the condition that $\lambda = \varpi^\vee_i$ for some fundamental weight $\varpi^\vee_i$ with $n_i=1$ \cite[VIII \S7.3 Proposition 8]{bourbaki}.
Furthermore, the minuscule $\varpi_i^\vee$  form a system of representatives for $P(\Phi_G^\vee)/ Q(\Phi_G^\vee)$.  But $G$ is simply connected, so $X^*(T) = P(\Phi_G)$ and $X_*(T) = Q(\Phi_G^\vee)$.  Thus none of the minuscule coweights lie in the cocharacter group, so there are no dominant minuscule cocharacters.  Therefore there are no coweights $\lambda$ with $\langle \lambda , \alpha_h \rangle =1$.

If $G^{\der}$ is simply connected but the root system is not irreducible, decompose the root system into irreducibles and write $G^{\der}$ as a product of simply connected groups with irreducible root systems.  The claim follows as the decomposition of the root system is orthogonal.
\end{proof}

\begin{cor} \label{cor:kisinvpoint}
Let $\cP \in \GMod^{\varphi}_{\cO_{\cE},\FF}$ and $\mu$ be a dominant cocharacter of $G'$ such that either
\begin{enumerate}[(i)]
\item  $\mu$ is strongly Fontaine-Laffaille;
\item  $\mu$ is Fontaine-Laffaille and $G^{\der}$ is simply connected.
\end{enumerate}
Then for any finite extension $\F'$ of $\F$, we have $Y_{\cP}^{\leq \mu} (\F')$ is either empty or a single element.
\end{cor}

\begin{proof}
Combine Proposition~\ref{prop:uniquelattice} with Lemma~\ref{lem:coweightpairing}.
\end{proof}

Next we establish a result about tangent spaces.  

\begin{prop} \label{prop:kvtspace}
Fix a finite extension $\F' / \F$, a dominant cocharacter $\mu$ of $G'$ in the Fontaine-Laffaille range, and $\fPbar \in Y^{\leq \mu}(\F')$.  Consider lifts $\fP_1, \fP_2 \in Y^{\leq \mu}(\F'[\epsilon]/(\epsilon^2))$ of $\fPbar$ that satisfy the mod-$p$ monodromy condition.  If $\epsilon_G(\fP_1) \simeq \epsilon_G(\fP_2)$ compatibly with the identification of the reduction with $\fPbar$, then  $\fP_1$ and $\fP_2$ are isomorphic as deformations of $\fPbar$.
\end{prop}

We continue our standard notation of $\LG'$, $\LpG'$, $\Gr_{G'}$, $\Gr^{\leq \mu}_{G'}$ and $\Gr^{\circ, \leq \mu}_{G'_{\F}}$ for loop groups and affine Grassmanians as summarized in \S\ref{notation:affinegrassmanian}, and use the description of the tangent space to $\Gr^{\circ, \leq \mu}_{G'_{\F}}$ provided by Lemma~\ref{lem:schuberttangent}.
For $C \in \LG'$, let $m_C$ denote the right multiplication by $C$ map on $\LG'$ or $\Gr_{G'}$. 

\begin{remark} \label{rmk:dualmult}
An $\F'[\epsilon]/(\epsilon^2)$-valued point of $\LG'$ corresponds to an $\F'$-valued point of $\LG'$ plus a tangent vector to $\LG'$ at that point; let $\iota$ denote the identification of tangent vectors with $\F[\epsilon]/(\epsilon^2)$-points. At the identity, multiplication of $\F[\epsilon]/(\epsilon^2)$-points corresponds to addition in the Lie algebra.
  For use in the proof of Proposition~\ref{prop:kvtspace}, we record how multiplication on $\F'[\epsilon]/(\epsilon^2)$-points interacts with identifying the tangent spaces with the Lie algebra via translation.

For $C \in \LG'(\F')$, the derivative $d m_{C}$ induces an isomorphism between the tangent space of $\LG'$ at the identity and the tangent space of $\LG'$ at $C$.  
Given $Y_1, Y_2$ in the Lie algebra, 
set $B_i := \iota(Y_i)$ so for example $\iota( dm_C(Y_2)) = B_2 C$.  As
\[
m_C(B_1) B_2=  m_C(B_1 \Ad_G(C)(B_2))
\]
we conclude that $B_1 C B_2 = \iota ( dm_C( Y_1 + \Ad_G(C) (Y_2)))$.  In contrast, $B_2 m_C(B_1) = B_2 B_1 C = m_C(B_2 B_1) = \iota( dm_C(Y_2 + Y_1))$.
\end{remark}

\begin{proof}[Proof of Proposition~\ref{prop:kvtspace}]
Trivializing $\fP_1$ and $\fP_2$, we obtain $C_1, C_2 \in \LG'(\F'[\epsilon]/(\epsilon^2))$
representing $\varphi_{\fP_1}$ and $\varphi_{\fP_2}$ and $D \in \LG'(\F'[\epsilon]/(\epsilon^2))$ 
 representing the isomorphism $\epsilon_G(\fP_1) \simeq \epsilon_G(\fP_2)$ such that
\begin{equation} \label{eq:one}
D C_1 = C_2 \varphi(D).
\end{equation}
We know that $D$ is the identity modulo $\epsilon$,
 and that $C_1$ and $C_2$ agree modulo $\epsilon$.  Let $\overline{C} \in \LG'(\F')$ be the common reduction of $C_1$ and $C_2$; we may write $\overline{C} = B_1 u ^{\mu'} B_2$ with $B_1, B_2 \in \LpG'(\F')$ and $\mu' \leq \mu$ as $\fPbar \in Y^{\leq \mu}(\F')$.
    Using the natural inclusion $\F' \into \F'[\epsilon]/(\epsilon^2)$, we may also view $\overline{C}$ as an element of $\LG'(\F'[\epsilon]/(\epsilon^2))$.
 
We wish to show that $D \in \LpG'(\F') $, which implies that $\fP_1$ and $\fP_2$ are equivalent deformations of $\fPbar$.  By appropriate choice of trivializations, without loss of generality we may assume that $B_1$ is the identity.   We rewrite \eqref{eq:one} as
\[
D C_1 B_2^{-1} = C_2 B_2^{-1} B_2 \varphi(D) B_2^{-1}.
\]
Notice that $D$ corresponds to a tangent vector at the identity, and $C_i B_2^{-1}$ corresponds to a tangent vector at $u^{\mu'}$. The derivative of the the right multiplication map $m_{u^{\mu'}}$ identifies the tangent space at the identity with the tangent space at $u^{\mu'}$, so there are $X, Y_1,$ and $Y_2$ in $\Lie \LG_{\F'} = \g_{\F'\lseries{u}}$
 such that $\imath(X)=D$ and $\imath(d m_{ u^\mu} (Y_i)) = C_i B_2^{-1}$.
Using Remark~\ref{rmk:dualmult}, they satisfy
\begin{equation}  \label{eq:lieeq}
X + Y_1 - Y_2 =  \AdG( u^{\mu'} B_2) ( \varphi(X)). 
\end{equation}

Let $n$ be the largest integer such that $X \in u^n \g'_{\F'\pseries{u}}$.  
We claim that $n \geq 0$, which  will show that $X$ lies in $\g'_{\F'\pseries{u}}$.  This in turn implies that $D \in \LpG'(\F')$ as desired.   

We begin by considering the left side of \eqref{eq:lieeq}.  Since $\fP_1$ and $\fP_2$ satisfy the mod-$p$ monodromy condition, Corollary~\ref{cor:tspacenabla} shows that that
\[
Y_i \in \ft'_{\F'\pseries{u}} \oplus \bigoplus _{\alpha \in \Phi_{G'}} \left( u^{\min(\langle \mu',\alpha \rangle, 0)} \g'_{\alpha, \F'[\![u]\!]} \right).
\]
In particular,  we see that
$$ Z := X+ Y_1 - Y_2 \in u^n \ft'_{\F'\pseries{u}} \oplus \bigoplus _{\alpha \in \Phi_{G'}} \left( u^{\min(\langle \mu',\alpha \rangle, n)} \g'_{\alpha, \F'[\![u]\!]} \right).$$

On the other hand, as $B_2 \in \LpG'(\F')$ we know $\Ad_{G'}(B_2)$ is an invertible linear transformation of $\g'_{\F'\pseries{u}}$ and hence $\Ad_{G'}(B_2) (\varphi(X))$ lies in $u^{pn} \g'_{\F'\pseries{u}}$ but not $u^{pn+1} \g'_{\F'\pseries{u}}$.
As $\Ad_{G'}(u^{\mu'})$ acts on the root space $\g'_\alpha$ via multiplication by $u^{ \langle \mu',\alpha \rangle}$ and it acts trivially on $\ft$, we see that
\[
Z=\Ad_G( u^{\mu'} B_2) (\varphi(X)) \in 
u^{pn} \ft'_{\F'\pseries{u}} \oplus \bigoplus _{\alpha \in \Phi_{G'} } \left( u^{pn + \langle \mu' , \alpha \rangle} \g'_{\alpha, \F'[\![u]\!]} \right).
\]
We also know there is either a root $\alpha$ for which the $\g'_\alpha$-component of $Z$ does not lie in $u^{pn+1 + \langle \mu',\alpha \rangle} \g'_{\alpha, \F'\pseries{u}}$ or the $\ft'$-component of $Z$ does not lie in $u^{pn+1}\ft'_{\F'\pseries{u}}$.
In the latter case,  if $n<0$ then the $\ft'$-component of $Z = X + Y_1 - Y_2$ does not lie in $u^n \ft'_{\F'\pseries{u}}$, a contradiction. 
 In the former case, this would imply that $pn + \langle \mu', \alpha \rangle \geq \min( \langle \mu',\alpha \rangle), n)$.  
If the minimum is $\langle \mu', \alpha \rangle$, then we see that $pn\geq 0$ and hence $ n \geq 0$.  If the minimum is $n$ we would have $(p-1)n \geq - \langle \mu',\alpha\rangle$.   
But since $\mu$ is Fontaine-Laffaille, $|\langle \mu', \alpha \rangle| \leq |\langle \mu, \alpha \rangle| < p-1$, so we again conclude that $n \geq 0$.
\end{proof}

\begin{cor} \label{cor:kisinvtan}
Fix a finite extension $\F' / \F$, $\mu$ in the Fontaine-Laffaille range, and $\fPbar \in Y^{\leq \mu}(\F')$.  Given lifts $\fP_1, \fP_2 \in Y^{\leq \mu}(\F'[\epsilon]/(\epsilon^2))$ of $\fPbar$ satisfying the mod-$p$ monodromy condition such that $\tT_{G, \F'[\epsilon]/(\epsilon^2)}(\fP_1) \simeq \tT_{G,\F'[\epsilon]/(\epsilon^2)}( \fP_2)$ (compatible with the identification of the reduction with $\rhobar|_{\Gamma_\infty} \simeq \tT_{G,\F'} (\fPbar)$), then $\fP_1$ and $\fP_2$ are isomorphic as deformations of $\fPbar$.
\end{cor}

\begin{proof}
This is an immediate consequence of Proposition~\ref{prop:kvtspace} as $T_{G,\F'[\epsilon]/(\epsilon^2)}$ is an equivalence of categories $\GMod_{\cO_{\cE , \F'[\epsilon]/(\epsilon^2)}}^\varphi \to \GRep(\F'[\epsilon]/(\epsilon^2))$ (Fact~\ref{fact:etaleequiv}).
\end{proof}

\begin{cor} \label{cor:kisinvspec}
Suppose $\mu$ is Fontaine-Laffaille and $G^{\der}$ is simply connected, or that $\mu$ is strongly Fontaine-Laffaille.  If $Y^{\leq \mu}_{\cP}(\F) \neq \emptyset$, then $Y^{\leq \mu}_{\cP} = \Spec(\F)$.
\end{cor}

\begin{proof}
Corollary~\ref{cor:kisinvpoint} shows that $Y^{\leq \mu}_{\cP}$ has only one point.
By Lemma~\ref{lem:kvmonodromy}, any $\F'[\epsilon]/(\epsilon^2)$-points satisfy the mod-$p$ monodromy condition.  Then Proposition~\ref{prop:kvtspace} shows that $Y^{\leq \mu}_{\cP}$ is reduced. 
\end{proof}

\subsection{Forgetting Kisin Modules} \label{ss:forgetkisin}

Fix $\rhobar : \Gamma_K \to G(\F)$ and a dominant cocharacter $\mu$ for $G'$.
Our next goal is to show the Kisin resolution is an isomorphism in our situation. 

\begin{prop} \label{prop:kisinresiso}
Let $\cP = M_{G,\F}(\rhobar)$, and suppose the Kisin variety $Y^{\leq \mu}_{\cP}$ is trivial (i.e. isomorphic to $\Spec(\F)$).  
Then $\Theta: X^{\mu}_{ \rhobar} \to \Spec R^{\mu,\square}_{\rhobar}$ is an isomorphism.  
\end{prop}

\begin{proof}
 Since $\Theta: X^{\leq \mu}_{\rhobar} \to \Spec R^{\mu,\square}_{\rhobar}$ is constructed as the limit of $\Theta_n : X^{\leq \mu}_{\rhobar, n} \to \Spec R^{\mu,\square}_{\rhobar} / \fm^n$ (where $\fm$ is the maximal ideal of $R^{\mu,\square}_{\rhobar}$), to check it is finite it suffices to check that each $\Theta_n$ is finite.  But $R^{\mu,\square}_{\rhobar} / \fm^n$ is a local Artin ring, so has a unique geometric point.  As the Kisin variety is trivial, the fiber is a single point, and so $\Theta_n$ is quasi-finite and hence finite as $\Theta$ is projective by Proposition~\ref{prop:resolution}. 
 Thus $X^{\leq \mu}_{\rhobar} = \Spec S$ for a local ring $S$ that is finite over $R^{\mu,\square}_{\rhobar}$. 

As the Kisin variety is trivial, $\Theta$ induces an injection on tangent spaces at the closed point.  Using the finiteness of $S$ and Nakayama's lemma, we see $\Theta$ is a closed immersion and hence $S$ is a quotient of $R_{\rhobar}^{\mu,\square}$.  But since $R_{\rhobar}^{\mu,\square}$ injects into $R_{\rhobar}^{\mu,\square}[\frac{1}{p}]$ (as $R_{\rhobar}^{\mu,\square}$ is $\Lambda$-flat) and $\Theta[\frac{1}{p}]$ is an isomorphism (by Proposition~\ref{prop:resolution}\eqref{resolutionIII}), it follows that $R_{\rhobar}^{\mu,\square} \to S$ is injective.  This completes the proof.
\end{proof}

\begin{cor} \label{cor:forgetkisin}
Fix a  $G$-Kisin module $\fPbar$ in $M_{G,\F}(\rhobar|_{\Gamma_\infty})$ of type $\leq \mu$.  If the Kisin variety is trivial
 then
$R_{\rhobar}^{\mu,\square} = R_{\rhobar,\fPbar} ^{\mu, \square,\pflat}$.  
\end{cor}

\begin{proof}
Proposition~\ref{prop:kisinresiso} shows that $X^{\mu}_{\rhobar}$ is represented by local ring; it is $\Lambda$-flat by Fact~\ref{fact:crystallinedeformation}.  
Using Proposition~\ref{prop:resolution}\eqref{resolutionI}, we see it is isomorphic to $\Spf R^{\mu,\square,\pflat}_{\rhobar,\fPbar}$.  
  \end{proof}
  
  \begin{remark} 
We know that the Kisin variety is trivial when $G^{\der}$ is simply connected and $\mu$ is Fontaine--Laffaille, or when $\mu$ is strongly Fontaine--Laffaille (Corollary~\ref{cor:kisinvspec}).  
We expect the Kisin variety to be trivial for most $\rhobar$ when $\mu$ is Fontaine--Laffaille, 
but we give an example below where the Kisin variety is non-trivial for a particular $\rhobar$ valued in the non-simply connected group $\PGL_2$ and a particular $\mu$ which is Fontaine-Laffaille but not strongly Fontaine-Laffaille.  
\end{remark}
  
\begin{example} \label{ex:pgl2}
Let $K = \Qp$ with $p \neq 2$, and consider the projection map $\mathrm{pr}:\GL_2 \rightarrow \mathrm{PGL}_2$.   Let $\overline{\omega}$ denote mod $p$ cyclotomic character.   Consider $\rhobar = \overline{\omega}^{\frac{p-1}{2}} \oplus 1$ and $\rhobar' = \mathrm{pr}(\rhobar)$, and let $\cP' = M_{\mathrm{PGL}_2, \F}(\rhobar')$.  We denote the cocharacter of $\GL_2$ sending $u$ to the diagonal matrix with entries $u^a$ and $u^b$ by $(a,b)$.  We 
abuse notation and use the same notation to denote the composition of this cocharacter with $\mathrm{pr}$.

 We claim that for $\mu = (\frac{p-1}{2}, 0)$, the $\mathrm{PGL}_2$ Kisin variety $Y^{\leq \mu}_{\cP'}$ has at least 2 closed points.   Note that $\mu$ is Fontaine--Laffaille but not strongly Fontaine--Laffaille, and $\PGL_2$ is not simply connected.  
 We start with the rank 2 Kisin module $\fM_1$ over $\bF$ with Frobenius given by $u^{\mu}$.   Clearly $\mathrm{pr}(\fM_1) = \fP'_1$ is a lattice in $\cP'$ with type $\leq \mu$.  We can define a second lattice $\fM_2$ in $\fM_1[1/u]$
 which scales the second basis vector by $u$; Frobenius is then given by $u^{(\frac{p-1}{2}, p-1)}$. 
 The pushout $\mathrm{pr}(\fM_2) = \fP'_2$ is a lattice in $\cP'$ different from $\fP_1'$ because $u^{(0,1)} \notin L^+ \mathrm{PGL}_2(\F)$.  Finally, the image of  $u^{(\frac{p-1}{2}, p-1)}$ in $\Gr_{\mathrm{PGL}_2}(\F)$ is the same as the image of $u^{(\frac{p-1}{2},0)}$ and so $\fP'_2$ has type $\leq \mu$.   
\end{example}

\subsection{Forgetting Galois Representations} \label{ss:forgetgalois}
We next study the forgetful map $D^{\mu,\beta,\square}_{ \rhobar,\fPbar} \to D^{\leq \mu,\beta,\square}_{\fPbar}$.  To do so, we will use the theory of $(\varphi,\Gammahat)$-modules with $G$-structure developed in \cite[\S4.2]{levin15}, extending Liu's theory of $(\varphi,\widehat{G})$-modules \cite{liu10}.  We briefly recall a concrete version here.

Let $\cO_{\Kbar}^\flat$ be the perfection of $\cO_{\Kbar} / (p)$ and $\Ainfbasic = W(\cO_{\Kbar}^\flat)$. A fixed compatible set $\{p^{1/p}, p^{1/p^2}, \ldots \}$ of $p$-power roots of $p$ defines an element $\pi \in \cO_{\Kbar}^\flat$.  Let $[\pi] \in \Ainfbasic$ be the Teichmuller lift of $\pi$.  There are embeddings $\fS \into \Ainfbasic$ and $\cO_{\cE} \into \Ainfbasic$ defined by sending $u \in \fS$ to $[\pi]$; they are compatible with Frobenius.  In this section only, let $\ft \in \Ainfbasic$ denote the period of $\fS(1)$, which satisfies $\varphi(\ft) = c_0^{-1} E(u) \ft$.

The theory of $(\varphi,\widehat{G})$-modules uses a ring $\widehat{R} \subset \Ainfbasic$ which contains $\fS$.  It is defined on page 5 of \cite{liu10}; we do not need detailed information about it.  For a $\Z_p$-algebra $A$, define $\widehat{R}_A := \widehat{R} \tensor{\Z_p} A$ and $\Ainf{A} := \Ainfbasic \tensor{\Z_p} A$.

As in \S\ref{notation:galoisgps}, let $K_\infty = K(p^{1/p}, p^{1/p^2}, \ldots )$ and $\Gamma_\infty = \Gal(\overline{K} / K_\infty)$.  Furthermore, set $K_{p^\infty} = \cup_n K(\zeta_{p^n})$ where $\zeta_{p^n}$ is a primitive $p^n$-th root of unity.  Let $K_{\infty,p^\infty}$ denote the compositum of $K_\infty$ and $K_{p^\infty}$.  It is Galois over $K$; set $\Gammahat := \Gal(K_{\infty, p^\infty} /K)$ and $\Gammahat_\infty := \Gal(K_{\infty,p^\infty}/K_\infty)$.

Fix a coefficient ring $A \in \widehat{\cC}_{\Lambda}$. 

\begin{defn}
Define $G_A(u^n)$ to be the kernel of the reduction map
\[
G(\Ainf{A}) \to G(\Ainf{A}/ (\varphi(\ft) u ^n)).
\]
\end{defn}

For our purposes, we do not need the general definition of a $(\varphi,\Gammahat)$-modules with $G$-structure \cite[Definition 4.2.6]{levin15}, and instead use the following description \cite[Proposition 4.3.10]{levin15}:

\begin{fact} \label{fact:crystallineGhat}
A crystalline $(\varphi,\Gammahat)$-module with $G$-structure and coefficients in $A$ is equivalent to a $G$-Kisin module $(\fP,\phi)$ with coefficients in $A$ 
 with a ``crystalline $\Gammahat$-structure''.  If we fix a trivialization $\beta$ of $\fP$, trivialize $\varphi^*(\fP)$ using $1 \tensor{\varphi } \beta$, and let $C_{\fP, \beta} \in G(\fS_{A}[ 1 / E(u)])$ correspond to $\varphi$, a crystalline $\Gammahat$-structure structure is a continuous map
\[
B_\bullet : \Gammahat \to G(\widehat{R}_A)
\] 
that satisfies the following conditions:
\begin{enumerate}[(i)]
\item $C_{\fP, \beta} \varphi(B_\gamma) = B_\gamma \cdot \gamma(C_{\fP, \beta})$ in $G(\Ainf{A})$ for all $\gamma \in \Gammahat$;
\item  $B_\gamma = \operatorname{Id}$ for all $\gamma \in \Gammahat_\infty$;
\item  $B_\gamma \in G_A(u^p)$ for all $\gamma \in \Gammahat$;
\item  $B_{\gamma \gamma'} = B_\gamma \cdot \gamma(B_{\gamma'})$ for all $\gamma, \gamma' \in \Gammahat$.
\end{enumerate}
\end{fact}

\begin{fact} \label{fact:crystallineequiv} 
Let $A$ be a finite $\Lambda$-algebra which is either flat or Artinian.  There is a functor $\widehat{T}_{G,A}$ from the category of crystalline $(\varphi,\Gammahat)$-modules with $G$-structure with coefficients in $A$ to $\GRep_A(\Gamma_K)$.  It is compatible with base change along finite flat maps.  For a crystalline $(\varphi,\Gammahat)$-module $\widehat{\fP}$ with $G$-structure and underlying $G$-Kisin module $\fP$, there is a natural isomorphism
\[
\widetilde{T}_{G,A}(\fP) \simeq \widehat{T}_{G,A}(\widehat{\fP}) |_{\Gamma_\infty}
\]
where $\widetilde{T}_{G,A}$ is the functor from $G$-Kisin modules to representations of $\Gamma_\infty$ in Definition~\ref{defn:tildeT}.

\end{fact}

\begin{proof}
The functor $\widehat{T}_{G,A}$ is discussed in \cite[\S4.2]{levin15}. 
\end{proof}

\begin{remark}
As the name suggests, crystalline $(\varphi,\Gammahat)$-modules with $G$-structure are related to crystalline representations.
Let $L'/L$ be a finite extension with ring of integer $\cO_{L'}$.
Then $\widehat{T}_{G,\cO_{L'}}$ gives an equivalence of categories between the category of crystalline $(\varphi,\Gammahat)$-modules with $G$-structure and coefficients in $\cO_{L'}$ and the category of crystalline representations in $\GRep_{\cO_{L'}}(\Gamma_K)$.  This is \cite[Proposition 4.3.5]{levin15}.
\end{remark}

We say that $C \in G(\fS_A[1/E(u)])$ has height in $[a,b]$ with respect to the adjoint representation if 
\[
E(u)^{a} \Lie(G) \tensor{\Lambda} \fS_A \subset \Ad_G(C)  \left( \Lie(G)\tensor{\Lambda} \fS_A \right) \subset E(u)^{b} \Lie(G) \tensor{\Lambda} \fS_A.
\]

\begin{lem} \label{lem:amplification}
Let $A$ be a $p$-adically complete $\Lambda$-algebra such that $p A =0$, and let $C \in G(\fS_A [ 1 / E(u)]) $ have height in $[-h,h]$ with respect to the adjoint representation.  If $h < p-1$, then for any $Y \in G_A(u^n)$ with $n \geq p$, we have 
$$ \varphi(C) \varphi(Y) \varphi(C)^{-1} \in G_A(u^{n+1}).$$
\end{lem}

\begin{proof}
Let $\cO_G$ denote the coordinate ring of $G$, and $I_e$ the ideal defining the identity.  We have that $\cO_G/I_e = \Lambda$ and $I_e / I_e^2 \simeq (\Lie G)^\vee$. For a $\Lambda$-algebra $B$, we know that $G(B)$ can be identified with maps of rings from $\cO_G$ to $B$; the identity of $G(B)$ is the natural map $\cO_G \to \cO_G/ I_e = \Lambda \to B$.
 Thus we can identify $G_A(u^n)$ with
\[
\{ f \in \Hom_\Lambda(\cO_G,\Ainf{A}) \,| \, f(I_e) \subset (\varphi(\ft) u^n ) \}.
\]
Notice that $\varphi(Y)$ is the composition of $Y$ with the endomorphism $\varphi$ of $\Ainf{A}$, so in particular we conclude that $\varphi(Y)(I_e) \subset (\varphi(\varphi(\ft) u^n)) = (\varphi(\varphi(\ft)) u^{pn})$.

Now conjugation by $C$ induces an automorphism of $G_{\fS_A[1/E(u)]}$, and hence an automorphism $$\Ad_{\cO_G}(C)^* : \cO_G \tensor{\Lambda} \fS_A[1/E(u)] \to \cO_G \tensor{\Lambda} \fS_A[1/E(u)].$$
Conjugation by $\varphi(C)$ likewise induces an automorphism, and is given by $(1 \otimes \varphi) \circ \Ad_{\cO_G}(C)^*$.  For $x \in I_e \otimes 1 $, we claim that
\begin{equation} \label{eq:conjden}
(\Ad_{\cO_G}(C)^*) (x) \in \sum_{j \geq 1} I_e^j \tensor{\Lambda} E(u)^{-hj}\fS_A.
\end{equation}
  By successive approximation, we can just study the induced automorphisms of the graded pieces $I_e^j / I_e^{j+1} \tensor{\Lambda} \fS_A[1/E(u)] \simeq \Sym^j(\Lie(G)^\vee) \tensor{\Lambda} \fS_A[1/E(u)]$.  Using the height condition,  the image of $\Sym^j(\Lie(G)^\vee) \tensor{\Lambda} \fS_A[1/E(u)]$ lies in $E(u) ^{-h j} \Sym^j(\Lie(G)^\vee) \tensor{\Lambda} \fS_A[1/E(u)]$ as desired.

Now viewing $\varphi(C) \varphi(Y) \varphi(C)^{-1}$ as a homomorphism from $\cO_G \tensor{\Lambda} \fS_A[1/E(u)]  $ to $\Ainf{A} \tensor{\Lambda} \fS_A[1/E(u)] $, observe that for $x \in I_e \otimes 1$
\[
(\varphi(C) \varphi(Y) \varphi(C)^{-1}) (x) = (\varphi(Y) \otimes 1) \circ ( 1 \otimes \varphi) \circ (\Ad_{\cO_G}(C)^*) (x).
\]
Using \eqref{eq:conjden}, as $\varphi(Y)(I_e^j) \subset \varphi(\varphi(\ft) )^j u ^{pn j}$ we see that
\[
(\varphi(C) \varphi(Y) \varphi(C)^{-1}) (x)  \in \sum_{j \geq 1} \varphi(\varphi(\ft) )^j u ^{pn j} \varphi(E(u))^{-h j}  \Ainf{A}.
\]
So to check that $\varphi(C) \varphi(Y) \varphi(C)^{-1} \in G_A(u^{n+1})$, it suffices to check that $u^{n+1} \varphi(\ft)$ divides $ \alpha_j := \varphi(\varphi(\ft) )^j u ^{pn j} \varphi(E(u))^{-h j}$ in $\Ainf{A}$ for any $j \geq 1$.  As $p A =0$ we have that $E(u) = u$, so using that $\ft E(u)$ divides $\varphi(\ft)$ in $\Ainfbasic$ (since $\varphi(\ft) = c_0^{-1} E(u) \ft$) we see that $\alpha_j$ is a multiple of  $\varphi(\ft) u^{pj} u^{pn j} u^{-p h j}$.  
But when $h <p-1$ and $n \geq p$, $(p-1)n \geq p (p-1) > p (h-1)$ and hence $(p + pn - p h)j > n$ as desired.
\end{proof}

\begin{lem} \label{lem:uniquecrystalline}
Suppose $A$ is a $p$-adically complete $\Lambda$-algebra such that $p A=0$.  For $\fP \in D_{\overline{\fP}}^{\leq \mu,\beta} (A)$ with $\mu$ in the Fontaine-Laffaille range, there is at most one crystalline $\widehat{\Gamma}$-structure on $\fP$.  
\end{lem}

\begin{proof}
Suppose we have two crystalline $\widehat{\Gamma}$-structures, with the action of $\gamma \in \widehat{\Gamma}$ given by $B_\gamma$ and $B'_\gamma$ in $G(\Ainf{A})$.  Using Property (iii) of Fact \ref{fact:crystallineGhat}, we have that $B_\gamma (B'_\gamma)^{-1} \in G_A(u^p)$.  Furthermore, if the Frobenius on $\fP$ is given by $C$ then
\[
B_\gamma (B_\gamma)^{-1} = \varphi(C) \varphi( B_\gamma (B'_\gamma)^{-1} ) \varphi(C)^{-1}.
\]
An inductive argument using Lemma~\ref{lem:amplification} shows that $B_\gamma (B'_\gamma)^{-1} \in G_A(u^n)$ for all $n \geq p$.  Thus $B_\gamma = B_\gamma'$ as desired.
\end{proof}

Note the forgetful map $\Spf R_{\rhobar,\fPbar}^{\mu,\beta,\square} \to \Spf R_{\fPbar}^{\leq \mu,\beta,\square}$ factors through the flat closure by Lemma~\ref{lem:flatmoduli}.

\begin{prop} \label{prop:tspaceinjection}
For $\mu$ in the  Fontaine-Laffaille range, the natural map $\Spf R_{\rhobar,\fPbar}^{\mu,\beta,\square,\pflat} \to \Spf R_{\fPbar}^{\leq \mu,\beta,\square}$ is injective on tangent spaces. 
\end{prop}

\begin{proof} 
Let $A$ be a finite $\FF$-algebra, and set $R:= R_{\rhobar,\fPbar}^{\mu,\beta,\square,\pflat}$.  We first claim that for every $A$-valued point $f_A :  R \to A$, there exists a finite flat $\Lambda$-algebra $B$ and a $B$-valued point $f_B : R \to B$ that lifts $f_A$.  We do so using an idea from \cite[Lemma 3.2.2]{bartlettpotential}.  Notice  that $R$ is a complete local $\Lambda$-algebra that is reduced and $\Lambda$-flat (Corollary~\ref{cor:forgetkisin} and Fact~\ref{fact:crystallinedeformation}).  Furthermore, $R$ is Nagata as it is a complete local Noetherian ring.  
Now $f_A : R \to A$ factors through $R/ \fm_R^j$ for some integer $j \geq 1$ by continuity, so we easily adapt \cite[Lemma 4.1.2]{bartlettirreducible} to find a  finite flat $\Lambda$-algebra $B$ and a $B$-valued point $f_B : R \to B$ that lifts $f_A$.  

We apply the previous paragraph with $A = \FF[\epsilon]/(\epsilon^2)$.  A tangent vector to $\Spf R_{\rhobar,\fPbar}^{\mu,\beta,\square,\pflat}$ at the closed point corresponds to a pair $(\fP_A, \rho_A)$ and is the reduction of a $B$-valued point for some finite flat $\Lambda$-algebra $B$.  This point gives a trivialized Kisin module $\fP_B$ and a Galois representation $\rho : \Gamma_K \to G(B)$ such that $\tT_{G,B}(\fP_B) = \rho|_{\Gamma_\infty}$ and $\rho \otimes L$ is crystalline.  The same argument used in the proof of \cite[Theorem 4.2.7]{levin15} which relies on \cite{liu10} shows that  $\fP_B$ admits a crystalline $(\varphi,\widehat{\Gamma})$-structure $\widehat{\fP}_B$ such that $\widehat{T}_{G, B}(\widehat{\fP}_B) = \rho$.  (In fact, as $B$ is finite flat it is the unique such structure.)   Therefore its reduction, the $G$-Kisin module $\fP_A$, also admits a crystalline $(\varphi,\widehat{\Gamma})$-structure such that $\widehat{T}_{G, A}(\fP_A) = \rho_A$.

Finally, given two tangent vectors to $\Spf R_{\rhobar,\fPbar}^{\mu,\beta,\square,\pflat}$ with the same underlying $G$-Kisin module, we know each admits a crystalline $(\varphi,\widehat{\Gamma})$-structure.  By Lemma~\ref{lem:uniquecrystalline} these structures are the same, and we know the crystalline structure determines the Galois representation.
\end{proof}

\begin{cor} \label{cor:closedimmersion}
If $\mu$ is Fontaine-Laffaille, the natural map $\Spf R_{\rhobar,\fPbar}^{\mu,\beta,\square,\pflat} \to \Spf R_{\fPbar}^{\leq \mu,\beta,\square}$ is a closed immersion.
\end{cor}

\begin{proof}
This follows from Proposition~\ref{prop:tspaceinjection} and Nakayama's lemma.
\end{proof}

\subsection{Relationships between Deformation Rings} \label{ss:relationships}

\begin{defn} \label{defn:defringmonodromy}
We let $ R_{\fPbar}^{\leq \mu, \beta,\square,\nabla} $ be the $\Lambda$-flat and reduced quotient of $R_{\fPbar}^{\leq \mu, \beta,\square}$ such that $\Spec R_{\fPbar}^{\leq \mu, \beta,\square,\nabla} [1/p]$ is the vanishing locus of the monodromy condition on $\Spec R_{\fPbar}^{\leq \mu, \beta,\square}[1/p]$.  Define $ R_{\fPbar}^{\leq \mu, \beta,\nabla} $ similarly. 
\end{defn}

There are closed immersions 
\[
\Spf R_{\fPbar}^{\leq \mu, \beta,\square,\nabla} \into \Spf R_{\fPbar}^{\leq \mu, \beta,\square} \quad \text{ and  } \quad \Spf  R_{\fPbar}^{\leq \mu, \beta,\nabla} \into  R_{\fPbar}^{\leq \mu, \beta}.
\]
As $R_\rhobar^{\mu,\square}$ is $\Lambda$-flat (recall Fact~\ref{fact:crystallinedeformation}), the forgetful map $\Spf R_{\rhobar,\fPbar}^{\mu,\square} \to \Spf R_\rhobar^{\mu,\square}$ factors through $R^{\mu,\square,\pflat}_{\rhobar,\fPbar}$.   As $\Spf R_{\fPbar}^{\leq \mu, \beta, \square}$ is $\Lambda$-flat by Lemma~\ref{lem:flatmoduli}, we likewise obtain a map $\Spf R^{\mu,\beta,\square,\pflat}_{\rhobar,\fPbar} \to \Spf R_{\fPbar}^{\leq \mu, \beta, \square}$.

\begin{thm} \label{thm:diagram}
Assume that $p \nmid \# \pi_1(G^{\ad})$ and that $\mu$ is Fontaine-Laffaille.  We continue to fix a $G$-Kisin module $\fPbar$ over $\F$ and a continuous Galois representation $\rhobar: \Gamma_K \to G(\F)$ together with an isomorphism $\tT_{G,\F}(\fPbar) \simeq \rhobar|_{\Gamma_\infty}$.
If the Kisin variety $Y^{\leq \mu}_{M_{G,\bF}(\rhobar)}$ is trivial, 
there is a commutative diagram of formal schemes
\begin{equation} \label{eq:bigdiagram}
\begin{tikzcd}
 & & \Spf R_{\fPbar}^{\leq \mu, \beta,\square,\nabla} \ar[d,hook] \ar[r,"f.s."] & \Spf R_{\fPbar}^{\leq \mu, \beta,\nabla} \ar[d,hook] \\
 & \Spf R^{\mu, \beta,\square,\pflat}_{\rhobar,\fPbar} \ar[d, "f.s."] \ar[r,  hook] \ar[ur ,"\imath",  hook] &  \Spf R_{\fPbar}^{\leq \mu, \beta,\square} \ar[r, "f.s."] &  \Spf R_{\fPbar}^{\leq \mu, \beta}\\
\Spf R_{\rhobar}^{\mu,\square} & \Spf R^{\mu,\square,\pflat}_{\rhobar,\fPbar} \ar[l,"\sim", swap]  &                        
\end{tikzcd}
\end{equation}
with the indicated arrows isomorphisms and closed immersions, and with the arrows labeled $f.s.$ formally smooth.  The square is cartesian.
\end{thm}

\begin{remark}
The hypothesis that $p \nmid \# \pi_1(G^{\ad})$ is equivalent to $Z^{\der}$ being \'{e}tale over $\Lambda$ and $p \nmid \pi_1(G^{\der})$, which are necessary to apply many of our results.  The restrictions ultimately trace back to Theorem~\ref{thm:affineschubert} and Proposition~\ref{prop:ts}.
\end{remark}

\begin{proof}
The maps in \eqref{eq:bigdiagram} labeled $f.s.$ come from forgetting a trivialization modulo $E(u)^N$ of a $G$-Kisin module, or forgetting a trivialization of a Galois representation.  These are formally smooth as the set of trivializations are a $G$-torsor and $G$ is smooth.

The horizontal isomorphism comes from forgetting the $G$-Kisin module; Corollary \ref{cor:forgetkisin} shows it is an isomorphism.

Corollary~\ref{cor:closedimmersion} shows that the forgetful map $\Spf R_{\rhobar,\fPbar}^{\mu,\beta,\square,\pflat} \to \Spf R_{\fPbar}^{\leq \mu,\beta,\square}$ is a closed immersion.  

We next claim that forgetful map $\Spf R_{\rhobar,\fPbar}^{\mu,\beta,\square} \to \Spf R_{\fPbar}^{\leq \mu,\beta,\square}$ factors through $\Spf R_{\fPbar}^{\leq \mu, \beta,\square,\nabla}$ and that $\imath : \Spf R_{\rhobar,\fPbar}^{\mu,\beta,\square} \to \Spf R_{\fPbar}^{\leq \mu, \beta,\square,\nabla}$ is a closed immersion.  It suffices to check this on $\overline{L}$-points as $R_{\rhobar,\fPbar}^{\mu,\beta,\square,\pflat}$ and $ R_{\fPbar}^{\leq \mu,\beta,\square}$ are flat $\Lambda$-algebras (by construction and by Lemma~\ref{lem:flatmoduli})
and $R_{\rhobar,\fPbar}^{\mu,\beta,\square,\pflat}[1/p]$ is reduced (which follows from Fact~\ref{fact:crystallinedeformation}). 
Let $A$ be the a finite flat $\Lambda$-algebra. 
  An $A$-point of $R_{\rhobar,\fPbar}^{\mu,\beta,\square}$ is a $G$-Kisin module $\fP$ with coefficients in $A$ together with a Galois representation $\rho : \Gamma_K \to G(A)$ extending $\widetilde{T}_{G,A}(\fP)$ (plus trivialization).   After inverting $p$, we know that $\rho$ is crystalline with $p$-adic Hodge type $\mu$ (see the discussion after Fact~\ref{fact:crystallinedeformation}).  
Corollary~\ref{cor:Gmonodromy} implies $\fP[1/p]$ satisfies the monodromy condition.  This gives the  factorization.
As $R_{\fPbar}^{\leq \mu,\beta,\square}$ surjects onto $R_{\rhobar,\fPbar}^{\mu,\beta,\square,\pflat}  $, we immediately see that $R_{\fPbar}^{\leq \mu, \beta,\square,\nabla}$ surjects onto $R_{\rhobar,\fPbar}^{\mu,\beta,\square,\pflat}$ and hence that $\imath$ is a closed immersion.

The remaining maps were discussed before the statement of the theorem.  The square is cartesian by construction.
\end{proof}

Using this, we can prove a technical version of our main theorem.  Recall that $P_{\mu,\F}$ (resp. $P_{\mu',\bf}$) are the parabolics over $\bF$ associated to the cocharacter $\mu$ (resp. $\mu'$).

\begin{thm} \label{thm:maintechnical}
Suppose that $p \nmid \# \pi_1(G^{\ad})$.
Fix a Galois representation $\rhobar : \Gamma_K \to G(\bF)$ with shape $\mu'$ and
a Fontaine-Laffaille type $\mu$ for $G$ with $\mu$ and $\mu'$ dominant and $\mu' \leq \mu$.  Suppose that
\begin{enumerate}[(i)]
\item \label{thm1} $\dim P_{\mu,\bF}  \backslash G_{\bF}  \geq \dim P_{\mu', \bF} \backslash G_{\bF}$,

\item \label{thm2} the Kisin variety $Y^{\leq \mu}_{M_{G,\bF}(\rhobar)}$ is trivial,

\item \label{thm3} and $\Spf R_{\rhobar}^{\mu,\square}$ is non-empty.
\end{enumerate}
Then $\Spf R_{\rhobar}^{\mu,\square}$ is formally smooth.
\end{thm}

\begin{proof}
We will show that $\Spf R_{\fPbar}^{\leq \mu, \beta,\square,\nabla}$ is either empty or formally smooth of the same dimension as $\Spf R^{\mu, \beta,\square,\pflat}_{\rhobar,\fPbar}$.  Since $\imath$ is a closed immersion, it follows that $\Spf R^{\mu, \beta,\square,\pflat}_{\rhobar,\fPbar}$ and hence $\Spf R^{\mu,\square}_{\rhobar}$ is formally smooth. 

Let $\overline{C} \in \LG(\bF)$ correspond to $\varphi_{\fPbar}$; it has shape $\mu'$.
Imposing the monodromy condition on the map from Proposition~\ref{prop:straightening}, we obtain a Cartesian diagram
\[
\begin{tikzcd}
  \Spf (R_{\fPbar}^{\leq \mu, \beta})_{\bF}^ {\nabla} \ar[r, "f.s." ] \ar[d,hook] &   \Spf \cO^{\wedge}_{\Gr_{G'_\bF}^{\leq \mu, \nabla} , \overline{C}} \ar[d,hook] \\
 \Spf (R_{\fPbar}^{\leq \mu, \beta})_\bF \ar[r, "f.s."] & \Spf \cO^{\wedge}_{\Gr^{\leq \mu}_{G'_{\F}}, \overline{C}}.
\end{tikzcd}
\]
Note that $ \Spf (R_{\fPbar}^{\leq \mu, \beta, \nabla})_\bF \subset \Spf (R_{\fPbar}^{\leq \mu, \beta})_{\bF}^ {\nabla} $ by Theorem~\ref{thm:monoapprox}.
As $\Spf R_{\fPbar}^{\leq \mu, \beta, \nabla}$ involves imposing the monodromy condition and then taking the flat closure, it is conceivable that it is empty.
The relative dimension of the bottom map is $\dim G'_{(N)}$, where $N$ is the fixed integer from \S\ref{ss:deformationproblems}
we have been using to define the deformation rings by trivializing $G$-Kisin modules modulo $E(u)^N$.
The upper right formal scheme is formally smooth of dimension $\dim P_{\mu',\bF} \backslash G_\bF$ by Theorem~\ref{thm:monodromyschubert}.
Thus  $\Spf (R_{\fPbar}^{\leq \mu, \beta, \nabla})_\bF$ is either empty or formally smooth of dimension at most $\dim P_{\mu',\bF} \backslash G'_\bF + \dim G'_{(N)}$.  Since by definition $R_{\fPbar}^{\leq \mu, \beta, \nabla}$ is $\Lambda$-flat, if $\Spf R_{\fPbar}^{\leq \mu, \beta, \nabla}$ is non-empty then it is formally smooth of relative dimension at most $\dim P_{\mu',\bF} \backslash G'_\bF + \dim G'_{(N)}$.  
We conclude that $\Spf R_{\fPbar}^{\leq \mu, \beta,\square, \nabla}$ is either empty or formally smooth of relative dimension at most
\begin{equation} \label{eq:dim1}
\dim P_{\mu',\bF} \backslash G'_\bF + \dim G'_{(N)} + \dim G_{\bF}
\end{equation}
as the set of trivializations on the $\Gamma_\infty$-representation is a $G_\bF$-torsor.

On the other hand, we assumed that $\Spf R_{\rhobar}^{\mu,\square}$ is non-empty.  
By Fact~\ref{fact:crystallinedeformation}, we know it has relative dimension $\dim G_{\bF} + \dim P_{\mu, \bF} \backslash G'_{\bF}$.  Hence $\Spf R^{\mu, \beta,\square,\pflat}_{\rhobar,\fPbar}$ has relative dimension
\begin{equation} \label{eq:dim2}
\dim G_{\bF} + \dim P_{\mu, \bF} \backslash G'_{\bF} + \dim G'_{(N)}.
\end{equation}
As $\dim P_{\mu,\bF}  \backslash G_{\bF}  \geq \dim P_{\mu', \bF} \backslash G_{\bF}$, the existence of the closed immersion $\imath$ completes the proof.
\end{proof}

\begin{remark} \label{rmk:cocharacterhyp}
\begin{enumerate}
\item  When $\mu$ is a regular cocharacter, condition (i) is automatic.   
\item We expect that $R_{\rhobar}^{\mu,\square}$ is in fact zero when  $\mu' \neq \mu$  so that condition (i) in Theorem \ref{thm:maintechnical} should not be necessary. 
\end{enumerate}
\end{remark}

Theorem~\ref{thm:mainintro} is a direct consequence.

\begin{proof}[Proof of Theorem~\ref{thm:mainintro}]
Take $\mu' = \mu$ to guarantee (\ref{thm1}), and use Corollary~\ref{cor:kisinvspec} to guarantee (\ref{thm2}).
\end{proof}

Finally, note that $\Spf R_{\rhobar}^{\mu,\square}$ is non-empty if and only if there exists a 
crystalline lift of $\rhobar$ with $p$-adic Hodge type $\mu$.  We now record a few partial results about the non-existence of crystalline lifts which follow from our methods.

\begin{cor} \label{cor:nolift} 
With the setup of Theorem~\ref{thm:maintechnical}, assume that (\ref{thm2}) holds but that $\dim P_{\mu,\bF} \backslash G_{\bF} < \dim P_{\mu',\bF} \backslash G_{\bF}$.  Then there does not exist a crystalline lift of $\rhobar$ with $p$-adic Hodge type $\mu$.
\end{cor}

\begin{proof}
By hypothesis \eqref{eq:dim1} is larger than \eqref{eq:dim2}.  
Then the proof of Theorem~\ref{thm:maintechnical} shows that $\Spf R^{\mu, \beta,\square,\pflat}_{\rhobar,\fPbar}$ is either empty or has relative dimension \emph{larger} than the relative dimension of $\Spf R_{\fPbar}^{\leq \mu, \beta,\square, \nabla}$.  As there is a closed immersion $\imath : \Spf R^{\mu, \beta,\square,\pflat}_{\rhobar,\fPbar} \to \Spf R_{\fPbar}^{\leq \mu, \beta,\square, \nabla}$, we conclude that $\Spf R^{\mu, \beta,\square,\pflat}_{\rhobar,\fPbar}$ is empty.
\end{proof}

\begin{remark}
The techniques of this paper can also be adapted to show that the existence of a crystalline lift of type $\mu$ satisfying the hypotheses (i) and (ii) from Theorem~\ref{thm:mainintro}
implies non-existence of a crystalline lift of weight $\mu'$ for $\mu' < \mu$.
This is a motivation for the expectation in Remark~\ref{rmk:cocharacterhyp}(2).
\end{remark}


\providecommand{\bysame}{\leavevmode\hbox to3em{\hrulefill}\thinspace}
\providecommand{\MR}{\relax\ifhmode\unskip\space\fi MR }
\providecommand{\MRhref}[2]{%
  \href{http://www.ams.org/mathscinet-getitem?mr=#1}{#2}
}
\providecommand{\href}[2]{#2}

\end{document}